\documentclass[11pt,reqno]{article}
\usepackage{amsmath,amssymb,amsthm,dsfont,mathrsfs}
\usepackage[margin=1in]{geometry}
\usepackage[colorlinks=true, linkcolor=blue, citecolor=magenta]{hyperref}
\usepackage{enumitem}
\usepackage{xcolor}
\usepackage{authblk}
\usepackage{comment}
\usepackage{bm}

\numberwithin{equation}{section}

\allowdisplaybreaks

\newcommand{\pa}{\partial}
\newcommand{\eps}{\varepsilon}

\newcommand{\bra}[1]{\left(#1\right)}

\newcommand{\R}{\mathbb{R}}

\newcommand{\intR}{\int_{\R^N}}
\newcommand{\intQT}{\int_{\tau}^T\intR}
\newcommand{\LO}[1]{L^{#1}(\R^N)}
\newcommand{\LQ}[1]{L^{#1}(Q_{\tau,T})}
\newcommand{\LQT}[1]{L^{#1}(Q_T)}
\newcommand{\Lxt}[2]{L^{#1}(\tau,T;L^{#2}(\R^N))}
\newcommand{\abs}[1]{\left|#1\right|}
\newcommand{\norm}[1]{\left\|#1\right\|}
\newcommand{\wt}{\widetilde}
\newcommand{\CF}{{\mathcal F}}
\newcommand{\CS}{ {\mathcal S}}
\newcommand{\N}{{\mathbb N}}

\newcommand{\ub}{ {\bm{u}}}

\newcommand{\Bnorm}[1]{\Big\|#1\Big\|}

\newtheorem{theorem}{Theorem}[section]
\newtheorem{definition}[theorem]{Definition}
\newtheorem{lemma}[theorem]{Lemma}
\newtheorem{proposition}[theorem]{Proposition}
\newtheorem{remark}[theorem]{Remark}
\newtheorem{corollary}[theorem]{Corollary}

\author[1]{Phuoc-Tai Nguyen}
\author[2]{Bao Quoc Tang\footnote{Corresponding author.}}
\affil[1]{\small Department of Mathematics and Statistics, Masaryk University, Brno, Czech Republic. \break
	\href{mailto:ptnguyen@math.muni.cz}{ptnguyen@math.muni.cz}}
\affil[2]{\small Department of Mathematics and Scientific Computing, University of Graz, Graz, Austria\break  \href{mailto:quoc.tang@uni-graz.at}{quoc.tang@uni-graz.at}}

\title{Well-Posedness for Fractional Reaction-Diffusion Systems \\ with Mass Dissipation in $\mathbb R^N$}

\date{}
\begin{document}
	\maketitle
	\begin{abstract} \noindent
		The global existence of bounded solutions to reaction-diffusion systems with fractional diffusion in the whole space $\mathbb R^N$ is investigated. The systems are assumed to preserve the non-negativity of initial data and to dissipate total mass. We first show that if the nonlinearities are at most quadratic then there exists a unique global bounded solution regardless of the fractional order. This result is achieved by combining a regularizing effect of the fractional diffusion operator and the H\"older continuity of a non-local inhomogeneous parabolic equation. When the nonlinearities might be super-quadratic, but satisfy some intermediate sum conditions, we prove the global existence of bounded solutions by adapting the well-known duality methods to the case of fractional diffusion. In this case, the order of the intermediate sum conditions depends on the fractional order. These results extend the existing theory for mass dissipated local reaction-diffusion systems to the case of fractional diffusion and unbounded domains.
	\end{abstract}
	
	\noindent{\scriptsize
		Key words:  Reaction-diffusion systems; Fractional diffusion; Duality methods; Quadratic growth; Intermediate sum conditions}
	\\{\scriptsize{MSC. 35A01, 35B65, 35K57, 35Q92, 35R11}}

	\tableofcontents
	
	\section{Introduction}\label{sec:intro}
	Reaction-diffusion systems that preserve non-negativity of initial data and dissipate (or conserve) the total mass have been investigated extensively in the last decades due to their wide applications as well as interesting mathematical structures. The recent development of novel techniques, including {the} duality method or entropy method, has enabled one to obtain the global well-posedness for a large class of such systems. Yet most of the existing works have dealt with the case of local diffusion, i.e. the diffusion represented by the classical Laplacian, and {have been considered in} bounded domains. In this paper, we extend the existing theory to reaction-diffusion systems with fractional diffusion in the entire space $\mathbb R^N$. Our work provides techniques to study systems with quadratic nonlinearities, and develop a duality method for fractional diffusion problem in $\mathbb R^N$, which can be of independent interest. 
	\subsection{Problem setting and state of the art}
	We study in this work the following reaction-diffusion system with non-local diffusion for the concentrations (or densities) $\ub = (u_1, \ldots, u_m): \R^N\times [0,\infty) \to \R^m$
	\begin{equation}\label{sys}
		\begin{cases}
			\pa_t u_i + d_i(-\Delta)^{\alpha}u_i = f_i(x,t,\ub), &x\in\R^N, t>0, i=1,\ldots, m,\\
			u_i(x,0) =u_{i,0}(x), &x\in\R^N, i=1,\ldots, m.
		\end{cases}
	\end{equation}
	Here $d_i>0$ are the diffusion coefficients; $u_{i,0}$ are given non-negative initial data; the operator $(-\Delta)^{\alpha}$, $\alpha \in (0,1)$, is the fractional Laplace operator defined as
	\begin{equation}\label{frac_lap}
		(-\Delta)^{\alpha}u(x):= C_{N,\alpha} \mathrm{P.V.}\int_{\mathbb R^N}\frac{u(x) - u(y)}{|x-y|^{N+2\alpha}}dy,\qquad  x\in \mathbb R^N,
	\end{equation}
	where $\mathrm{P.V.}$ stands for the principle value and $C_{N,s}$ is a positive constant depending only on $N,\alpha$. The nonlinearities $f_i: \R^N\times \R_+ \times \R^m \to \R$, with $\R_+:= [0,\infty)$, $i=1,\ldots,m$, satisfy
	\begin{enumerate}[ref=\textbf{F}, label=(\textbf{F}), align=left]
		\item\label{F} (continuity) for any $i=\{1,\ldots, m\}$, $f_i(\cdot,\cdot,\ub)$ is $C^{\beta}$ in $x$ ($\beta\in (0,1)$)  and continuous in $t$, for fixed $\ub$, and $f_i$ is locally Lipschitz continuous in $\ub$ uniformly in $(x,t)\in \R^N\times \R_+$,
	\end{enumerate}
	\begin{enumerate}[ref=\textbf{P}, label=(\textbf{P}), align=left]
		\item\label{P} (quasi-positivity) for any $i\in \{1,\ldots, m\}$,
		\begin{equation*}
			f_i(x,t,\ub)\ge 0 \quad \text{ for all } \quad x\in\R^N, t\ge 0, \ub\in \R_+^m\; \text{ with }\; u_i = 0,
		\end{equation*}
	\end{enumerate}
	\begin{enumerate}[ref=\textbf{M}, label=(\textbf{M}), align=left]
		\item\label{M} (mass dissipation) for any $i\in \{1,\ldots, m\}$,
		\begin{equation*}	
			\sum_{i=1}^mf_i(x,t,\ub)\le 0, \quad \forall x\in \R^N, t\ge 0, \ub\in \mathbb R_+^m.
		\end{equation*}
	\end{enumerate}
	
	The moderate continuous assumption \eqref{F} is to ensure the existence of a local solution. In case where $f_i$ depends explicitly only on $\ub$, the regularity assumptions on $x$ and $t$ are automatically satisfied (see Remark \ref{remark1}). The quasi-positivity assumption \eqref{P} guarantees the non-negativity of solutions, as long as the initial data are non-negative. The simple physical interpretation of this is that if a concentration is zero then it cannot not be consumed/used in a reaction/interaction. The third condition \eqref{M} means that the total mass is non-increasing in time, which can be seen easily by summing the equations in \eqref{sys} and integrating over $\R^N$ (provided a suitable solution which decays as $|x|\to\infty$). These assumptions are natural and appear in many models ranging from natural sciences including physics, chemistry, biology to social and life sciences. Therefore the study of reaction-diffusion systems under these assumptions has attracted a lot of attention, especially in the last decades. When there is no diffusion, i.e. \eqref{sys} reduces to a system of differential equations $\dot u_i = f_i(\ub)$, it is straightforward that, under \eqref{F}, \eqref{P} and \eqref{M}, there exists a global solution which is bounded uniformly in time. This becomes more challenging when the diffusion is taken into account, partially because the maximum and/or comparison principle no longer holds for systems in general. In the case { where} the diffusion operator is local, i.e. it is represented by the classical Laplacian, global existence for systems of type \eqref{sys} has been addressed in the early works \cite{alikakos1979lp,masuda1983global,haraux1988result} for $m = 2$ and later extended in \cite{hollis1987global,morgan1989global,Mor90} to arbitrary $m$ under additional assumptions on the nonlinearities, besides \eqref{P} and \eqref{M}. Interestingly, it was pointed out in \cite{PS00} that in general conditions \eqref{P} and \eqref{M} are not sufficient to prevent solutions from blowing up, in sup-norm, at finite time. When \eqref{sys} satisfies an additional entropy condition, i.e. $\sum_{i=1}^m f_i(\ub)\ln u_i \le 0$, \cite{GV10,CV09} showed that global bounded solutions exist for quadratic nonlinearities if $N\le 2$ and for strictly sub-quadratic nonlinearities for all $N\ge 3$. The case of quadratic nonlinearities in all dimensions has been recently solved in the recent works \cite{Sou18,CGV19,fellner2020globalA}, {the last of which does not impose} the entropy condition. It should be also remarked that the case of mass conservation in the {entire space} was treated in \cite{kanel2000global}. This quadratic growth turns out to be sharp, i.e. for any $\eps>0$ there are systems with nonlinearities growing as $|\ub|^{2+\eps}$ satisfying \eqref{P} and \eqref{M} to which the solution becomes unbounded at a finite time. Therefore, when the nonlinearities are super-quadratic, additional assumptions such as intermediate sum conditions should be imposed, see e.g. \cite{morgan2020boundedness,fitzgibbon2021reaction}. While most of these works have dealt with the case of bounded domains, some have also considered the case of the entire space $\R^N$ including \cite{kanel1984cauchy,Kan90,herrero1998global,GV10,CV09,CGV19,la2024uniform}. In many situations, the diffusion might be of non-local nature, for instance random walks with long jumps, to which the associated operator is the fraction Laplacian $(-\Delta)^{\alpha}$, for some $\alpha \in (0,1)$, defined in \eqref{frac_lap}, instead of the classical Laplacian $-\Delta$. Due to its intriguing mathematical properties, the study of PDE models with fractional, or in general non-local, diffusion has flourished with numerous fundamental works, see e.g. \cite{FLS2016,caffarelli2010drift,caffarelli2011regularity} and many others. The effect of the non-local diffusion in reaction-diffusion systems has also been investigated in e.g. \cite{berestycki2015effect,han2023propagation,zhang2024bistable}. These works focus on propagation phenomenon and the systems therein satisfy comparison principles, which makes global existence of bounded solutions easy to obtain. The well-posedness for general fractional reaction-diffusion systems satisfying \eqref{F}, \eqref{P} and \eqref{M} has been under-explored, except for the recent works \cite{alsaedi2018triangular, daoud2024class, laurenccot2023nonlocal}. More precisely,  \cite{alsaedi2018triangular} studied a $2\times 2$ system with triangular fractional diffusion, \cite{daoud2024class} considered the case of bounded domains and some linear intermediate sum condition (i.e. \eqref{ISC} below with $\rho = 1$), and \cite{laurenccot2023nonlocal} dealt with a special $2\times 2$ Gray-Scott system and investigated the diffusive limit $\alpha \to 1$.
	
	\medskip
	In this paper, we develop an adequate theory for reaction-diffusion systems with fractional diffusion \eqref{sys} satisfying the natural assumptions \eqref{F}, \eqref{P} and \eqref{M}. More precisely, we show that systems with quadratic nonlinearities have global bounded solutions regardless of the fractional order $\alpha \in (0,1)$. When the nonlinearities might be super-quadratic, we impose some nonlinear intermediate sum conditions and obtain global bounded solutions by developing {a duality method} for fractional diffusion in the entire space. Note that in both cases, we also prove that the solution is bounded uniformly in time, which would be helpful in studying the large time dynamics of the system.

	\subsection{Main results and key ideas}
	We start with the notion of solutions that we consider in this work. In the sequel of this paper, we will frequently consider the initial time $\tau\ge 0$, which is convenient for proving the uniform-in-time bounds of solutions later on. 
		
		\begin{definition}\label{def:solution}
			Assume $p \geq  1$, $0 \leq \tau<T$ and $\ub_\tau=(u_{1,\tau},\ldots,u_{m,\tau}) \in (L^p(\R^N))^m$.	A vector of functions $\ub = (u_1,\ldots, u_m): \R^N\times [\tau,T)$ is called a mild solution to system 
			\begin{equation}\label{sys-tau}
				\begin{cases}
					\pa_t u_i + d_i(-\Delta)^{\alpha}u_i = f_i(x,t,\ub), &(x,t) \in {Q_{\tau,T}:=\R^N \times (\tau,T)}, \; i=1,\ldots, m,\\
					u_i(x,\tau) =u_{i,\tau}(x), &x\in\R^N, i=1,\ldots, m.
				\end{cases}
			\end{equation}
			if $\ub \in (C([\tau,T];L^p(\R^N)))^m$, $f_i(x,t,\ub(x,t))\in L^1((\tau,T);L^p(\R^N))$ and
			\begin{equation} \label{sys:Duhamel}
				u_i(t) = S_{i,\alpha}(t-\tau)u_{i,\tau} + \int_\tau^tS_{i,\alpha}(t-s)f_i(x,s,\ub(s))ds, \quad t \in (\tau,T), \quad i=1,\ldots,m,
			\end{equation}
			where $\{S_{i,\alpha}(t)\}_{t\geq \tau}$ is the semigroup generated by the species-dependent fractional operator $d_i(-\Delta)^{\alpha}$. Here in formula \eqref{sys:Duhamel}, we ignore the dependence of $u_i$ on $x$. We say that this solution is non-negative in $Q_{\tau,T}$ if $u_i(x,t)\ge 0$ for a.e. $x\in \R^N$, $t\geq \tau$ and all $i=1,\ldots, m$.
	\end{definition}
	
	Our first main result is the global existence and boundedness of solutions when the nonlinearities are of quadratic growth.
	\begin{theorem}[Quadratic growth rates]\label{thm:quadratic}
		Assume \eqref{F}, \eqref{P}, \eqref{M} and nonlinearities have at most quadratic growth rates, i.e. there exists a constant $C>0$ such that
		\begin{equation}\label{qgr} 
			|f_i(x,t,\ub)| \le C\bra{1+|\ub|^2} \quad  \forall (x,t)\in \R^N\times \R_+,\;\forall \ub\in \R_+^m, \quad \forall i=1,\ldots, m.
		\end{equation}
		Then for any non-negative initial data $\ub_0 = (u_{1,0},\ldots,u_{m,0}) \in (\LO{1}\cap \LO{\infty})^m$, there exists a unique global non-negative mild solution to \eqref{sys} which is bounded uniformly in time, i.e.
		\begin{equation*} 
			\limsup_{t\to\infty}\sup_{i=1,\ldots, m}\|u_i(t)\|_{\LO{\infty}} < +\infty.
		\end{equation*}
	\end{theorem}
	\begin{remark}\label{remark1}\hfill
		\begin{itemize}
			\item Thanks to the regularity of mild solutions (see Lemma \ref{lem:reg-sys}) and the boundedness, the solution obtained in Theorem \ref{thm:quadratic} is in fact a strong solution in the sense that $\partial_t u_i(t)$, $(-\Delta)^{\alpha}u_i$, $f_i(\cdot, t, u_i(t)) \in C^{\beta}(\R^N)$ for some $\beta\in (0,1)$ for a.e. $t\in (0,\infty)$ and the equations in \eqref{sys} are satisfied pointwise.
			\item The quadratic growth rate \eqref{qgr} can be slightly improved in the sense that there exists a small positive constant $\eps$ depending on $\alpha$, $N$, $m$, and diffusion coefficients $d_i$ such that the {result in Theorem \ref{thm:quadratic} still holds} if \eqref{qgr} is replaced by
			\begin{equation*} 
				|f_i(x,t,\ub)| \le C(1+|\ub|^{2+\eps}), \quad  \forall (x,t)\in \R^N\times \R_+,\; \forall \ub \in \R_+^m, \quad \forall i=1,\ldots, m.
			\end{equation*}
			\item As mentioned earlier, it is interesting that the quadratic growth of nonlinearities is independent of the fractional order $\alpha \in (0,1)$. Considering that the case {$\alpha = 0$, namely the case of ODEs}, is straightforward and the case of local diffusion $\alpha = 1$ has been proved in \cite{fellner2020globalA}, with some abuse of notation, we can say that result in Theorem \ref{thm:quadratic} is true for all $\alpha \in [0,1]$.
			\item {We also remark that the condition $u_{i,0}\in \LO{1} \cap \LO{\infty}$ ensures a uniform bound of the $L^1$-norm of the mild solution, which plays a role in the proof of the uniform-in-time bound in Theorem \ref{thm:quadratic}. If only the condition $u_{i,0}\in \LO{\infty}$ is imposed}, the global boundedness in time is in fact a delicate issue since there might be some infinite supply of fuel from infinity. We refer the interested reader to the recent work \cite{la2024uniform} for more details. Note also that the large time behavior of solutions in the entire space $\R^N$ can be also very different compared to the case of bounded domains, see e.g. \cite{mielke2024convergence}. The boundedness of mild solutions to \eqref{sys} with bounded initial data remains an interesting open problem.
			\item The results in Theorem \ref{thm:quadratic} are applicable to obtain the global existence and boundedness of solutions to bimolecular reactions $\mathsf S_1 + \mathsf S_3 \leftrightarrows \mathsf S_2 + \mathsf S_4$ with fractional diffusion
			\begin{equation*}
				\begin{cases}
					\pa_t u_i + d_i(-\Delta)^{\alpha}u_i = f_i(\ub):= (-1)^i(u_1u_3 - u_2u_4), &x\in\R^N, t>0, i=1,\ldots, 4,\\
					u_i(x,0) = u_{i,0}(x), &x\in\R^N, i=1,\ldots, 4.
				\end{cases}
			\end{equation*}
			Note that, due to the quadratic growth and symmetry, the global existence of bounded mild solutions to this system, even in the case of local diffusion $\alpha = 1$, is difficult and has been solved only recently in \cite{CGV19,fellner2020globalA,Sou18}. 
		\end{itemize}
	\end{remark}
	
	We now sketch the main ideas to prove Theorem \ref{thm:quadratic}. We first consider the case of mass conservation, i.e. \eqref{M} is satisfied with an equality sign. The case of mass dissipation can be easily transformed into the conservation case by adding an equation to the system, see Lemma \ref{lem:equivalence}. By summing up the equations in \eqref{sys} then \textit{integrating only in time}, one obtains
	\begin{equation*} 
		\sum_{i=1}^mu_i(x,t) + (-\Delta)^{\alpha}\int_0^t\sum_{i=1}^md_iu_i(x,s)ds = \sum_{i=1}^mu_{i,0}(x), \quad x\in \R^N, t>0.
	\end{equation*}
	Denoting $v(x,t):= \int_0^t\sum_{i=1}^md_iu_i(x,s)ds$, we see that the $\LO{\infty}$-norm of $u_i$ can be estimated from above by the $\LO{\infty}$-norm of $(-\Delta)^{\alpha}v$. To estimate the latter, we notice that $v$ solves the following parabolic equation with non-local diffusion
	\begin{equation*} 
		b(x,t)\partial_t v + (-\Delta)^\alpha v = \sum_{i=1}^mu_{i,0}
	\end{equation*}
	where the function $b(x,t)$ is defined as
	\begin{equation*} 
		\frac{1}{\max_{1 \leq i \leq m}\{d_i\}} \le b(x,t):= \frac{\sum_{i=1}^m u_i(x,t)}{\sum_{i=1}^md_iu_i(x,t)} \le \frac{1}{\min_{1 \leq i \leq m} \{d_i\}} \quad \forall (x,t)\in \R^N\times\R_+.
	\end{equation*}
	At this point, we refine the De Giorgi machinery for inhomogeneous non-local parabolic equation, combining \cite{caffarelli2011regularity} and \cite{vasseur2016giorgi}, to show that $v$ is H\"older continuous of order $\gamma \in (0,1)$. We can always choose $\gamma < \alpha$ if necessary. This is the starting point of the following series of estimates, with $U = \sup_{Q_T}\sup_{i=1,\ldots, m}u_i(x,t)$ and $Q_T = \R^N\times (0,T)$,
	\begin{align*}
		U \lesssim 1 + \sup_{Q_T}|(-\Delta)^\alpha v| \lesssim (\sup_{Q_T} |(-\Delta)^{\frac{\alpha}{2}}v|)^{\frac{1}{2}}\bra{1 + \sum_{i=1}^m\sup_{Q_T} |(-\Delta)^{\frac{\alpha}{2}}u_i|}^{\frac{1}{2}}\\
		\lesssim \bra{1 + U^{\frac{\alpha-\gamma}{2\alpha - \gamma}}}^{\frac{1}{2}}\bra{1 + {U^{\frac{3}{2}}}}^{\frac{1}{2}} \lesssim 1 + U^{\frac{\alpha-\gamma}{2(2\alpha-\gamma)} + \frac{3}{4}}.
	\end{align*}
	Since the exponent of $U$ on the right hand side is smaller than one, by Young's inequality, we derive $U\lesssim 1$, which guarantees the global existence of the local bounded solution. The uniform boundedness in time is then obtained by using a time cut-off function to show that $\|u_i\|_{L^{\infty}(\R^N\times(\tau,\tau+1))} \le C$ for a constant $C$ independent of $\tau$.
	
	\medskip
	When the nonlinearities are of super-quadratic growth rates, results in \cite{pierre2023examples} suggest that further assumptions need to be imposed. One {of} such assumptions, which is quite natural and appears frequently in applications, is the following
	\begin{enumerate}[ref=\textbf{ISC}, label=(\textbf{ISC}), align=left] 
		\item\label{ISC}({$\rho$}-order intermediate sum condition) for each $i\in \{1,\ldots, m-1\}$ there are $a_{ij}\ge 0$, $j=1,\ldots, i$ with $a_{ii}>0$ such that 
		\begin{equation*}
			\sum_{j=1}^{i}a_{ij}f_j(x,t,\ub)\le C(\Phi(x,t)+|\ub|^{\rho})\quad \forall x \in \R^N, \, t \ge 0, \, \ub\in \R_+^m,
		\end{equation*}
		for  constants $C>0$ and {$\rho>0$} independent of $i, j$, $\ub$, and $0\le \Phi \in L^1(Q_T)\cap L^\infty(Q_T)$ for any $T>0$. 
		
		By dividing by $a_{ii}>0$ for each $i=1,\ldots, m$, we can assume, without loss of generality, that $a_{ii}=1$ for all $i=1,\ldots, m$.
	\end{enumerate}
	We also assume that the nonlinearities are \textit{bounded from above} by a polynomial of {\textit{arbitrary order}}
	\begin{enumerate}[ref=\textbf{Pol}, label = (\textbf{Pol}), align=left]
		\item\label{Pol} there exist $C, \nu > 0$ such that
		\begin{equation*} 
			f_i(x,t,\ub)\le C\bra{\Phi(x,t)+|\ub|^{\nu}}, \quad \forall x\in \R^N, \;t\ge 0,\; \ub\in \R_+^m, \quad \forall i=1,\ldots, m,
		\end{equation*}
		where we take the same function $\Phi$ as in \eqref{ISC} for simplicity.
	\end{enumerate}
	It is remarked that there is no restriction on the exponent $\nu>0$. It can be seen from \eqref{ISC} and \eqref{Pol} that except for the first nonlinearity\footnote{Due to permutation, it could be any of the nonlinearities.} all other nonlinearities can have arbitrary growth rates provided some good ``cancellation'' through linear combination as in \eqref{ISC}. The following second result, therefore, covers a large class of systems that cannot be handled by Theorem \ref{thm:quadratic}.
	\begin{theorem}[Intermediate sum condition]\label{thm:ISC}
		Assume \eqref{F}, \eqref{P}, \eqref{M}, \eqref{Pol}, and \eqref{ISC}. Then there exists a constant $\eps_* > 0$ depending only on $N,\alpha$, $d_i$, $i=1,\ldots,m$, such that if ${\rho}$ satisfies
		\begin{equation} \label{cond:mu}
			1\le {\rho} \le \min\left\{1 + \frac{2\alpha(2+\eps_*)}{N+2\alpha},2\right\},
		\end{equation}
		then for any non-negative initial data $\ub_0\in (\LO{1}\cap \LO{\infty})^m$, there exists a global mild solution to \eqref{sys} according to Definition \ref{def:solution}. Moreover, assume that $\sup_{T>0}\|\Phi\|_{L^1(Q_T)\cap L^{\infty}(Q_T)} < +\infty$, then the solution is bounded uniformly in time, i.e.
		\begin{equation}\label{uniform-bound-2}
			\limsup_{t\to +\infty}\sup_{i=1,\ldots, m}\|u_i(t)\|_{\LO{\infty}} < +\infty.
		\end{equation}
	\end{theorem}
	\begin{remark}\hfill
		\begin{itemize}
			\item Similarly to Remark \ref{remark1}, the solution obtained in Theorem \ref{thm:ISC} is in fact a strong solution.
			\item Unlike Theorem \ref{thm:quadratic}, the order {$\rho$} of the intermediate sum condition in Theorem \ref{thm:ISC} depends on $\alpha$, and we see that ${\rho} \to 1$ as $\alpha \to 0$.
			\item When $N\ge 3$, we see that ${\rho}<2$ (unless $\eps_*$ is large, which is still unclear), so we have sub-quadratic intermediate sum condition in three or higher dimensions. When $N = 2$, if $\alpha\in (0,1)$ but close enough $1$, we still can have ${ \rho} = 2$, which extends the results in \cite{morgan2020boundedness}. 
		\end{itemize}
	\end{remark}
	To prove Theorem \ref{thm:ISC}, we develop duality methods, see e.g. \cite{pierre2010global}, for the case of the fractional Laplacian in the whole space $\R^N$. More precisely, we first prove an $L^p-L^q$ regularizing effect of the fractional heat operator. Then by showing an improved duality estimate, cf. \cite{canizo2014improved}, we get $u_i\in \LQT{2+\eps_0}$ for some $\eps_0>0$. Now, combining this with the fact that $f_1(x,t,\ub)$ is bounded from above by $\Phi + |\ub|^{{\rho}}$, we can use the above $L^p-L^q$ regularizing effect to get $u_1 \in \LQT{p_1}$ for some $p_1 > p_0$. Then, by proving another duality estimate, we show that this $\LQT{p_1}$-bound can be propagated to $u_2$, $u_3$, $\ldots$, $u_m$. Repeating this bootstrap argument we get $u_i\in \LQT{p}$ for any $p<\infty$, and ultimately $u_i \in \LQT{\infty}$ thanks to \eqref{Pol}, which proves the global existence. The uniform-in-time boundedness is then again shown by using the time cut-off function.
	
	\medskip
	\textbf{The paper is organized as follows.} In the next section, we provide the preliminaries of \eqref{sys} including the local existence, blow-up criterion and non-negativity of mild solutions. Section \ref{sec:quadratic} provides the proof of Theorem \ref{thm:quadratic}, started with a regularization of fractional diffusion equation with respect to some a-priori H\"older continuity in subsection \ref{subsec:reg_frac_diff}, followed by some interpolation technique in subsection \ref{subsec:interp} to show that the $L^{\infty}$-bound of the mild solution is bounded, hence the global existence. In Section \ref{sec:isc}, we first show some $L^p-L^q$ regularizing effect of the fractional diffusion in subsection \ref{subsec:reg_frac_diff_LpLq}, then develop some duality method for non-local diffusion problems in subsection \ref{subsec:duality}, and finally present the proof of Theorem \ref{thm:ISC} using {a} bootstrap argument in subsection \ref{subsec:bootstrap}. The Appendix \ref{appendix} is devoted to some technical results.
	
	\medskip
	\textbf{Notation.} We denote by $C$ a generic constant, which can be different from line to line, or even in the same line. Occasionally we write $C_{T,\gamma,\text{etc}}$ to emphasize the dependence of $C$ on $T$, $\gamma$, etc. { The notation $A \gtrsim B$ (resp. $A \lesssim B$) means $A \geq c\,B$ (resp. $A \leq c\,B$) where the implicit $c$ is a positive constant depending on some initial parameters but independent of $A$ and $B$. If $A \gtrsim B$ and $A \lesssim B$, we write $A \asymp B$.}
	
	\section{Preliminaries}\label{preliminaries} 
	Let $K_\alpha$ be the fundamental solution to the fractional heat equation 
	$$\partial_t u + (-\Delta)^\alpha u = 0$$
	in $\R^N \times (0,+\infty)$, which can be defined via the Fourier transform by
	\begin{equation}\label{def:Kalpha}
	K_\alpha(x,t) = \CF^{-1}(e^{-t |\cdot|^{2\alpha}})(x).
	\end{equation}
	{where $\CF$ denote the Fourier transform}.
	By \cite[estimate (2.4)]{Vaz2018}, we have
	\begin{equation} \label{selfsim} K_\alpha(x,t) \asymp t(t^{\frac{1}{\alpha}}+|x|^2)^{-\frac{N+2\alpha}{2}} \quad \forall x \in \R^N,\, t >0.
	\end{equation}
	Moreover, $K_\alpha$ has the self-similar form
	$$ K_\alpha(x,t) = t^{-\frac{N}{2\alpha}}\tilde K_\alpha(t^{-\frac{1}{2\alpha}}x) \quad \forall x \in \R^N, \, t>0,
	$$
	where $\tilde K_\alpha$ is a radially symmetric function. It is known (see e.g. \cite[Lemma 2.1]{MYZ2008}) that
	\begin{align} \label{est:tilde-funda-1} |\tilde K_\alpha(x)| &\leq C(1+|x|)^{-N-2\alpha} \quad \forall x \in \R^N, \\ \label{est:grad-tilde-funda-1}
		|\nabla \tilde K_\alpha (x)| &\leq C(1+|x|)^{-(N+1)}\quad \forall x \in \R^N.	
	\end{align}
	As a consequence of the above estimates, $\tilde K_\alpha \in L^p(\R^N)$ for any $p \in [1,\infty]$. 
	Moreover, from \cite[Lemma 2.2 and Remark 2.1 (iii)]{MYZ2008}, we have
	\begin{equation} \label{est:fracfunda}
		|(-\Delta)^{\frac{\alpha}{2}} K_\alpha(x,t)| \leq C (t^{\frac{1}{2\alpha}}+|x|)^{-(N+\alpha)} \quad \forall x \in \R^N,\, t>0.	
	\end{equation}
	Let $S_\alpha(t)=e^{-t(-\Delta)^\alpha}$ be the semigroup generated by $(-\Delta)^\alpha$, namely $S_\alpha(t) \varphi = K_\alpha(\cdot,t)* \varphi$  for all $t>0$. By \cite[Lemma 3.1]{MYZ2008}, for any $1 \leq r \leq p \leq \infty$, $\varphi \in L^r(\R^N)$ and $\beta>0$,
	\begin{align} \label{est:Lp-S}
		\| S_\alpha(t) \varphi \|_{L^p(\R^N)} &\leq C(N,\alpha,r,p)t^{-\frac{N}{2\alpha}(\frac{1}{r} - \frac{1}{p}) }\| \varphi \|_{L^r(\R^N)}, \quad \forall t >0,	\\ \label{est:Lp-DS}
		\| (-\Delta)^{\beta} (S_\alpha(t) \varphi) \|_{L^p(\R^N)} &\leq C(N,\alpha,\beta,r,p) t^{-\frac{\beta}{\alpha} -  \frac{N}{2\alpha}(\frac{1}{r} - \frac{1}{p})} \| \varphi \|_{L^r(\R^N)}, \quad \forall t >0. 
	\end{align}
	
	Next we state basic results regarding the local existence, regularity and blowup criteria for problem \eqref{sys}.
	
	\begin{proposition} \label{prop:local-existence}
			Assume \eqref{F} holds. Let $\tau \geq 0$ and  $\ub_\tau=(u_{1,\tau},\ldots,u_{m,\tau}) \in (L^1(\R^N) \cap L^\infty(\R^N))^m$. 
			\begin{enumerate}
				\item  Local existence. There exists $T>\tau$ with $T-\tau$ depending only on $\| \ub_\tau \|_{L^1(\R^N) \cap L^\infty(\R^N)}$, $d_i$, and $f_i$, $1 \leq i \leq m$, such that system \eqref{sys-tau} admits a unique mild solution $\ub$ in $Q_{\tau,T}$ in the sense of Definition \ref{def:solution}.

				\item Blow-up criterion. Let $T_*$ be the maximal time for the existence of the solution $\ub$. If $T_*<\infty$ then 
				$$ \lim_{t \nearrow T_*}\sum_{i=1}^m \| u_i(t)\|_{L^\infty(\R^N)} = +\infty.
				$$
				\item {Non-negativity}. Assume that \eqref{P} holds. Then from $u_{i,\tau}\ge 0$ a.e. in $\R^m$ for all $i=1,\ldots, m$, it follows $u_{i}(\cdot, t)\ge 0$ a.e. in $\R^m$ for a.e. $t\in (0,T_{*})$.
			\end{enumerate}
	\end{proposition}
	\begin{proof}
		The proof of this proposition should be standard. For the sake of completeness, we sketch the main steps here.
		
		\medskip
		\noindent (1)\textit{ Local existence}. Let $ T > \tau$ to be determined later and define  
		$$ \mathscr{E}_T:= \{ \ub =(u_1,\ldots,u_m): \| \bm{u} \|_{C([\tau,T];L^1(\R^N) \cap L^\infty(\R^N))} \leq 2\| \ub_\tau \|_{L^1(\R^N) \cap L^\infty(\R^N)}  \}.
		$$		
		Define the solution mapping $\Phi(\ub)=(\Phi_1(\ub),\ldots,\Phi_m(\ub))$ by
		$$ \Phi_i(\ub)(x,t) := (S_{i,\alpha}(t-\tau)u_{i,\tau})(x) + \int_{\tau}^t (S_{i,\alpha}(t-s)f_i(\cdot,s,\ub(s)))(x)ds, \quad x \in \R^N, \, t \in (\tau,T),
		$$
		where $\{S_{i,\alpha}(t)\}_{t\ge 0}$ is the semigroup generated by  $d_i(-\Delta)^{\alpha}$.
		By the contraction property of $S_{i,\alpha}$ on $L^p(\R^N)$, for any $1 \leq p \leq \infty$, and the local Lipschitz continuity of $f_i(x,t,\cdot)$, we can choose $T-\tau$ sufficiently small such that $\Phi$ is a contraction mapping from $\mathscr E_T$ to $\mathscr E_T$. This gives the existence and uniqueness of the local mild solution. 

		\medskip
		\noindent (2) \textit{Blow-up criteria}. This follows from a straightforward contradiction argument.
		
		\medskip
		\noindent (3) \textit{Non-negativity.} Denote $u_{i}^{+} = \max\{0, u_i\}$ and $\ub^+ = (u_1^+,\ldots, u_m^+)$. Consider the  system
		\begin{equation}\label{aux_sys}
			\begin{cases}
				\partial_t u_i + (-\Delta)^{\alpha}u_i = f_i(x,t,\ub^+), &x\in\R^N, \; t\in (\tau,T_*),\\
				u_i(x,\tau) = u_{i,\tau}(x), &x\in \R^N.
			\end{cases}
		\end{equation}
		By using the same argument in (1), we can obtain a unique local solution. By multiplying the equation of $u_i$ by $u_i^-:= \min\{0, u_i\}$, then integrating over $\R^N$, we obtain
		\begin{equation*}
			\frac 12 \frac{d}{dt}\int_{\R^N}|u_i^{-}|^2dx + \int_{\R^N}|(-\Delta)^{\frac{\alpha}{2}}u_i^{-}|^2dx \le \int_{\R^N}f_i(x,t,\ub^+)u_i^{-}dx \le 0,
		\end{equation*}
		where at the last inequality, we use the quasi-positivity assumption \eqref{P}. Therefore, for $t\in (\tau,T_*)$, $\|u_i^-(t)\|_{\LO{2}} \le \|u_{i,\tau}^{-}\|_{\LO{2}} = 0$,
		since $u_{i,\tau}\ge 0$. This implies that $u_i(t)\ge 0$ for all $t\in (\tau,T_*)$ and all $i=1,\ldots, m$. It follows that $u^+ = u$, and hence the solution to \eqref{aux_sys} is also a solution to \eqref{sys-tau}. By the uniqueness, we obtain the non-negativity of solution to \eqref{sys-tau}, which finishes the proof.
	\end{proof}
	
	\begin{lemma} \label{lem:reg-sys}
	Assume \eqref{F} holds. Let $0 \leq \tau < T$, $\ub_\tau=(u_{1,\tau},\ldots,u_{m,\tau}) \in (L^1(\R^N) \cap L^\infty(\R^N))^m$ and $\ub$ is a mild solution of system \eqref{sys-tau} in $Q_{\tau,T}$. Then for any $1 \leq p \leq \infty$, $1<q,r<\infty$, $\beta \in (0,\min\{1,2\alpha\})$, $i=1,\ldots,m$, and $t_0 \in (\tau,T)$,
	\begin{equation} \label{lem:reg-ui} \begin{aligned}  
	&\| u_i \|_{L^\infty((t_0,T);C^\beta(\R^N))} +  \| \partial_t u_i \|_{L^\infty((t_0,T);C^\beta(\R^N))} +	\| (-\Delta)^{\frac{\alpha}{2}}u_i \|_{L^\infty((t_0,T);L^p(\R^N))} \\
	&\quad \quad + \| (-\Delta)^{\alpha} u_i \|_{L^q((t_0,T);L^r(\R^N))}	 + \| (-\Delta)^\alpha u_i \|_{L^\infty((t_0,T);C^\beta(\R^N))} \leq C.
	\end{aligned} \end{equation}
	The constant $C$ depends on $N,\alpha,\beta,t_0-\tau,T-\tau,f_i,\| u_{i,\tau}\|_{L^\infty(\R^N)},\| \bm{u}\|_{L^\infty((\tau,T);L^1(\R^N) \cap L^\infty(\R^N))})$. In particular, for a.e. $t \in (\tau,T)$,
	$u(t), \partial_tu(t), (-\Delta)^\alpha u(t) \in C^{\beta}(\R^N)$.
	\end{lemma}
	\begin{proof}
	Let $p \in [1,\infty]$. By \eqref{est:Lp-DS} with $p=r$ and $\beta=\alpha/2$, we have, for any $i=1,\ldots,m$, 
	$$ \| (-\Delta)^{\frac{\alpha}{2}}(S_{i,\alpha}(t-\tau)u_{i,\tau})\|_{L^p(\R^N)} \leq C(t-\tau)^{-\frac{1}{2}}\| u_{i,\tau}\|_{L^1(\R^N) \cap L^\infty(\R^N)}.
	$$
	Next, by Lemma \ref{lem:interchange-2} (for $p<+\infty$) and Lemma \ref{lem:interchange-3} (for $p=+\infty$), Young's convolution inequality and \eqref{est:tilde-funda-1}, \eqref{est:grad-tilde-funda-1}, we see that
	\begin{align*} \left\| (-\Delta)^{\frac{\alpha}{2}} \int_{\tau}^tS_{i,\alpha}(t-s)f_i(\cdot,s,\bm{u}(s))ds \right \|_{L^p(\R^N)} &\leq (t-\tau)^{\frac{1}{2}}\| f_i(\cdot,\cdot,\ub)\|_{L^\infty((\tau,t);L^1(\R^N) \cap L^\infty(\R^N))} \\
		&\leq L_i(t-\tau)^{\frac{1}{2}}\| \bm{u} \|_{L^\infty((\tau,t);L^1(\R^N) \cap L^\infty(\R^N))},
	\end{align*}
	where $L_i$ is the Lipschitz constant of $f_i$ with respect to the argument $\ub$ in the ball of center zero and radius $\| \ub \|_{L^\infty((\tau,t);L^1(\R^N) \cap L^\infty(\R^N))}$.
	Combining formulation \eqref{sys:Duhamel} and the above estimates yields 
	\begin{equation} \label{est:frac-u-sys} \begin{aligned} 
		\| (-\Delta)^{\frac{\alpha}{2}} u_i(t) \|_{L^p(\R^N)} &\leq (t-\tau)^{-\frac{1}{2}}\| u_{i,\tau}\|_{L^1(\R^N) \cap L^\infty(\R^N)} \\
		&+ L_i(t-\tau)^{\frac{1}{2}}\| \bm{u} \|_{L^\infty((\tau,t);L^1(\R^N) \cap L^\infty(\R^N))}.
	\end{aligned}	\end{equation}
	Let $r \in (1,\infty)$ and $t_0 \in (\tau,T)$. By \eqref{est:Lp-DS}, we have, for any $i=1,\ldots,m$, and $t \in (t_0,T)$,
	\begin{equation} \label{est:frac-alpha-far-1} \begin{aligned}\| (-\Delta)^{\alpha}(S_{i,\alpha}(t-\tau)u_{i,\tau})\|_{L^r(\R^N)} 
		\leq C(t-\tau)^{-1}\| u_{i,\tau}\|_{L^1(\R^N) \cap L^\infty(\R^N)},
		\end{aligned} 
	\end{equation} 
	where $C$ is a positive constant depending on $N,r,\alpha$. Now let $\phi_i$ be the solution to problem 
	\begin{equation}\label{sys-phi}
		\pa_t \phi_i + d_i(-\Delta)^{\alpha}\phi_i = f_i(x,t,\bm{u}), \; (x,t) \in Q_{\tau,T},\quad 
		\phi_i(x,\tau) = 0, \; x\in\R^N.
	\end{equation}
Let $q \in (1,\infty)$. By the $L^p-L^q$ maximal regularity in Theorem \ref{thm:LpLq-maximal}, we have
	\begin{equation} \label{est:frac-alpha-far-2} \begin{aligned} \| (-\Delta)^\alpha \phi_i \|_{L^q((\tau,T);L^r(\R^N))} 
			&\leq C\| f_i(x,t,\bm{u})\|_{L^q((\tau,T);L^r(\R^N))} \\
			& \leq C(T-\tau)^{\frac{1}{q}}\| f_i(x,t,\bm{u})\|_{L^\infty((\tau,T);L^1(\R^N) \cap L^\infty(\R^N))} \\
			&\leq CL_i(T-\tau)^{\frac{1}{q}} \| \bm{u} \|_{L^\infty((\tau,T);L^1(\R^N) \cap L^\infty(\R^N))},
		\end{aligned}
	\end{equation}
	where $C$ is a positive constant depending only on $N,r,q,\alpha$.
	Using the relation $u_i(t) = S_{i,\alpha}(t-\tau)u_{i,\tau} + \phi_i(t)$, together with \eqref{est:frac-alpha-far-1} and \eqref{est:frac-alpha-far-2}, we derive
	\begin{equation} \label{est:frac-alp} \begin{aligned}
		\| (-\Delta)^\alpha u_i \|_{L^q((t_0,T);L^r(\R^N))} &\leq C(t_0-\tau)^{-1}\| u_{i,\tau}\|_{L^1(\R^N) \cap L^\infty(\R^N)} \\
		&+ CL_i(T-\tau)^{\frac{1}{q}} \| \bm{u} \|_{L^\infty((\tau,T);L^1(\R^N) \cap L^\infty(\R^N))}.
	\end{aligned} \end{equation}
	Let $0<\beta < \min\{1,2\alpha\}$. From Remark \ref{est:Holder-u}, we deduce that, for any $t \in (\tau,T)$, 
	\begin{equation} \label{est:u-Holder-x}
		\| u_i(t) \|_{C^\beta(\R^N)} \leq C (t-\tau)^{-\frac{\beta}{2\alpha}}\| u_{i,\tau}\|_{L^\infty(\R^N)} + CL_i(t-\tau)^{1-\frac{\beta}{2\alpha}}\| \bm{u}\|_{L^\infty(Q_{\tau,T})}.
	\end{equation}
	Consequently, for any $\alpha<\beta<\min\{1,2\alpha\}$ and any $t_0 \in (\tau,T)$, we have
	\begin{align} \label{est:u-Holder-far}
		\| u_i \|_{L^\infty((t_0,T);C^\beta(\R^N))} \leq C (t_0-\tau)^{-\frac{\beta}{2\alpha}}\| u_{i,\tau}\|_{L^\infty(\R^N)} + CL_i(T-\tau)^{1-\frac{\beta}{2\alpha}}\| \bm{u}\|_{L^\infty(Q_{\tau,T})},	
	\end{align}
	where $C$ depends on $N,\alpha,\beta$. Let $\varphi:[\tau,T]$ be a smooth function such that $\varphi = 0$ in $[\tau,\frac{\tau+t_0}{2}]$, $\varphi = 1$ in $[t_0,T]$ and put $v_i=\varphi u_i$. Then  
	\begin{equation*}
		\begin{cases}
			\pa_t v_i + d_i(-\Delta)^{\alpha} v_i = \varphi' u_i + \varphi f_i(x,t,u), &(x,t) \in Q_{\frac{\tau+t_0}{2},T},\\
			v_i(x,\frac{\tau+t_0}{2}) = 0, &x\in\R^N,
		\end{cases}
	\end{equation*}
	Now let $\gamma=\frac{1}{2}(\beta +\min\{1,2\alpha\})$. By \cite[(A.7)]{caffarelli2013regularity}, we have
	\begin{align*}
		&\| \partial_ t v_i \|_{L^\infty((\frac{\tau+t_0}{2},T);C^\beta(\R^N))} + \| (-\Delta)^\alpha v_i \|_{L^\infty((\frac{\tau+t_0}{2},T);C^\beta(\R^N))} \\
		&	\leq C(1+ \| \varphi' u_i + \varphi f_i(x,t,\bm{u}) \|_{L^\infty((\frac{\tau+t_0}{2},T);C^\gamma(\R^N))}) \\
		&\leq C(1 +  \| \bm{u} \|_{L^\infty((\frac{\tau+t_0}{2},T);C^\gamma(\R^N))}) \\
		&\leq 
		C(1 + (t_0-\tau)^{-\frac{\beta}{2\alpha}}\| u_{i,\tau}\|_{L^\infty(\R^N)} + L_i(T-\tau)^{1-\frac{\beta}{2\alpha}}\| \bm{u}\|_{L^\infty(Q_{\tau,T})}) 
	\end{align*}
	where $C$ depends on $N,\alpha,\beta$. This implies that
	\begin{align*} 
		&\| \partial_ t u_i \|_{L^\infty(t_0,T);C^\beta(\R^N))} + \| (-\Delta)^\alpha u_i \|_{L^\infty((t_0,T);C^\beta(\R^N))} \\
		&\leq C(1 + (t_0-\tau)^{-\frac{\beta}{2\alpha}}\| u_{i,\tau}\|_{L^\infty(\R^N)} + L_i(T-\tau)^{1-\frac{\beta}{2\alpha}}\| \bm{u}\|_{L^\infty(Q_{\tau,T})}), 
	\end{align*}
	where $C$ depends on $N,\alpha,\beta$.	This, together with \eqref{est:frac-u-sys}, \eqref{est:frac-alp} and \eqref{est:u-Holder-far},  implies \eqref{lem:reg-ui}.
	\end{proof}
	
	\section{Systems with quadratic growth rates}\label{sec:quadratic}
	\subsection{Regularization with H\"older continuity}\label{subsec:reg_frac_diff}
		Let $1 \leq p < \infty$, $\mu>0$, $0 \leq \tau<T$, $u_\tau \in L^p(\R^N)$ and $f \in L^1((\tau,T);L^p(\R^N))$. Let $\{S_{\alpha,\mu}(t)\}_{t \geq 0}$ is the semigroup generated by $\mu (-\Delta)^{\alpha}$. Since $\{S_{\alpha,\mu}(t)\}_{t \geq \tau}$ is a strongly continuous semigroup of contractions on $L^p(\R^N)$ (see e.g. \cite{Kwa2017}), the initial value problem
		\begin{equation} \label{prob:linear}
			\left\{  \begin{aligned}
				\partial_t u + \mu(-\Delta)^\alpha u &= f \quad &&\text{in } Q_{\tau,T}, \\
				u(\cdot,\tau) &= u_\tau \quad &&\text{in } \R^N, 
			\end{aligned} \right.	
		\end{equation}
		admits a unique mild solution $u$ in the sense that $u \in C([\tau,T];L^p(\R^N))$ and $u$ satisfies the Duhamel's formula
		\begin{equation} \label{eq:Duhamel-1} 
			u(t) = S_{\alpha,\mu}(t-\tau)u_\tau + \int_\tau^t S_{\alpha,\mu}(t-s)f(s)ds \quad \forall t \in [\tau,T),
		\end{equation}
		 (see e.g. \cite[Page 106]{Pazy1983}). Here we ignore the $x$-dependence in the notation. Moreover, since $-(-\Delta)^\alpha$ is $m$-accretive with dense domain (see e.g. \cite[Theorem 1.3.12]{BrezisCazenave1993}), the solution $u$ defined by \eqref{eq:Duhamel-1} satisfies  (see e.g. \cite[Lemma 1.5.3]{BrezisCazenave1993})
		 \begin{equation} \label{est:Lp-Lp} 
		 	\| u \|_{C([\tau,T];L^p(\R^N))} \leq \| u_\tau \|_{L^p(\R^N)} + \| f \|_{L^1((\tau,T);L^p(\R^N))}.
		 \end{equation}

		 \begin{remark}
		 	When $p=\infty$, we consider the space 
		 	$C_0(\R^N):=\{v \in C(\R^N): \lim_{|x| \to \infty}v(x)=0\}$.  
		 	Since $\{S_{\alpha,\mu}(t)\}_{t \geq \tau}$ is a strongly continuous semigroup of contractions on $C_0(\R^N)$, we see that, for any $u_\tau \in C_0(\R^N)$ and $f \in L^1((\tau,T);C_0(\R^N))$, the function $u$ defined by \eqref{eq:Duhamel-1} is the unique weak solution of \eqref{prob:linear}. Moreover, by \cite[Lemma 1.5.3]{BrezisCazenave1993}, $u \in C([\tau,T];C_0(\R^N))$ and
		 	$$ \| u \|_{C([\tau,T];C_0(\R^N))} \leq \| u_\tau \|_{C_0(\R^N)} + \| f \|_{L^1((\tau,T);C_0(\R^N))}.
		 	$$
		 \end{remark}

		 \begin{remark}
		 	It follows from \eqref{eq:Duhamel-1} and \eqref{est:Lp-S} that if $u_\tau \in L^p(\R^N)$ and $f \in L^p(Q_{\tau,T})$ for $1 \leq p < \infty$, then the unique solution $u$ of \eqref{prob:linear} satisfies
		 	\begin{equation} \label{est:mild-1}
		 		\| u(t) \|_{L^p(\R^N)} \leq C( \| u_\tau \|_{L^p(\R^N)} + (t-\tau)^{\frac{p-1}{p}}\| f \|_{L^p(Q_{\tau,t})}) \quad \forall t \in (\tau,T),
		 	\end{equation}
		 	where $C$ depends only on $N,\alpha,\mu,p$. 
		 	Similarly, if $u_0 \in C_0(\R^N)$ and $f \in C([\tau,T];C_0(\R^N))$ then the unique solution $u$ of \eqref{prob:linear} satisfies
		 	\begin{equation*} 
		 		\| u(t) \|_{L^\infty(\R^N)} \leq C( \| u_\tau \|_{L^\infty(\R^N)} + (t-\tau)\| f \|_{L^\infty(Q_{\tau,t} )} ) \quad \forall t \in (\tau,T)
		 	\end{equation*}
		 	where $C$ depends only on $N,\alpha,\mu$. 
		 	Consequently, for $1 \leq p \leq \infty$,
		 	\begin{equation} \label{est:mild-1b}
		 		\| u \|_{L^p(Q_{\tau,T})} \leq C_{T-\tau}( \| u_\tau \|_{L^p(\R^N)} +\| f \|_{L^p(Q_{\tau,T})} ),
		 	\end{equation}
		 	where $C_{T-\tau}$ depends only on $N,\alpha,p$ and $T-\tau$.
		 \end{remark}
		 
		 Moreover, by the same argument as in the proof of Lemma \ref{lem:reg-sys}, if $f\in L^\infty((\tau,T);C^{\gamma}(\R^N))$ for some $\gamma \in (0,1)$ then $u$ enjoys the regularity $u(t), \partial_tu(t), (-\Delta)^\alpha u(t) \in C^{\beta}(\R^N)$ for each $t\in (\tau,T)$ any $\beta \in (0,\min\{\gamma,2\alpha\})$.
%

	\begin{theorem} \label{th:regular-1}
		Assume $\mu >0$, $\alpha<\theta<1$,  $\frac{N}{\alpha}<p \leq \infty$, $\tau \in [0,T)$, $u_\tau \in C^{\theta}(\R^N) \cap L^p(\R^N)$ and $f \in L^\infty((\tau,T);L^p(\R^N))$. Let $u$ be the mild solution to problem
		\begin{equation}\label{prob:QtauT}
			\begin{cases}
				\partial_t u + \mu(-\Delta)^\alpha u = f \quad \text{in }  Q_{\tau,T},\\
				u(x,\tau) = u_\tau(x), \quad x\in\R^N.
			\end{cases}
		\end{equation}
		Assume that
		\begin{equation} \label{assump:Holder} |u(x,t) - u(y,t)| \leq H|x-y|^{\gamma} \quad \forall x,y \in \R^N, \, t \in (\tau,T),
		\end{equation}
		for some $H > 0$ and $\gamma \in [0,\alpha)$. Then 
		\begin{equation} \label{est:interpolation-1}
			\|(-\Delta)^{\frac{\alpha}{2}} u \|_{L^\infty(Q_{\tau,T})} \lesssim \|u_\tau\|_{C^{\theta}(\R^N)} +  \| f \|_{L^\infty((\tau,T);L^p(\R^N))}^{\frac{(\alpha-\gamma)p}{2\alpha p - N -\gamma p} }H^{\frac{\alpha p - N}{2\alpha p - N - \gamma p}}.
		\end{equation}
	\end{theorem}
	\begin{proof} 
		Denote by $K_{\alpha,\mu}$ the heat kernel associated to the operator $\mu (-\Delta)^{\alpha}$. Then we have $K_{\alpha,\mu}(x,t) =K_\alpha(x,\mu t)$ with $K_\alpha$ defined in \eqref{def:Kalpha}.  Moreover,
		$S_{\alpha,\mu}(t)\varphi := K_{\alpha,\mu} (\cdot,t) * \varphi$ for $t>0$. 
		Let $k>0$ to be made precisely later on. Since $u$ is a solution of \eqref{prob:QtauT}, it is also a solution of 
		\begin{equation} \label{prob:linear-kappa}
			\left\{  \begin{aligned}
				\partial_t u + \mu (-\Delta)^\alpha u + k u &= f + k u \quad &&\text{in } Q_{\tau,T}, \\
				u(\cdot,\tau) &= u_\tau \quad &&\text{in } \R^N,  
			\end{aligned} \right.	
		\end{equation}	
		namely 
		\begin{equation} \label{eq:Duhamel-k} u(x,t) = e^{-k t}S_{\alpha,\mu}(t-\tau)u_\tau(x) + \int_\tau^t e^{-k(t-s)} S_{\alpha,\mu}(t-s)(f(x,s) + ku(x,s))ds, \quad \forall (x,t) \in Q_{\tau,T}.
		\end{equation}

		Since $u_\tau \in C^{\theta}(\R^N)$ with $\alpha<\theta<1$, by Lemma \ref{lem:fracDeltaSphi-Holder} with $\beta=\alpha/2$, we deduce that 
		\begin{equation}  \label{est:frac-sol-homogeneous} \| (-\Delta)^{\frac{\alpha}{2} }S_{\alpha,\mu}(t-\tau) u_\tau\|_{C^{\theta-\alpha}(\R^N)} \leq C \| u_\tau \|_{C^{\theta}(\R^N)}, \quad \forall t \in (\tau,T).
		\end{equation}
		
{Next, by Lemmata  \ref{lem:interchange-2}, \ref{lem:interchange-3}, \ref{lem:HolderHolder} (if $p=\infty$) and Lemmata \ref{lem:interchange}--\ref{lem:interchange-2} (if $p<\infty$), and  the fact that 
$\int_{\R^N}(-\Delta)^{\frac{\alpha}{2}} K_{\alpha,\mu}(x-y,t-s)dy=0$, we obtain
		\begin{equation} \label{est:reg-alpha/2-1} \begin{split}
				&(-\Delta)^{\frac{\alpha}{2}}\int_\tau^t e^{-k(t-s)} [S_{\alpha,\mu}(t-s)(f(\cdot,s)+ku(\cdot,s))](x)ds \\
				&= \int_\tau^t e^{-k(t-s)} (-\Delta)^{\frac{\alpha}{2}}[S_{\alpha,\mu}(t-s)(f(\cdot,s)+ku(\cdot,s))](x)ds  \\
				&=\int_\tau^t e^{-k(t-s)} \int_{\R^N} (-\Delta)^{\frac{\alpha}{2}} K_{\alpha,\mu}(x-y,t-s) [f(y,s)+k(u(y,s)-u(x,s))]dyds. 
		\end{split}\end{equation}	
		This, together with \eqref{est:fracfunda} and the H\"older continuity assumption \eqref{assump:Holder} on $u$, implies 
		\begin{equation} \label{interchange-2}  \begin{split}
				&\left|(-\Delta)^{\frac{\alpha}{2}}\int_\tau^t e^{-k(t-s)} S_{\alpha,\mu}(t-s)[f(x,s) + k u(x,s)]ds\right|\\
				& \leq C\| f \|_{L^\infty((\tau,T);L^p(\R^N))}\int_\tau^t e^{-k(t-s)} \left(\int_{\R^N} [(t-s)^{\frac{1}{2\alpha}} + |x-y| ]^{-(N+\alpha)p'}dy\right)^{\frac{1}{p'} } ds \\
				& + CkH\int_\tau^t e^{-k(t-s)} \int_{\R^N} [(t-s)^{\frac{1}{2\alpha}} + |x-y| ]^{-(N+\alpha)}|x-y|^{\gamma}dyds \\
				&=: I_1 + I_2,
		\end{split}\end{equation}	 		
where $1/p + 1/p' = 1$. By the change of variable $z=(t-s)^{-\frac{1}{2\alpha}}(y-x)$, we have
		\begin{align*}
			I_1 &= C\| f \|_{L^\infty((\tau,T);L^p(\R^N))}\int_\tau^t e^{-k(t-s)}(t-s)^{\frac{-(N+\alpha)p'+N}{2\alpha p'}} \left(\int_{\R^N}(1+|z|)^{-(N+\alpha)p'}dz\right)^{\frac{1}{p'}}ds \\
			&\leq \tilde A_1\| f \|_{L^\infty((\tau,T);L^p(\R^N))}k^{\frac{N}{2\alpha p} -  \frac{1}{2}}.
		\end{align*}
		Similarly, we have
		$$
		I_2 \leq CkH\int_\tau^t e^{-k(t-s)} (t-s)^{-\frac{\alpha-\gamma}{2\alpha}} \int_{\R^N} (1 + |z|)^{-(N+\alpha-\gamma)}dzds \leq \tilde A_2 H k^{\frac{\alpha-\gamma}{2\alpha}}.
		$$
		Combining the above estimates, we derive
		\begin{align*}
			&\left|(-\Delta)^{\frac{\alpha}{2}}\int_\tau^t e^{-k(t-s)} S_{\alpha,\mu}(t-s)[f(x,s) + k u(x,s)]ds\right| \leq \tilde A_1\| f \|_{L^\infty((\tau,T);L^p(\R^N))}k^{\frac{N}{2\alpha p}-\frac{1}{2}} + \tilde A_2 H k^{\frac{\alpha-\gamma}{2\alpha}}.
		\end{align*}	
		By minimizing the right hand side of the above estimate over $k>0$, 
		we derive that
		\begin{equation} \label{est:frac-sol-nonho-Lp}
			\left|(-\Delta)^{\frac{\alpha}{2}}\int_\tau^t e^{-k(t-s)} S_\alpha(t-s)[f(x,s) + k u(x,s)]ds\right| \leq \tilde A_3 \| f \|_{L^\infty((\tau,T);L^p(\R^N))}^{\frac{(\alpha-\gamma)p}{2\alpha p - N -\gamma p} }H^{\frac{\alpha p - N}{2\alpha p - N - \gamma p}}.
		\end{equation}
		From formula \eqref{eq:Duhamel-k} and estimates \eqref{est:frac-sol-homogeneous}, \eqref{est:frac-sol-nonho-Lp}, we derive \eqref{est:interpolation-1}. }
	\end{proof}

	\begin{theorem} \label{th:inter-2}
		Let $u$ be a mild solution of \eqref{prob:QtauT} in $Q_{\tau,T}$ with $u_\tau=0$ and assume $u$ satisfies \eqref{assump:Holder} for some $\gamma \in [0,\alpha)$. 
		
		\noindent	(i) Assume $f \in L^\infty((\tau,T);C^{\theta}(\R^N))$ for $\alpha<\theta<1$. Then
		\begin{equation}\label{est:1}
			\| (-\Delta)^\alpha u\|_{L^\infty(Q_{\tau,T})} \leq C	\| (-\Delta)^{\frac \alpha 2}u\|_{L^\infty(Q_{\tau,T})}^{\frac{1}{2}} \|(-\Delta)^{\frac\alpha 2}f\|_{L^\infty(Q_{\tau,T})}^{\frac{1}{2}},
		\end{equation}
		\begin{equation} \label{est:regular-1-0} 
			\| (-\Delta)^\alpha u\|_{L^\infty(Q_{\tau,T})} \leq C\| f \|_{L^\infty((\tau,T);C^{\theta}(\R^N))}^{\frac{1}{2}+\frac{\alpha-\gamma}{2(2\alpha-\gamma)}}H^{\frac{\alpha}{2(2\alpha-\gamma)}}.
		\end{equation}
		
		\noindent	(ii) Assume $f \in L^\infty((\tau,T);H^{\alpha,p}(\R^N))$ for $\frac{N}{\alpha}<p<\infty$. Then
		\begin{equation} \label{est:regular-1-3} 
			\| (-\Delta)^\alpha u \|_{L^\infty(Q_{\tau,T})} 
			\leq  C \| f \|_{L^\infty((\tau,T);H^{\alpha,p}(\R^N))}^{\frac{3\alpha p - N - 2\gamma p}{2(2\alpha p - N -\gamma p)} } H^{\frac{\alpha p - N}{2(2\alpha p - N - \gamma p)}}.
		\end{equation}
		Here the spaces $H^{\alpha,p}(\R^N)$ are defined in 	\eqref{H^alpha}.
	\end{theorem}
	\begin{proof}

		Let $k>0$. By using \eqref{eq:Duhamel-k}  with $u_\tau=0$, we have 
		\begin{equation} \label{formula:Duhamel-2} u(x,t) = \int_\tau^t e^{-k(t-s)}S_{\alpha,\mu}(t-s)(f(x,s) + ku(x,s))ds.
		\end{equation}
		
		(i) Assume $f \in L^\infty((\tau,T);C^{\theta}(\R^N))$ for $\alpha<\theta<1$. Since $\theta>\alpha$, it follows that $(-\Delta)^{\frac{\alpha}{2}}f \in L^\infty((\tau,T);L^\infty(\R^N))$. By \eqref{est:mild-1b}, $u \in L^\infty(Q_{\tau,T})$ and by Theorem \ref{th:regular-1}, $(-\Delta)^{\frac{\alpha}{2}}u \in L^\infty(Q_{\tau,T})$. 
		From \eqref{formula:Duhamel-2}, Lemma \ref{lem:interchange-3} with $\beta=\alpha/2$, Lemma \ref{lem:HolderHolder} with $\beta=\alpha/2$ and Lemma \ref{lem:interchange-2}, we obtain
		\begin{equation} \label{formula:Duhamel-23} \begin{aligned}
				(-\Delta)^\alpha u(x,t) 
				=  \int_\tau^t e^{-k(t-s)}(-\Delta)^{\frac{\alpha}{2}}S_{\alpha,\mu}(t-s)((-\Delta)^{\frac{\alpha}{2}}f(x,s) + k(-\Delta)^{\frac{\alpha}{2}}u(x,s))ds. 
		\end{aligned} \end{equation}
		Consequently, by using estimate \eqref{est:Lp-DS} with $\beta=\alpha/2$, $r=p=\infty$, we have, for any $t \in (\tau,T)$,
		\begin{align} \nonumber
			&\| (-\Delta)^\alpha u(\cdot,t) \|_{L^\infty(\R^N)} \\ \nonumber
			&\leq C\int_\tau^t e^{-k(t-s)} (t-s)^{-\frac{1}{2}} (\| (-\Delta)^{\frac{\alpha}{2}} f(\cdot,s)\|_{L^\infty(\R^N)} + k \| (-\Delta)^{\frac{\alpha}{2}} u(\cdot,s)\|_{L^\infty(\R^N)}) ds\\ \nonumber  
			&\leq Ck^{-\frac{1}{2}}\|(-\Delta)^{\frac\alpha 2} f \|_{L^\infty(Q_{\tau,T})} + Ck^{\frac{1}{2}} \| (-\Delta)^{\frac{\alpha}{2}}u \|_{L^\infty(Q_{\tau,T})}.
		\end{align}
		By minimizing over $k>0$, we deduce, for any $t \in (\tau,T)$, 
		\begin{equation*} 
			\| (-\Delta)^\alpha u(\cdot,t) \|_{L^\infty(\R^N)} \leq C\|(-\Delta)^{\frac\alpha 2} f \|_{L^\infty(Q_{\tau,T})}^{\frac{1}{2}} \| (-\Delta)^{\frac{\alpha}{2}}u \|_{L^\infty(Q_{\tau,T})}^{\frac{1}{2}},
		\end{equation*}
		which proves \eqref{est:1}. By virtue of Theorem \ref{th:regular-1}, we derive \eqref{est:regular-1-0} as
		\begin{align*}
			\| (-\Delta)^\alpha u \|_{L^\infty(Q_{\tau,T})} &\leq C\| f \|_{L^\infty((\tau,T);C^{\theta}(\R^N))}^{\frac{1}{2}} \| f \|_{L^\infty(Q_{\tau,T})}^{\frac{\alpha-\gamma}{2(2\alpha-\gamma)} }H^{\frac{\alpha}{2(2\alpha-\gamma)}} \leq C\| f \|_{L^\infty((\tau,T);C^{\theta}(\R^N))}^{\frac{1}{2}+\frac{\alpha-\gamma}{2(2\alpha-\gamma)}}H^{\frac{\alpha}{2(2\alpha-\gamma)}}.
		\end{align*}  
		
		(ii) Assume $f \in L^\infty((\tau,T);H^{\alpha,p}(\R^N))$ for $\frac{N}{\alpha}<p<\infty$. By {estimate \eqref{est:mild-1}} and Lemma \ref{lem:interchange} with $\beta=\alpha/2$, $u \in L^\infty((\tau,T);H^{\alpha,p}(\R^N))$. 
		Therefore, by Lemma \ref{lem:interchange} with $\beta=\alpha/2$ and Lemma \ref{lem:interchange-2} (i), we deduce that $u$ satisfies \eqref{formula:Duhamel-23}. 
		Consequently, by using again estimate \eqref{est:fracDS(t-s)f-1}  with $\beta=\alpha/2$, and estimate \eqref{est:Lp-DS} with $\beta=\alpha/2$, $r=p=\infty$, we have, , for any $t \in (\tau,T)$, 
		\begin{align} \nonumber
			&\| (-\Delta)^\alpha u(\cdot,t) \|_{L^\infty(\R^N)} \\ \nonumber
			&\leq C\int_\tau^t e^{-k(t-s)} (t-s)^{-\frac{1}{2}} (\| (-\Delta)^{\frac{\alpha}{2}} f(\cdot,s)\|_{L^p(\R^N)} + k \| (-\Delta)^{\frac{\alpha}{2}} u(\cdot,s)\|_{L^\infty(\R^N)}) ds\\ \nonumber 
			&\leq Ck^{-\frac{1}{2}}\| f \|_{L^\infty((\tau,T);\dot{H}^{\alpha,p}(\R^N))} + Ck^{\frac{1}{2}} \| (-\Delta)^{\frac{\alpha}{2}}u \|_{L^\infty(Q_{\tau,T})}.
		\end{align}
		By minimizing over $k>0$, we deduce
		\begin{equation*} 
			\| (-\Delta)^\alpha u(\cdot,t) \|_{L^\infty(\R^N)} \leq C\| f \|_{L^\infty((\tau,T);\dot{H}^{\alpha,p}(\R^N))}^{\frac{1}{2}} \| (-\Delta)^{\frac{\alpha}{2}}u \|_{L^\infty(Q_{\tau,T})}^{\frac{1}{2}}.
		\end{equation*}
		This and Theorem \ref{th:regular-1} imply
		\begin{equation*} 
			\begin{aligned}
				\| (-\Delta)^\alpha u \|_{L^\infty(Q_{\tau,T})} &\leq  C \| f \|_{L^\infty((\tau,T);\dot{H}^{\alpha,p}(\R^N))}^{\frac{1}{2}} \| f \|_{L^\infty((\tau,T);L^p(\R^N))}^{\frac{(\alpha-\gamma)p}{2(2\alpha p - N -\gamma p)} }H^{\frac{\alpha p - N}{2(2\alpha p - N - \gamma p)}} \\
				&\leq  C \| f \|_{L^\infty((\tau,T);H^{\alpha,p}(\R^N))}^{\frac{3\alpha p - N - 2\gamma p}{2(2\alpha p - N -\gamma p)} } H^{\frac{\alpha p - N}{2(2\alpha p - N - \gamma p)}}.
		\end{aligned} \end{equation*}
		The proof is complete.
	\end{proof}
	
	Let $\dot{H}^\alpha(\R^N)$ be the space defined in \eqref{H^alphadot} (with $p=2$).
	
	\begin{proposition}[Feedback estimate] \label{prop:feedback}
		Let $0 \leq \tau<T$. Assume 
		
		(i) $\bm{w}_\tau=(w_{1,\tau},\ldots,w_{m,\tau}) \in ((L^1 \cap L^\infty \cap \dot{H}^{\alpha})(\R^N))^m$.
		
		(ii) $\bm{F}=(F_1,\ldots,F_m) \in ((L^1 \cap L^\infty)(Q_{\tau,T}))^m$ such that
		\begin{equation} \label{mass_conservation_Fi} \sum_{i=1}^m F_i(x,t)=0 \quad \text{for all } (x,t) \in Q_{\tau,T}. 
		\end{equation}
		(iii) $\bm{\psi}=(\psi_1,\ldots,\psi_m)$ such that $0 \leq \psi_i \in L^1(Q_{\tau,T}) \cap L^\infty((\tau,T);C^\theta(\R^N))$ for some $\theta \in (\alpha,1)$ (hence  $(-\Delta)^{\frac{\alpha}{2}}\psi_i \in L^\infty(Q_{\tau,T})$) and
		\begin{align} \label{est:L1_psi} 
			&\| \psi_i \|_{L^1(Q_{\tau,T})} \leq \Theta_1, \\ \label{est:LxinftyLt1psi}
			&\left\| \int_{\tau}^T \psi_i(x,s)ds \right \|_{L^\infty(\R^N)} \leq \Theta_2, \quad \forall i=1,2,\ldots,m,
		\end{align}
		where $\Theta_1,\Theta_2$ are positive constants depending only on $N,m,\alpha$ and $T-\tau$. Assume $\bm{w}=(w_1,\ldots,w_1)$ with $w_i \geq 0$ is a mild solution of 
		\begin{equation}\label{sys:psi_tauT}
			\begin{cases}
				\pa_t w_i(x,t) + d_i(-\Delta)^{\alpha}w_i(x,t) = \psi_i(x,t) + F_i(x,t), & (x,t) \in Q_{\tau,T}, i=1,\ldots, m,\\
				w_i(x,\tau) =w_{i,\tau}(x), & x\in\R^N, i=1,\ldots, m.
			\end{cases}
		\end{equation}
		Then there holds
		\begin{equation} \label{est:feedback}
			\sum_{i=1}^m\| w_i \|_{L^\infty(Q_{\tau,T})} \leq C  + C\sum_{i=1}^m \left( \| \psi_i \|_{L^\infty(Q_{\tau,T})}^{1-\delta} + \| (-\Delta)^{\frac{\alpha}{2}}\psi_i \|_{L^\infty(Q_{\tau,T})}^{\frac{2}{3}-\delta} 	+ \| F_i \|_{L^\infty(Q_{\tau,T})}^{\frac{1}{2}-\delta} \right), 
		\end{equation}
		where $\delta \in (0,1/4)$ and $C>0$ depend only on $N,m,\alpha$, $T-\tau$, $\Theta_1$, $\Theta_2$, $d_i$, $\| w_{i,\tau}\|_{C^\theta(\R^N)}$, $i=1,\ldots,m$.
	\end{proposition}
	\begin{remark}
		The system \eqref{sys:psi_tauT} can be seen as an auxiliary one of \eqref{sys}, where the feedback estimate in Proposition \ref{prop:feedback} will be applied for $\psi_i = 0$ and $\psi_i = \varphi_\tau u_i$ for a suitable cut-off function in time $\varphi_\tau$, and $F_i(\cdot): = f_i(\cdot,u(\cdot))$ is the nonlinearity. This enables us to avoid the repetition of analysis in the proof of the uniform-in-time boundedness of solutions (as in e.g. \cite{morgan2020boundedness}).
	\end{remark}
	
	To prove Proposition \ref{prop:feedback}, we need to develop some intermediate estimates.
	
	\medskip
	By summing the equations of $w_i$ in \eqref{sys:psi_tauT}, using the condition \eqref{mass_conservation_Fi}, then integrating with respect to the time variable from $\tau$ to $t$, we have
	\begin{equation}\label{h1} 
		\sum_{i=1}^{m}w_i(x,t) + (-\Delta)^{\alpha} \int_\tau^t\sum_{i=1}^md_i w_i(x,s)ds = \sum_{i=1}^{m}\int_{\tau}^t \psi_i(x,s)ds + \sum_{i=1}^{m} w_{i,\tau}(x), \quad \forall (x,t)\in Q_{\tau,T}.
	\end{equation}
	We define the following function, which plays an important role in our subsequent analysis,
	\begin{equation}\label{def_v} 
		v(x,t):=  \int_\tau^t\sum_{i=1}^m d_i w_i(x,s)ds.
	\end{equation}
	From \eqref{h1}, we have the following three equivalent equations involving $v$ as follows
	\begin{equation}\label{h2}
		\sum_{i=1}^{m} w_{i}(x,t) + (-\Delta)^{\alpha}v(x,t) = \sum_{i=1}^{m}\int_{\tau}^t \psi_i(x,s)ds +   \sum_{i=1}^{m}w_{i,\tau}(x), \quad \forall (x,t)\in Q_{\tau,T},
	\end{equation}
	\begin{equation}\label{h3}
		\partial_tv(x,t) + (-\Delta)^{\alpha}v(x,t) = \sum_{i=1}^{m}(d_i-1)w_i(x,t) + \sum_{i=1}^{m}\int_{\tau}^t \psi_i(x,s)ds  +  \sum_{i=1}^m w_{i,\tau}(x),  \, \forall (x,t)\in Q_{\tau,T},
	\end{equation}
	\begin{equation}\label{h4} 
		b(x,t)\partial_t v(x,t) + (-\Delta)^{\alpha}v(x,t) = \sum_{i=1}^{m}\int_{\tau}^t \psi_i(x,s)ds  + \sum_{i=1}^{m}w_{i,\tau}(x),  \quad \forall (x,t)\in Q_{\tau,T},
	\end{equation}
	where the function 
	\begin{equation*} 
		b(x,t):= \frac{\sum_{i=1}^{m} w_i(x,t)}{\sum_{i=1}^{m}d_i w_i(x,t)}
	\end{equation*}
	satisfies the upper and lower bounds
	\begin{equation}\label{bound_b} 
		\frac{1}{\max_{i=1,\ldots,m} d_i} \le b(x,t) \le \frac{1}{\min_{i=1,\ldots,m} d_i}, \quad \forall (x,t)\in Q_{\tau,T}.
	\end{equation}
	
	It is clear from \eqref{h2} that $L^\infty$-norm of $w_i$ can be estimated by certain norms of $v$, which will be shown in the following lemmas.
	\begin{lemma}\label{lem:boundedness_v}
		The function $v$ defined in \eqref{def_v} is bounded, i.e.
		\begin{equation}\label{bound-1}
			\|v\|_{L^\infty(Q_{\tau,T})} \leq C,
		\end{equation}
		where $C$ depends only on $N,m,\alpha$, $T-\tau$, $\Theta_1,\Theta_2$, $d_i$, $\| w_{i,\tau} \|_{(L^1 \cap L^\infty)(\R^N)}$, $i=1,\ldots,m$.
	\end{lemma}
	\begin{proof}
		From \eqref{h1} and the fact that $\int_{\R^N}(-\Delta)^\alpha w_i dx = 0$ for all $1 \leq i \leq m$, we have
		$$
		\sum_{i=1}^m\int_{\R^N} w_i(x,t)dx = \sum_{i=1}^m \int_{\tau}^t \int_{\R^N} \psi_i (x,s)dx ds + \sum_{i=1}^m\int_{\R^N}w_{i,\tau}dx.	
		$$
		Therefore
		\begin{equation} \label{est:v_L1norm}
		\begin{aligned} 
			\|v(\cdot,t)\|_{\LO{1}} 
			&=  \sum_{i=1}^m d_i \int_{\tau}^t \int_{\tau}^s \int_{\R^N} \psi_i (x,\sigma)dx d\sigma ds + (t-\tau)\sum_{i=1}^m d_i \int_{\R^N}w_{i,\tau}dx \\
			&\leq (T-\tau)\sum_{i=1}^m d_i(\| \psi_i\|_{L^1(Q_{\tau,T})} + \|w_{i,\tau}\|_{L^1(\R^N)}) \leq C_1,
		\end{aligned}
		\end{equation}
		where $C_1$ depends on $T-\tau$, $\Theta_1$, $\| w_{i,\tau}\|_{L^1(\R^N)}$, $d_i$, $i=1,\ldots,m$. 
		It follows from \eqref{h2} and the non-negativity of $w_i$ that
		\begin{equation*} 
			(-\Delta)^{\alpha}v(x,t) \leq \sum_{i=1}^{m}\int_{\tau}^t \psi_i(x,s)ds + \sum_{i=1}^m w_{i,\tau}(x), \quad (x,t)\in Q_{\tau,T}.
		\end{equation*}
		By interpolation,
			\begin{equation*}
				\sum_{i=1}^m \Big\|\int_{\tau}^T\psi_i(x,s)ds\Big\|_{\LO{\frac{N}{\alpha}}} \leq m\Theta_1^{\frac{\alpha}{N}}\Theta_2^{\frac{N-\alpha}{N}} =: \Theta_3.
			\end{equation*}
		For $\ell > 1$, by multiplying this inequality by $v^{\ell-1}$ and using the Stroock-Varopoulos inequality in Lemma \ref{SV-ineq} and Young's inequality, we have
		\begin{align*} 
			&\frac{4(\ell-1)}{\ell^2}\norm{(-\Delta)^{\frac{\alpha}{2}}(v(t)^{\frac{\ell}{2}})}_{\LO{2}}^2 \\
			&\leq \int_{\R^N} v(x,t)^{\ell-1} \sum_{i=1}^{m} \int_{\tau}^t \psi_i(x,s)dsdx +  \intR v(x,t)^{\ell-1} \sum_{i=1}^{m} w_{i,\tau}(x)dx \\
			&  {\le\sum_{i=1}^m\|v\|_{\LO{\frac{N(\ell-1)}{N-\alpha}}}^{\ell-1}\Big\|\int_{\tau}^t\psi_i(x,s)ds\Big\|_{\LO{\frac{N}{\alpha}}} + \sum_{i=1}^m\|v\|_{\LO{\frac{N(\ell-1)}{N-\alpha}}}^{\ell-1}\|w_{i,\tau}\|_{\LO{\frac{N}{\alpha}}} } \\
			&{\le \frac{\ell-1}{\ell}\|v\|_{\LO{\frac{N(\ell-1)}{N-\alpha}}}^{\ell}\times\bra{\sum_{i=1}^m\bra{\Big\|\int_{\tau}^t\psi_i(x,s)ds\Big\|_{\LO{\frac{N}{\alpha}}} + \|w_{i,\tau}\|_{\LO{\frac{N}{\alpha}}} }} } \\
			&{\quad + \frac{1}{\ell}\sum_{i=1}^m\bra{\Big\|\int_{\tau}^t\psi_i(x,s)ds\Big\|_{\LO{\frac{N}{\alpha}}} + \|w_{i,\tau}\|_{\LO{\frac{N}{\alpha}}} }}\\
			&{\le \frac{\ell-1}{\ell}(C_2 + \Theta_3)\|v\|_{\LO{\frac{N(\ell-1)}{N-\alpha}}}^{\ell} + \frac{1}{\ell}(C_2+\Theta_3)},
		\end{align*}
		where $C_2=\sum_{i=1}^m \| w_{i,\tau}\|_{L^{\frac{N}{\alpha}}(\R^N)}$, which is independent of  $\ell$. Applying the Sobolev inequality in Lemma \ref{fGN-ineq} yields
		\begin{align*} 
			& \|v(t)\|_{\LO{\frac{N\ell}{N-2\alpha}}}^{\ell}= \norm{v(t)^{\frac{\ell}{2}}}_{\LO{\frac{2N}{N-2\alpha}}}^2 \\ &\leq C_{\text{Sob}}^{2}\|(-\Delta)^{\frac{\alpha}{2}}(v(t)^{\frac{\ell}{2}}) \|_{L^2(\R^N)}^2 \leq  \frac{C_{\text{Sob}}^{2}{(C_2+\Theta_3)}\ell}{4}\|v(t)\|_{{\LO{\frac{N(\ell-1)}{N-\alpha}}}}^{\ell} + \frac{C_{\text{Sob}}^{2}{(C_2+\Theta_3)}\ell}{4(\ell-1)}.
		\end{align*}
		We now define a sequence $\{\varrho_k\}_{k\ge 1}$ such that {$\varrho_1 = 1$} and for all $k\ge 1$,
			\begin{equation} \label{sequence_varrho}
				\varrho_{k+1} = 1 + \frac{N-\alpha}{N-2\alpha}\varrho_k.
			\end{equation}
			Denote $\omega_{k} = \frac{N\varrho_{k}}{N-2\alpha}$. By choosing $\ell = \varrho_{k+1}$ and noting that $\omega_k = \frac{N(\varrho_{k+1}-1)}{N-\alpha}$, we get
			\begin{equation*}
				\|v(t)\|_{\LO{\omega_{k+1}}} \le \bra{C_3(C_2+\Theta_3)\varrho_{k+1}}^{\frac{1}{\varrho_{k+1}}}\bra{\|v(t)\|_{\LO{\omega_k}}^{\varrho_{k+1}} +  \frac{1}{\varrho_{k+1}-1} }^{\frac{1}{\varrho_{k+1}}}
			\end{equation*}
			with $C_3 = C_{\text{Sob}}^2/4$.
		Put 
		$$ M_k(t):=\max\left\{ \|v(t)\|_{{L^{\omega_{k}}(\R^N)}}, {\Big(\frac{1}{\varrho_{k+1}-1}\Big)^{\frac{1}{\varrho_{k+1}}}} \right\}, 
		$$
		then we see that
		$$ \|v(t)\|_{\LO{\omega_{k+1}}} \leq {(C_3(C_2+\Theta_3)\varrho_{k+1})^{\frac{1}{\varrho_{k+1}}}} M_k(t).
		$$
		On the other hand, we have
		\begin{align*}
			\left(\frac{1}{\varrho_{k+2}-1} \right)^{\frac{1}{\varrho_{k+2}}}	\leq \Lambda_k \left( \frac{C_3(C_2+\Theta_3) \varrho_{k+1}}{\varrho_{k+1}-1} \right)^{\frac{1}{\varrho_{k+1}}} \leq \Lambda_k (C_3(C_2+\Theta_3)\varrho_{k+1})^{\frac{1}{\varrho_{k+1}}} M_k(t),
		\end{align*}
		where
		$$ \Lambda_k := \left(\frac{1}{\varrho_{k+2}-1} \right)^{\frac{1}{\varrho_{k+2}}} \left( \frac{\varrho_{k+1}-1}{C_3(C_2+\Theta_3)\varrho_{k+1}} \right)^{\frac{1}{\varrho_{k+1}}}.
		$$
		By choosing  {$C_3$ large enough} if necessary, we have
		\begin{align*}
			\Lambda_k \leq 	\left(\frac{1}{C_3(C_2+\Theta_3)(\varrho_{k+2}-1)} \right)^{\frac{1}{\varrho_{k+2}}} \leq \left(\frac{1}{\varrho_{k+2}} \right)^{\frac{1}{\varrho_{k+2}}} \leq 1, \quad \forall k \in \N.
		\end{align*}
		Combining the preceding estimates  leads to
		$M_{k+1}(t)	\leq	{(2C_3(C_2+\Theta_3)\varrho_{k+1})^{\frac{1}{\varrho_{k+1}}}} M_k(t)$. 
		This yields
		\begin{equation*} 
			\ln M_{k+1}(t) \leq  {\ln(2C_3(C_2+\Theta_3))\frac{1}{\varrho_{k+1}}} + {\frac{\ln \varrho_{k+1}}{\varrho_{k+1}}} + \ln M_k(t), \quad \forall k \geq 1,
		\end{equation*}
		which in turn implies
		$$
		{\ln M_{k+1}(t)} \leq {\ln(2C_3(C_2+\Theta_3)) \sum_{k \geq 1} \frac{1}{\varrho_{k+1}} + \sum_{k \geq 1} \frac{\ln \varrho_{k+1}}{\varrho_{k+1}}} + \ln M_1(t), \quad \forall k \geq 1.
		$$
		From the definition of sequence $\{\varrho_k\}$, the series on the right hand side are convergent. Hence
		$$ \| v(t) \|_{\LO{\omega_{k+1}}} \leq C_4(C_2+ \Theta_3)^{a}\bra{\| v(t) \|_{L^{\frac{N}{N-2\alpha}}(\R^N)}+1}, \quad \forall k \geq 1,
		$$
		where $a=\sum_{k \geq 1}\frac{1}{\varrho_{k+1}}$ (by \eqref{sequence_varrho}, $\frac{2N-3\alpha}{N-\alpha}<a<\frac{N-2\alpha}{\alpha}$) and $C_4$ is independent of $\Theta_3$ and $k$. Letting $k\to \infty$ yields
		\begin{equation}\label{c1}
			\|v(t)\|_{\LO{\infty}}\leq C_4(C_2+ \Theta_3)^{a}\bra{\| v \|_{L^{\frac{N}{N-2\alpha}}(\R^N)}+1}.
		\end{equation}
		By using Lebesgue interpolation inequalities and $\LO{1}$-bound of $v$ in \eqref{est:v_L1norm}, we obtain  
		$$
		\|v(t)\|_{\LO{\frac{N}{N-2\alpha}}} \leq \|v(t)\|_{\LO{1}}^{\frac{N-2\alpha}{N}} \|v(t)\|_{\LO{\infty}}^{\frac{2\alpha}{N}}\leq C_1^{\frac{N-2\alpha}{N}} \|v(t)\|_{\LO{\infty}}^{\frac{2\alpha}{N}}.
		$$  
		Inserting it into \eqref{c1}, we derive
		\begin{equation*}
			\|v(t)\|_{\LO{\infty}}\leq C_4(C_2+ \Theta_3)^{a} \bra{C_1^{\frac{N-2\alpha}{N}} \|v(t)\|_{\LO{\infty}}^{\frac{2\alpha}{N}}+1}.
		\end{equation*}
		This and Young's inequality imply \eqref{bound-1}. The proof is complete.
	\end{proof}	
	
	
	\begin{lemma}\label{lem:Holder_v}
		The function $v$ defined in \eqref{def_v} is H\"older continuous, namely there exists a constant $\widetilde\gamma\in (0,1)$ independent of $v$ such that
		\begin{equation*} 
			|v(x,t) - v(x_1,t_1)| \lesssim |x-x_1|^{\widetilde\gamma} + |t-t_1|^{\widetilde\gamma/2} \quad \forall (x,t), (x_1,t_1)\in Q_{\tau,T}.
		\end{equation*}
	\end{lemma}
	\begin{proof}
		To prove Lemma \ref{lem:Holder_v} we exploit the famous De Giorgi-Nash-Moser theory. Note that the equation of $v$ in \eqref{h4} is of non-divergence form, as we do not have any other regularity of $b$ except the upper and lower bounds \eqref{bound_b}. Therefore, up to our knowledge, H\"older regularity for fractional Laplacian or integral equations, as in e.g. \cite{caffarelli2010drift,caffarelli2011regularity,ding2021local}, is not directly applicable in our case. However, we can adapt the techniques in \cite{caffarelli2011regularity} with some suitable modifications to get our desired result. {To avoid interrupting the train of thought, the proof with key steps is deferred to Section \ref{appendix:Holder_v}}.
	\end{proof}
	\begin{lemma}\label{lem:gradient_v}
		There exists some $\gamma\in (0,\alpha)$ such that 
		\begin{equation} \label{est:Delta_alpha/2}
			\|(-\Delta)^{\frac\alpha 2}v\|_{L^\infty(Q_{\tau,T})} \leq C\bra{1 + \sum_{i=1}^{m}\|w_i\|_{L^\infty(Q_{\tau,T})}^{\frac{\alpha - \gamma}{2\alpha - \gamma}} + \|\psi_i\|_{L^\infty(Q_{\tau,T})}^{\frac{\alpha - \gamma}{2\alpha - \gamma}}},
		\end{equation}
		where $C$ depends only on $N,m,\alpha$, $\gamma$, $T-\tau$, $\Theta_1,\Theta_2$, $d_i$, $\| w_{i,\tau} \|_{(L^1 \cap L^\infty)(\R^N)}$, $i=1,\ldots,m$.
	\end{lemma}
	\begin{proof}
		From the boundedness in Lemma \ref{lem:boundedness_v}, the H\"older continuity in Lemma \ref{lem:Holder_v}, and the continuous embedding $\dot{C}^{\gamma_1}(\R^N)\cap \LO{\infty}\hookrightarrow \dot{C}^{\gamma_2}(\R^N)$ for $\gamma_1>\gamma_2>0$, we have, for any $\gamma\in (0,\widetilde{\gamma})$, that $v\in C_{x,t}^{\gamma,\gamma/2}(Q_{\tau,T})$. We now choose $\gamma\in (0,\min\{\alpha,\widetilde{\gamma}\})$, and then apply Theorem \ref{th:regular-1} with $p=\infty$ to equation \eqref{h3} (note that $v(\cdot,\tau)=0$ in $\R^N$) to obtain the desired estimate \eqref{est:Delta_alpha/2}.
	\end{proof}
	\begin{lemma}\label{bound_Delta_v}
		We have
		\begin{align*} 
			\|(-\Delta)^{\alpha}v\|_{L^\infty(Q_{\tau,T})} 
			&\leq C \left(1 +\sum_{i=1}^m(\|w_i\|_{L^\infty(Q_{\tau,T})}^{\frac{\alpha - \gamma}{2\alpha - \gamma}} + \|\psi_i\|_{L^\infty(Q_{\tau,T})}^{\frac{\alpha - \gamma}{2\alpha - \gamma}}  )\right)^{\frac{1}{2}}\\
			&\qquad \times\left(1+\sum_{i=1}^m(\|(-\Delta)^{\frac \alpha 2}w_i\|_{L^\infty(Q_{\tau,T})} + \| (-\Delta)^{\frac{\alpha}{2}}\psi_i\|_{L^\infty(Q_{\tau,T})})\right)^{\frac{1}{2}},
		\end{align*}
		where $C$ depends only on $N,m,\alpha$, $\theta$, $T-\tau$, $\Theta_1,\Theta_2$, $d_i$, $\| w_{i,\tau} \|_{(L^1 \cap L^\infty)(\R^N)}$, $\| (-\Delta)^{\frac{\alpha}{2}}w_{i,\tau} \|_{L^\infty(\R^N)}$.
	\end{lemma}
	\begin{proof}
		Fix $\tau \in (0,T)$ small. We can assume that $w_{i,\tau}$ is smooth enough  (otherwise we can increase $\tau$) so that $w_{i,\tau} \in C^\theta(\R^N)$ for any $i=1,\ldots,m$, where $\theta \in (\alpha,1)$ is in (iii) of Proposition \ref{prop:feedback}. We infer from the assumption on $\psi_i$ that $\psi_i \in L^\infty((\tau,T);C^\theta(\R^N))$. Consequently, we deduce that $w_i \in L^\infty((\tau,T);C^\theta(\R^N))$. By applying estimate \eqref{est:1} to \eqref{h3} and using \eqref{est:Delta_alpha/2}, we have
		\begin{align*} 
			&\|(-\Delta)^\alpha v\|_{L^\infty(Q_{\tau,T})} \\
			&\leq C\|(-\Delta)^{\frac\alpha 2}v\|_{L^\infty(Q_{\tau,T})}^{\frac{1}{2}}\left\| \sum_{i=1}^{m}(-\Delta)^{\frac{\alpha}{2}}(w_{i,\tau} + w_i + \int_{\tau}^t \psi_i(x,s)ds)\right\|_{L^\infty(Q_{\tau,T})}^{\frac{1}{2}}\\
			&\leq C \left(\sum_{i=1}^{m}\|w_{i,\tau}\|_{\LO{\infty}}^{\frac{\alpha - \gamma}{2\alpha - \gamma}} +  \sum_{i=1}^{m}\Big(\|w_i\|_{L^\infty(Q_{\tau,T})}^{\frac{\alpha - \gamma}{2\alpha - \gamma}} + \|\psi_i\|_{L^\infty(Q_{\tau,T})}^{\frac{\alpha - \gamma}{2\alpha - \gamma}}\Big) \right)^{\frac{1}{2}} \\
			&\times \left( \sum_{i=1}^N (\| (-\Delta)^{\frac{\alpha}{2}}w_{i,\tau} \|_{L^\infty(\R^N)} + \| (-\Delta)^{\frac{\alpha}{2}} w_i \|_{L^\infty(Q_{\tau,T})} + (T-\tau)\| (-\Delta)^{\frac{\alpha}{2}}\psi_i\|_{L^\infty(Q_{\tau,T})}   ) \right)^{\frac{1}{2}},
		\end{align*}
		which implies the desired estimate.
	\end{proof}
	\begin{lemma}\label{lem:gradient_ui}
		For each $i=1,\ldots, m$, it holds
		\begin{equation*} 
			\|(-\Delta)^{\frac \alpha 2}w_i\|_{L^\infty(Q_{\tau,T})}\leq C + C\|w_i\|_{L^\infty(Q_{\tau,T})}^{\frac{1}{2}}\Big(\|\psi_i\|_{L^\infty(Q_{\tau,T})}^{\frac{1}{2}} +\|F_i\|_{L^\infty(Q_{\tau,T})}^{\frac{1}{2}}\Big),
		\end{equation*}
		where $C$ depends only on $N,m,\alpha$, $\theta$, $T-\tau$, $\Theta_1,\Theta_2$, $d_i$, $\| w_{i,\tau} \|_{C^\theta(\R^N)}$, $i=1,\ldots,m$.
	\end{lemma}
	\begin{proof}
		As in the proof of Lemma \ref{bound_Delta_v},  we assume that $w_{i,\tau} \in C^{\theta}(\R^N))$. By applying Theorem \ref{th:regular-1} to the equation of $w_i$ in \eqref{sys:psi_tauT} with $\gamma = 0$, we have
		\begin{align*} 
			\|(-\Delta)^{\frac \alpha 2}w_i\|_{L^\infty(Q_{\tau,T})}&\leq C\left(\| w_{i,\tau}\|_{C^\theta(\R^N)} + \|w_i\|_{L^\infty(Q_{\tau,T})}^{\frac{1}{2}}\|\psi_i + F_i\|_{L^\infty(Q_{\tau,T})}^{\frac{1}{2}}\right)
		\end{align*}
		which implies the desired result.
	\end{proof}
	
	\medskip
	We are now ready to give the
	\begin{proof}[Proof of Proposition \ref{prop:feedback}]
		By using \eqref{h2} and applying Lemmas \ref{bound_Delta_v}, \ref{lem:gradient_ui} successively and Young's inequality, we get
		\begin{align*}
			&\sum_{i=1}^m\|w_i\|_{L^\infty(Q_{\tau,T})} \\ 
			&\leq  \sum_{i=1}^m\|w_{i,\tau}\|_{L^\infty(Q_{\tau,T})} + \sum_{i=1}^m \left\| \int_{\tau}^t \psi_i(\cdot,s)ds \right\|_{L^\infty(Q_{\tau,T})} + \|(-\Delta)^\alpha v\|_{L^\infty(Q_{\tau,T})} \\
			&\leq C+ C\left(1 +\sum_{i=1}^m(\|w_i\|_{L^\infty(Q_{\tau,T})}^{\frac{\alpha - \gamma}{2\alpha - \gamma}} + \|\psi_i\|_{L^\infty(Q_{\tau,T})}^{\frac{\alpha - \gamma}{2\alpha - \gamma}}  )\right)^{\frac{1}{2}} \\
			&\quad \times 
			\left(1 + \sum_{i=1}^m(\|(-\Delta)^{\frac \alpha 2}w_i\|_{L^\infty(Q_{\tau,T})} + \| (-\Delta)^{\frac{\alpha}{2}}\psi_i\|_{L^\infty(Q_{\tau,T})})\right)^{\frac{1}{2}}\\
			&\leq  C +C\left(1 + \sum_{i=1}^m\Big(\|w_i\|_{L^\infty(Q_{\tau,T})}^{\frac{\alpha - \gamma}{2(2\alpha - \gamma)}} + \|\psi_i\|_{L^\infty(Q_{\tau,T})}^{\frac{\alpha - \gamma}{2(2\alpha - \gamma)}}  \Big)\right) \\
			&\quad \times \left(1 +\sum_{i=1}^{m}\bra{ \|w_i\|_{L^\infty(Q_{\tau,T})}^{\frac{1}{2}}\left(\|\psi_i\|_{L^\infty(Q_{\tau,T})}^{\frac{1}{2}} +\|F_i\|_{L^\infty(Q_{\tau,T})}^{\frac{1}{2}} } + \| (-\Delta)^{\frac{\alpha}{2}}\psi_i\|_{L^\infty(Q_{\tau,T})} \right) \right)^{\frac{1}{2}}
			\\ 
			&\leq C+ C\sum_{i=1}^{m}\left(\|\psi_i\|_{L^\infty(Q_{\tau,T})}^{\frac{2(3\alpha-2\gamma)}{3(2\alpha-\gamma)}}  + \| (-\Delta)^{\frac{\alpha}{2}}\psi_i\|_{L^\infty(Q_{\tau,T})}^{\frac{2\alpha-\gamma}{3\alpha-\gamma}} + \|F_i\|_{L^\infty(Q_{\tau,T})}^{\frac{3\alpha-2\gamma}{3(2\alpha-\gamma)}} \right),
		\end{align*}	
		where $C$ depends only on $N$, $\alpha$, $\theta$, $m$, $T-\tau$, $\Theta_1$, $\Theta_2$, $d_i$, $\| w_{i,\tau}\|_{C^\theta(\R^N)}$.
		This implies \eqref{est:feedback}.
	\end{proof}
	
	\subsection{Global existence and boundedness}\label{subsec:interp}
	\begin{lemma}[Equivalence of mass dissipation and mass conservation]\label{lem:equivalence}
		Assume \eqref{F}, \eqref{P} and \eqref{M}. Then the system \eqref{sys} has a unique global bounded mild solution under assumption \eqref{M} if and only if the same holds true under the assumption
		\begin{equation}\label{mass_conservation} 
			\sum_{i=1}^mf_i(x,t,\ub)=0, \quad \forall x \in \R^N,\; t>0,\; \ub \in \mathbb R_+^m.
		\end{equation}
	\end{lemma}
	\begin{proof}
		The implication from \eqref{mass_conservation} to \eqref{M} is trivial. It remains to show that one can replace \eqref{M} by \eqref{mass_conservation}. Consider now an additional equation for $u_{m+1}$ as
		\begin{equation*} 
			\begin{cases}
				\partial_t u_{m+1}(x,t) + (-\Delta)^{\alpha}u_{m+1}(x,t) = -\sum_{i=1}^m f_i(x,t,\ub),\\
				u_{m+1}(x,0) = 0. 
			\end{cases}
		\end{equation*}
		Now, we consider a new system for $\bm{w} = (w_1,\ldots, w_{m+1})$ where $w_i = u_i$, $i=1,\ldots, m+1$, and define
		\begin{equation*} 
			g_i(x,t,\bm{w}):= f_i(x,t,\ub) \text{ for } i=1,\ldots, m,  \quad g_{m+1}(x,t,\bm{w}):= -\sum_{i=1}^mf_i(x,t,\ub) = \sum_{i=1}^m g_i(x,t,\bm{w})\geq 0.
		\end{equation*}
		Then we have
		\begin{equation}\label{sys_w} 
			\begin{cases}
				\partial_t w_i(x,t) + d_i(-\Delta)^{\alpha}w_{i}(x,t) = g_i(x,t,\bm{w}), &i=1,\ldots, m+1,\\
				w_i(x,0) = w_{i,0}(x), &x\in\R^N,
			\end{cases}
		\end{equation}
		where $d_{m+1} = 1$ and $u_{m+1,0} \equiv 0$. It is straightforward to check that  system \eqref{sys_w} satisfies \eqref{P} and \eqref{M}, and \eqref{sys_w} has a unique global bounded solution if and only if the same is true for system \eqref{sys}. The proof is complete.
	\end{proof}
	Thanks to Lemma \ref{lem:equivalence}, \textit{we assume in this section the mass conservation \eqref{mass_conservation}.}
	
	We also need the following elementary lemma, whose proof is straightforward, so we omit it.
	\begin{lemma} \label{lem:sequence}
		Let $\{a_k\}$ be a sequence of nonnegative real numbers. Put $\mathcal{K}:=\{k \in \N: a_k \leq a_{k+1}\}$. If there exists a positive constant $A_0$ such that 
		$a_k \leq A_0$ for all $k \in \mathcal{K}$	
		then 
		$$a_k \leq \max\{a_0,A_0\} \quad \forall k \in \N. $$ 
	\end{lemma}
	
	We are now ready to prove Theorem \ref{thm:quadratic}. 
		\begin{proof}[\textbf{Proof of Theorem \ref{thm:quadratic}}]
			Let $T_*$ be the maximal time for the existence of the mild solution $\ub$  and $0<\tau<T < T_*$. Thanks to the regularity of the solution in Lemma \ref{lem:reg-sys}, we can assume that the initial data is smooth $u_{i,\tau}\in C^{\beta}(\R^N)\cap \LO{1}\cap \LO{\infty}$ for all $\beta\in (0,\min\{1,2\alpha\})$. We apply Proposition \ref{prop:feedback} with $w_i = u_i$, $F_i(x,t) = f_i(x,t,\ub)$, $\psi_i= 0$ (which implies $\Theta_1 = \Theta_2 = 0$), to obtain
			\begin{equation*}
				\sum_{i=1}^m\|u_i\|_{L^\infty(Q_{\tau,T})} \le C + C\sum_{i=1}^m\|f_i(\cdot,\cdot,\ub)\|_{L^\infty(Q_{\tau,T})}^{\frac{1}{2} - \delta},
			\end{equation*}
			where	$C>0$ {depends} only on $N,m,\alpha$, $T-\tau$,  $d_i$, $\| u_{i,\tau}\|_{C^\beta(\R^N)}$, $i=1,\ldots,m$.
			Due to the quadratic growth $|f_i(x,t,\ub)| \le C(1+|\ub|^2)$, we get
			\begin{equation*}
				\sum_{i=1}^m\|u_i\|_{L^\infty(Q_{\tau,T})} \le C + C\sum_{i=1}^m\bra{1+\sum_{i=1}^m\|u_i\|_{L^\infty(Q_{\tau,T})}^{1-2\delta}}.	
			\end{equation*}
			By Young's inequality we get 
			$\sum_{i=1}^m\|u_i\|_{L^\infty(Q_{\tau,T})} \le C$,  where $C$ depends only on $N,m,\alpha$, $T-\tau$, $d_i$, $\| u_{i,\tau}\|_{C^\beta(\R^N)}$, $i=1,\ldots,m$. 
			This concludes the global existence thanks to the blow-up criteria in Proposition \ref{prop:local-existence}. The uniqueness of the mild solution follows from the local Lipschitz continuity of the nonlinearities and Gronwall's inequality.
			
			We now turn to the uniform-in-time bound. First, by summing equations of $u_i$ in \eqref{sys} and using \eqref{mass_conservation} and integrating over $\R^N$, we obtain
				$\partial_t \left(\sum_{i=1}^m \int_{\R^N}u_i(x,t)dx \right)= 0$,
				which implies 
				\begin{equation} \label{est:ui_L1}
					\sup_{1 \leq i \leq m}\| u_i(t) \|_{L^1(\R^N)} \leq C, \quad \forall t \geq 0.
				\end{equation}
			Fix $\tau \ge 0$ arbitrary and let $\varphi_\tau: \R \to [0,1]$ be a smooth function such that $\varphi_\tau|_{(-\infty,\tau]} = 0$, $\varphi_{\tau}|_{[\tau+1,\infty)} = 1$, $0\le \varphi_\tau' \le 2$ for $s\in [\tau,\tau+1]$. Denote by $\zeta_i = \varphi_\tau u_i$, it follows that
			\begin{equation}\label{eq_zeta}
				\begin{cases}
					\partial_t \zeta_i + d_i(-\Delta)^{\alpha}\zeta_i = \varphi_\tau' u_i + \varphi_\tau f_i(x,t,\ub), &(x,t)\in Q_{\tau,\tau+2},\\
					\zeta_i(x,\tau) = 0, &x\in\R^N.
				\end{cases}
			\end{equation}		
			We aim to apply Proposition \ref{prop:feedback} with $\psi_i = \varphi_\tau' u_i, F_i(x,t) = \varphi_\tau f_i(x,t,\ub(x,t))$ and $T = \tau+2$. {We see that estimate \eqref{est:L1_psi} is satisfied due to \eqref{est:ui_L1}}, hence it remains to verify  estimate \eqref{est:LxinftyLt1psi}. By summing up the equations of $\zeta_i$, integrating on $(\tau,t)$, $t\in (\tau,\tau+2)$, we obtain
			\begin{equation}\label{v1}
				(-\Delta)^{\alpha}v(x,t) \leq  \sum_{i=1}^m\int_\tau^t \varphi_\tau'(s) u_i(x,s)ds \quad \text{ with } \quad v(x,t)= \sum_{i=1}^m\int_{\tau}^td_i\varphi_\tau u_i(x,s)ds.
			\end{equation}
Multiplying this inequality by $v^{\ell-1}$ for $\ell > 1$, and using similar arguments as in Lemma \ref{lem:boundedness_v}, we get
			\begin{align*}
				\|v(t)\|_{\LO{\frac{N\ell}{N-2\alpha}}}^{\ell} &\le C_{\text{Sob}}^2\|(-\Delta)^{\frac{\alpha}{2}}(v(t)^{\frac{\ell}{2}})\|_{\LO{2}}^2\\
				& \le C_{\text{Sob}}^2\frac{\ell^2}{4(\ell-1)}\int_{\R^N}\bra{ \sum_{i=1}^m\int_\tau^t \varphi_\tau'(s) u_i(x,s)ds}v(x,t)^{\ell-1}dx\\
				&\le C_1\int_{\R^N}\bra{\int_{\tau}^t \sum_{i=1}^m u_i(x,s)ds}^{\ell}ds
			\end{align*}
			where $C_1$ depends only on $C_{\text{Sob}}$, $\ell$, $d_i$, $i=1,\ldots,m$, due the boundedness of $\varphi_\tau$ and $\varphi_\tau'$. Taking $t=\tau+2$, using the definition of $v$ in \eqref{v1} and the definition of $\varphi_\tau$, it follows that
			\begin{equation*}
				\Bnorm{\int_{\tau+1}^{\tau+2}\sum_{i=1}^m u_i(\cdot,s)ds}_{\LO{\frac{N\ell}{N-2\alpha}}}^{\ell} \leq C_2\Bnorm{\int_{\tau}^{\tau+2} \sum_{i=1}^m u_i(\cdot,s)ds}_{\LO{\ell}}^{\ell},
			\end{equation*}
where $C_2$ depends only $C_1$ and $d_i$. By the interpolation $\|f\|_{\LO{\ell}} \le \|f\|_{\LO{\frac{N\ell}{N-2\alpha}}}^{\eta}\|f\|_{\LO{1}}^{1-\eta}$ for some $\eta \in (0,1)$ depending only on $N,\alpha,\ell$, and taking into account \eqref{est:ui_L1}, we obtain
			\begin{equation}\label{v2}
				\Bnorm{\int_{\tau+1}^{\tau+2}\sum_{i=1}^m u_i(\cdot,s)ds}_{\LO{\frac{N\ell}{N-2\alpha}}}^{\ell} \le C_3\Bnorm{\int_{\tau}^{\tau+2} \sum_{i=1}^m u_i(\cdot,s)ds}_{\LO{\frac{N\ell}{N-2\alpha}}}^{\eta\ell},
			\end{equation}
		 	for some $C_3$ independent of $\tau$. For all $\tau\in \mathbb N$ such that $$\Bnorm{\int_{\tau}^{\tau+1} \sum_{i=1}^m u_i(\cdot,s)ds}_{\LO{\frac{N\ell}{N-2\alpha}}}^{\ell}\le 	\Bnorm{\int_{\tau+1}^{\tau+2}\sum_{i=1}^mu_i(\cdot,s)}_{\LO{\frac{N\ell}{N-2\alpha}}}^{\ell},$$ 
		 	it follows from \eqref{v2} and Young's inequality that
		 	\begin{equation*}
		 			\Bnorm{\int_{\tau+1}^{\tau+2} \sum_{i=1}^m u_i(\cdot,s)ds}_{\LO{\frac{N\ell}{N-2\alpha}}}^{\ell} \le C_4,
		 	\end{equation*}
			where $C_4$ is independent of $\tau$. Thanks to the elementary Lemma \ref{lem:sequence}, we have for any $\ell > 1$ there is a constant $C_5$ (which might depend on $\ell$) such that
			\begin{equation*}
				\sup_{\tau\in \mathbb N}	\Bnorm{\int_{\tau+1}^{\tau+2} \sum_{i=1}^m u_i(\cdot,s)ds}_{\LO{\frac{N\ell}{N-2\alpha}}}^{\ell} \le C_5 < +\infty.
			\end{equation*}
By choosing $\ell=\frac{N-2\alpha}{\alpha}$, using $\psi_i=\varphi_{\tau}' u_i$ and $0 \leq \varphi' \leq 2$, we obtain
			\begin{align*}
				\Bnorm{\int_\tau^{\tau+2}\psi_i(\cdot,s)ds}_{\LO{\frac{N}{\alpha}}}^{\frac{N}{\alpha}} &\le 2^{\frac{N}{\alpha}}\Bnorm{\int_\tau^{\tau+1}u_i(\cdot,s)ds}_{\LO{\frac{N}{\alpha}}}^{\frac{N}{\alpha}} \leq 2^{\frac{N}{\alpha}} C_5^{\frac{N}{N-2\alpha}}.
			\end{align*}
At this point, we can repeat the arguments in Lemma \ref{lem:boundedness_v} to obtain $\sup_{\tau\in \mathbb N}\|v\|_{L^{\infty}(Q_{\tau,\tau+2})} \le C_6$, which in turn leads to, due to the definition of $v$ in \eqref{v1} and $\psi_i = \varphi_\tau' u_i$,
			\begin{equation*}
				\sup_{\tau\in \mathbb N}\Bnorm{\int_{\tau}^{\tau+2}\psi_i(x,s)ds}_{\LO{\infty}} \le C_7.
			\end{equation*}

			We can now apply Proposition \ref{prop:feedback} to the equations of $\zeta_i$ in \eqref{eq_zeta} with $\psi_i=\varphi_{\tau}'u_i$, $F_i(x,t)=\varphi_{\tau}f_i(x,t,\ub(x,t))$ and $T=\tau+2$ to get
			\begin{equation} \label{est:zetai_Q2} \begin{aligned}
					&\sum_{i=1}^m\| \zeta_i \|_{L^\infty(Q_{\tau,\tau+2})} \\
					&\leq C  + C\sum_{i=1}^m \left( \| u_i \|_{L^\infty(Q_{\tau,\tau+2})}^{1-\delta} + \| (-\Delta)^{\frac{\alpha}{2}}(\varphi_{\tau}'u_i) \|_{L^\infty(Q_{\tau,\tau+2})}^{\frac{2}{3}-\delta} 	+ \| \varphi_{\tau}f_i(x,t,\ub(x,t)) \|_{L^\infty(Q_{\tau,\tau+2})}^{\frac{1}{2}-\delta} \right),
				\end{aligned}
			\end{equation}	
			where $\delta \in (0,1/4)$ and $C>0$ depends only on $N,m,\alpha$, $d_i$, $i=1,\ldots,m$.	
			
			By Young's inequality and the fact that $\varphi_{\tau}=1$ on $[\tau+1,\tau+2]$, we obtain 
			\begin{equation} \label{est:ui_Q2} \begin{aligned}
					\sum_{i=1}^m \| u_i \|_{L^\infty(Q_{\tau,\tau+2})}^{1-\delta} &\leq C	\sum_{i=1}^m \left(\| u_i \|_{L^\infty(Q_{\tau,\tau+1})}^{1-\delta} + \| u_i \|_{L^\infty(Q_{\tau+1,\tau+2})}^{1-\delta} \right) \\
					&\leq C\left(1+ \sum_{i=1}^m\| u_i \|_{L^\infty(Q_{\tau,\tau+1})}^{1-\delta}\right)  + \frac{1}{2} \sum_{i=1}^m \| \zeta_i \|_{L^\infty(Q_{\tau+1,\tau+2})}. 
				\end{aligned}
			\end{equation}
			
			Next, we infer from \eqref{sys} that 
			\begin{equation}\label{sys:tau'u}
				\begin{cases}
					\partial_t (\varphi_{\tau}' u_i) + d_i(-\Delta)^{\alpha}(\varphi_{\tau}'u_i) = \varphi_{\tau}''u_i  + \varphi_{\tau}' f_i(x,t,\ub), &\text{in }  Q_{\tau,\tau+2}, i=1,\ldots, m,\\
					\varphi_{\tau}' u_i(x,\tau) =0, & x\in\R^N, i=1,\ldots, m.
				\end{cases}
			\end{equation}
			We deduce from the quadratic condition \eqref{qgr} that
			\begin{equation} \label{est:fi_Q2}
				\| \varphi_{\tau}' f_i(x,t,\ub) \|_{L^\infty(Q_{\tau,\tau+2})} \leq C(1+\sum_{j=1}^m \| u_j \|_{L^\infty(Q_{\tau,\tau+1})}^2),	
			\end{equation}
			which implies that
			$\| \varphi_{\tau}' f_i(x,t,\ub) \|_{L^\infty(Q_{\tau,\tau+2})}^{\frac{1}{2}-\delta} \leq C(1+\sum_{j=1}^m \| u_j \|_{L^\infty(Q_{\tau,\tau+1})}^{1-2\delta}).
			$
Applying Theorem \ref{th:regular-1} to \eqref{sys:tau'u} and using \eqref{est:fi_Q2} yield
			\begin{align*}
				\| (-\Delta)^{\frac{\alpha}{2}}(\varphi_{\tau}'u_i) \|_{L^\infty(Q_{\tau,\tau+2})} &\leq C \| \varphi_{\tau}'u_i \|_{L^\infty(Q_{\tau,\tau+1})}^{\frac{1}{2}} \| \varphi_{\tau}''u_i  + \varphi_{\tau}' f_i(x,t,\ub) \|_{L^\infty(Q_{\tau,\tau+1})}^{\frac{1}{2}} \\
				&\leq C(1 + \sum_{j=1}^m \| u_j \|_{L^\infty(Q_{\tau,\tau+1})}^{\frac{3}{2}}).
			\end{align*}
			This, together with Young's inequality and \eqref{est:ui_Q2}, leads to 
			\begin{equation} \label{est:D^alpha/2ui_Q2}
				\begin{aligned}
					\| (-\Delta)^{\frac{\alpha}{2}}(\varphi_{\tau}'u_i) \|_{L^\infty(Q_{\tau,\tau+2})}^{\frac{2}{3}-\delta} &\leq C\left(1+\sum_{j=1}^m \| u_j \|_{L^\infty(Q_{\tau,\tau+1})}^{1-\frac{3\delta}{2}}\right)	\leq C\left(1+ \sum_{i=1}^m\| u_i \|_{L^\infty(Q_{\tau,\tau+1})}^{1-\delta}\right).	
				\end{aligned}
			\end{equation}
			Combining \eqref{est:zetai_Q2}, \eqref{est:ui_Q2}, \eqref{est:fi_Q2} and \eqref{est:D^alpha/2ui_Q2} yields
			\begin{equation*}  \begin{aligned}
					&\sum_{i=1}^m\| \zeta_i \|_{L^\infty(Q_{\tau,\tau+2})} \\
					&\leq C  + C\sum_{i=1}^m \left( \| u_i \|_{L^\infty(Q_{\tau,\tau+2})}^{1-\delta} + \| (-\Delta)^{\frac{\alpha}{2}}(\varphi_{\tau}'u_i) \|_{L^\infty(Q_{\tau,\tau+2})}^{\frac{2}{3}-\delta} 	+ \| \varphi_{\tau}f_i(x,t,\ub(x,t)) \|_{L^\infty(Q_{\tau,\tau+2})}^{\frac{1}{2}-\delta} \right) \\
					&\leq C+C\sum_{i=1}^m \| u_i \|_{L^\infty(Q_{\tau,\tau+1})}^{1-\delta} + \frac{1}{2} \sum_{i=1}^m \| \zeta_i \|_{L^\infty(Q_{\tau+1,\tau+2})}. 
				\end{aligned}
			\end{equation*}	
			This implies
			\begin{equation} \label{est:sum_ui_1} \sum_{i=1}^m \| u_i \|_{L^\infty(Q_{\tau+1,\tau+2})} \leq C\Big(1 + \sum_{i=1}^m \| u_i \|_{L^\infty(Q_{\tau,\tau+1})}^{1-\delta} \Big).
			\end{equation}
			Suppose $\tau \in \N$ and consider $\tau$ such that 
			$ \sum_{i=1}^m \| u_i \|_{L^\infty(Q_{\tau,\tau+1})} \leq \sum_{i=1}^m \| u_i \|_{L^\infty(Q_{\tau+1,\tau+2})}.
			$
			By \eqref{est:sum_ui_1}, 
			$$ \sum_{i=1}^m \| u_i \|_{L^\infty(Q_{\tau+1,\tau+2})} \leq C\Big(1 + \sum_{i=1}^m \| u_i \|_{L^\infty(Q_{\tau+1,\tau+2})}^{1-\delta} \Big),
			$$
			hence by Young's inequality,
			\begin{equation} \label{est:sum_ui_2}
				\sum_{i=1}^m \| u_i \|_{L^\infty(Q_{\tau+1,\tau+2})} \leq C
			\end{equation}
			where $C$ is independent of $\tau$. By using Lemma \ref{lem:sequence}, we derive that \eqref{est:sum_ui_2} also holds true for any $\tau \in \N$ with another constant $C>0$ independent of $\tau$, which implies the a uniform bound for $u_i$, $i=1,\ldots,m$. The proof is complete. 
	\end{proof}

	\section{Systems with intermediate sum conditions}\label{sec:isc}
	
	\subsection{$L^p-L^q$-regularization of the fractional Laplacian}\label{subsec:reg_frac_diff_LpLq}

	\begin{lemma}[$L^p$-maximal regularity]\label{Lp-regularity}
		Let $\mu>0$, $p\in (1,\infty)$, $0 \leq \tau < T$, $f\in L^p(Q_{\tau,T})$ and $u$ be the mild solution to
		\begin{equation}\label{e1}
			\begin{cases}
				\partial _t u + \mu (-\Delta)^{\alpha}u = f \quad &\text{in } Q_{\tau,T},\\
				u(x,\tau) = 0 \quad &\text{in } \R^N.
			\end{cases}
		\end{equation}
		There is a smallest constant $C_{p,N}$ depending only on $p$ and $N$ such that
		\begin{equation}\label{e3}
			\|(-\Delta)^{\alpha}u\|_{L^p(Q_{\tau,T})}\le \frac{C_{p,N}}{\mu}\|f\|_{L^p(Q_{\tau,T})}.
		\end{equation}
		In particular, $C_{2,N}\le 1$.
	\end{lemma}
	\begin{proof}
		For $\mu=1$, the estimate \eqref{e3} is a consequence of \cite[Theorem 1]{lamberton1987} since the fractional Laplacian $(-\Delta)^{\alpha}$ satisfies the assumptions (H1) and (H2) therein. For $0< \mu \ne 1$, we make the change of variable $\wt{u}(x,t) = u(x, t/\mu)$ and $\wt{f}(x,t) = f(x,t/\mu)$ to get
		\begin{equation*}
			\begin{cases}
				\partial _t \tilde u +  (-\Delta)^{\alpha}\tilde u = \frac{1}{\mu} f \quad &\text{in } Q_{\mu \tau,\mu T},\\
				\tilde u(x,\mu \tau) = 0 \quad &\text{in } \R^N.
			\end{cases}
		\end{equation*}
		Applying the case $\mu=1$ gives
		\begin{equation*}
			\|(-\Delta)^{\alpha}\wt{u}\|_{L^p(Q_{\mu \tau,\mu T})} \leq \frac{C_{p,N}}{\mu}\|\wt{f}\|_{L^p(Q_{\mu \tau,\mu T})}.
		\end{equation*}
		Changing back to the original variables we obtain the desired estimate \eqref{e3}. To show $C_{2,N}\le 1$, we multiply \eqref{e1} by $(-\Delta)^\alpha u$ to obtain
		\begin{equation*} 
			\begin{aligned}
				\|(-\Delta)^{\frac{\alpha}{2}}u(T)\|_{\LO{2}}^2 + \mu\|(-\Delta)^\alpha u\|_{\LQ{2}}^2 &\leq \|f\|_{\LQ{2}}\|(-\Delta)^{\alpha}u\|_{\LQ{2}}\\ &\leq \frac{1}{2\mu}\|f\|_{\LQ{2}}^2 + \frac{\mu}{2}\|(-\Delta)^\alpha u\|_{\LQ{2}}^2.
			\end{aligned}
		\end{equation*}
		This implies $\|(-\Delta)^{\alpha}u\|_{\LQ{2}} \leq \frac{1}{\mu}\|f\|_{\LQ{2}}$, 
		which gives the desired estimate.
	\end{proof}
	
	The following Stroock-Varopoulos inequality (see e.g. \cite[Lemma 7.4]{biler2010nonlinear}) is needed in the proof of the heat regularization.
	
	\begin{lemma}[Stroock-Varopoulos inequality] \label{SV-ineq}
		Assume $0 \leq \alpha \leq 1$ and $\ell>1$. Then we have
		$$ \int_{\R^N} |v|^{\ell-2}v \, (-\Delta)^\alpha v \, dx \geq \frac{4(\ell-1)}{\ell^2} \int_{\R^N} \left| (-\Delta)^{\frac{\alpha}{2}} (|v|^{\frac{\ell}{2} }) \right|^2 dx 
		$$	
		for any $v \in L^\ell(\R^N)$ such that $(-\Delta)^\alpha v \in L^\ell(\R^N)$. 
	\end{lemma}
	
	Next we recall the fractional Gagliardo-Nirenberg inequality (see e.g. \cite[Proposition 3.1]{FLS2016}). Put
	$$ 2_\alpha^*:= 
	\left\{   \begin{aligned} 
		&\frac{2N}{N-2\alpha}\quad &\text{if } \alpha < \frac{N}{2}, \\
		&+\infty &\text{if } \alpha \geq \frac{N}{2}. 
	\end{aligned} \right.
	$$ 
	\begin{lemma}[Fractional Gagliardo-Nirenberg inequality] \label{fGN-ineq}
		For any $2<q<2_\alpha^*$, we have
		\begin{equation} \label{frac-GN}
			\| v \|_{L^q(\R^N)}	\leq C_{\mathrm{GN}} \| v \|_{L^2(\R^N)}^{\theta} \| (-\Delta)^{\frac{\alpha}{2}}v\|_{L^2(\R^N)}^{1-\theta} \quad \forall v \in H^\alpha(\R^N),
		\end{equation}
		where $C_{\mathrm{GN}}$ is the best constant in the fractional Gagliardo-Nirenberg inequality and
		$$ \theta = \frac{2\alpha q - N(q-2)}{2\alpha q} \in (0,1).
		$$
		In case $\theta=0$, we have the fractional Sobolev inequality
		\begin{equation*} 
			\|v\|_{\LO{\frac{2N}{N-2\alpha}}} \le C_{\mathrm{Sob}}\|(-\Delta)^{\frac{\alpha}{2}}v\|_{\LO{2}} \quad \forall v \in \dot{H}^{\alpha}(\R^N).
		\end{equation*}
	\end{lemma}

	\begin{lemma}[Heat regularization]\label{lem:heat_unsigned}
		Assume $\mu>0$, $p\in [1,\infty]$, $0 \leq \tau < T$ and $f\in L^p(Q_{\tau,T})$. Let $u$ be the solution to 
		\begin{equation}\label{e1_1}
			\begin{cases}
				\partial_t u + \mu(-\Delta)^\alpha u = f, &(x,t)\in Q_{\tau,T},\\
				u(x,\tau) = u_\tau(x), &x\in\R^N.
			\end{cases}
		\end{equation}
		Then it holds
		\begin{equation} \label{est:Lp-Lq}
			\|u\|_{L^q(Q_{\tau,T})}\le C(T-\tau,\mu)\bra{\|u_\tau\|_{L^p(\R^N)\cap L^q(\R^N)}+\|f\|_{L^p(Q_{\tau,T})}}
		\end{equation}
		where $p \le q \le \widehat{p}$ with 
		\begin{equation} \label{est:q} 
			\left\{
			\begin{aligned}
				&\widehat{p} < \frac{N+2\alpha}{N} \text{ arbitrary } && \text{ if } p = 1,\\
				&\widehat{p} =   \frac{(N+2\alpha)p}{N+2\alpha - 2p\alpha} &&\text{ if } 1 < p < \frac{N+2\alpha}{2\alpha},\\
				&\widehat{p}  < +\infty \text{ arbitrary } && \text{ if } p = \frac{N+2\alpha}{2\alpha},\\
				&\widehat{p} = +\infty  \text{ arbitrary } &&\text{ if } p > \frac{N+2\alpha}{2\alpha}.
			\end{aligned}
			\right.
		\end{equation}
	\end{lemma}
	\begin{proof}	
		
		%
		
		By the standard density argument, together with the existence result for classical solutions \cite[Lemma 2.1]{chen2016fractional} and the regularity result \cite[Theorem B.1]{HY2023}, it is sufficient to prove \eqref{est:Lp-Lq} for $u$ such that $u \in C_{x,t}^{2\alpha +\sigma,1+\sigma}(Q_{\tau,T})$ for some $\sigma >0$ and $u(t) \in H^{2\alpha,q}(\R^N)$ for any $t \in (\tau,T)$. \medskip
		
		\noindent \textbf{Case 1:} $p=1$. For $k>0$, we define the truncation function $T_k(z):= \max\{-k,\min\{k,z\}\}$. Multiplying \eqref{e1_1} by $T_k(u)$ we have
		\begin{equation*}
			\frac{d}{dt}\int_{\R^N}P_k(u)dx + \mu\int_{\R^N} T_k(u) (-\Delta)^{\alpha}u dx = \int_{\R^N}T_k(u)fdx,
		\end{equation*}
		where $P_k(z) = \int_0^zT_k(r)d r$ is a primitive function of $T_k$. Applying \cite[Proposition 3]{LPPS2015} and integrating the above identity lead to
		\begin{equation} \label{est:(p=1)-1}
			\begin{aligned} 
				\intR P_k(u(x,t))dx + \mu\int_\tau^t\intR|(-\Delta)^{\frac \alpha 2}T_k(u(x,s))|^2dxds\\
				\le \intR P_k(u_\tau(x))dx + \int_\tau^t\intR T_k(u(x,s))f(x,s)dxds.
			\end{aligned}
		\end{equation}
		We estimate the right-hand side of \eqref{est:(p=1)-1}, using $|P_k(z)| \le k|z|$ and $|T_k(u)|\le k$, as 
		\begin{equation*}
			\intR P_k(u_\tau(x))dx + \int_\tau^t\intR T_k(u(x,s))f(x,s)dxds \le k\bra{\|u_\tau\|_{\LO{1}} + \|f\|_{L^1(Q_{\tau,t})}}.
		\end{equation*}
		By \eqref{est:Lp-Lp} with $p=1$, we have
		$\| u \|_{L^\infty(\tau,T;L^1(\R^N))} \leq \| u_\tau \|_{L^1(\R^N)} + \| f \|_{L^1(Q_{\tau,T})}$,
		which implies 	
		\begin{equation}\label{b2}
			\|T_k(u)\|_{L^{\infty}(\tau,T;\LO{1})} \le \|u_\tau\|_{\LO{1}} + \|f\|_{L^1(Q_{\tau,T})}.
		\end{equation} 
		On the other hand, using the fractional Sobolev inequality in Lemma \ref{fGN-ineq}, we get
		\begin{align*} 
			\int_\tau^t\intR|(-\Delta)^{\frac \alpha 2}T_k(u(x,s))|^2dxds &\geq C_{\mathrm{Sob}}^{-2}\int_\tau^t\|T_k(u(s))\|_{\LO{\frac{2N}{N-2\alpha}}}^2ds\\
			&\ge C_{\textrm{Sob}}^{-2}\|T_k(u)\|_{L^2(\tau,t;\LO{\frac{2N}{N-2\alpha}})}^2,
		\end{align*}
		and consequently,
		\begin{equation}\label{b3} 
			\|T_k(u)\|_{L^2(\tau,T;\LO{\frac{2N}{N-2\alpha}})} \le \mu^{-\frac 12}C_{\mathrm{Sob}}k^{\frac 12}\bra{\|u_\tau\|_{\LO{1}} + \|f\|_{\LQ{1}}}^{\frac 12}.
		\end{equation}
		We now use the interpolation inequality
		\begin{equation}\label{b4} 
			\|T_k(u)\|_{\LQ{\sigma}} \le \|T_k(u)\|_{L^\infty(\tau,T;\LO{1})}^{\theta}\|T_k(u)\|_{L^2(\tau,T;\LO{\frac{2N}{N-2\alpha}})}^{1-\theta}
		\end{equation}
		where $\sigma = \frac{2(N+\alpha)}{N}$ and $\theta = \frac{\alpha}{N+\alpha}$.
		Now, combining \eqref{b2}, \eqref{b3} and \eqref{b4} gives
		\begin{equation*} 
			\|T_k(u)\|_{\LQ{\frac{2(N+\alpha)}{N}}} \le Ck^{\frac{1-\theta}{2}} = C\mu^{-\frac{N}{2(N+\alpha)}} k^{\frac{N}{2(N+\alpha)}}(\| u_\tau \|_{L^1(\R^N)} + \| f \|_{\LQ{1}})^{\frac{N+2\alpha}{2(N+\alpha)}}
		\end{equation*}
		and thus
		\begin{equation}\label{b5}
			\|T_k(u)\|_{\LQ{\frac{2(N+\alpha)}{N}}}^{\frac{2(N+\alpha)}{N}} \leq C\mu^{-1}k (\| u_\tau \|_{L^1(\R^N)} + \| f \|_{\LQ{1}})^{\frac{N+2\alpha}{N}}
		\end{equation}
		where $C$ is independent of $k$. 
		
		Let $M^p(Q_{\tau,T})$, $1<p<\infty$, be the Marcinkiewicz space (or weak $L^p$ space) defined by
		\begin{equation*} 
			M^p(Q_{\tau,T}):= \{v:Q_{\tau,T}\to \R:\, \sup_{\lambda>0}\lambda^p|\{(x,t) \in Q_{\tau,T}: |v(x,t)|\ge \lambda \}|<+\infty \}
		\end{equation*}
		with the norm
		\begin{equation*} 
			\|v\|_{M^p(Q_{\tau,T})}:= \bra{\sup_{\lambda>0}\lambda^p|\{(x,t) \in Q_{\tau,T}: |v(x,t)|\geq \lambda\}|}^{\frac{1}{p}}.
		\end{equation*}
		We can estimate
		\begin{align*}
			\|u\|_{M^{\frac{N+2\alpha}{N}}(Q_{\tau,T})}^{\frac{N+2\alpha}{N}} 
			&= \sup_{k>0}\int_{\tau}^T\intR k^{\frac{N+2\alpha}{N}}\mathbf{1}_{\{|u|\ge k\}}dxdt \le \sup_{k>0}k^{-1}\int_\tau^T\intR T_k(u)^{\frac{2(N+\alpha)}{N}}dxdt\\
			&\leq \sup_{k>0}k^{-1}\cdot C\mu^{-1}k (\| u_\tau \|_{L^1(\R^N)} + \| f \|_{L^1(Q_{\tau,T})})^{\frac{N+2\alpha}{N}} \qquad (\text{thanks to (\ref{b5})})\\
			&\leq C\mu^{-1} (\| u_\tau \|_{L^1(\R^N)} + \| f \|_{L^1(Q_{\tau,T})})^{\frac{N+2\alpha}{N}}.
		\end{align*}
		Then by combining the above estimates and the interpolation inequality that for any $q \in [1,\frac{N+2\alpha}{N})$,
		\begin{align*}
			\| u \|_{L^q(Q_{\tau,T})} \leq C\| u \|_{L^1(Q_{\tau,T})}^{\frac{N+2\alpha-Nq}{2\alpha q}} \| u \|_{M^{\frac{N+2\alpha}{N}} (Q_{\tau,T})}^{\frac{(N+2\alpha)(q-1)}{2\alpha q}},	
		\end{align*}
		we derive \eqref{est:Lp-Lq} for $p=1$.
		
		\medskip
		\noindent\textbf{Case 2:} $1<p<\frac{N+2\alpha}{2\alpha}$. Let $1<\ell \le q$ to be made precise later. 
		Multiplying \eqref{e1_1} by $|u|^{\ell-2}u$ 
		and using the Stroock-Varopoulos inequality (see Lemma \ref{SV-ineq}) for $u$, we obtain, for $t \in (\tau,T)$,
		\begin{equation*}
			\frac{1}{\ell}\frac{d}{dt}\|u(t)\|_{\LO{\ell}}^\ell + \mu\frac{4(\ell-1)}{\ell^2}\norm{(-\Delta)^{\frac{\alpha}{2}}|u(t)|^{\frac \ell 2})}_{\LO{2}}^2 \le \|f(t)\|_{\LO{p}}\|u(t)\|_{\LO{\frac{(\ell-1)p}{p-1}}}^{\ell-1}.
		\end{equation*}
		Integrating the above estimate on $(\tau,T)$ and using H\"older's inequality, we have
		\begin{equation}\label{e2} \begin{split}
				\norm{|u|^{\frac{\ell}{2}}}_{\Lxt{\infty}{2}}^2 +  &\norm{(-\Delta)^{\frac{\alpha}{2}}(|u|^{\frac \ell 2})}_{L^2(Q_{\tau,T})}^2 \le \|u_\tau\|_{\LO{\ell}}^\ell + C\|f\|_{L^p(Q_{\tau,T})}\|u\|_{L^{\frac{(\ell-1)p}{p-1}}(Q_{\tau,T})}^{\ell-1}.
		\end{split} \end{equation}
		By the Gagliardo-Nirenberg inequality (see Lemma \ref{fGN-ineq}) with $v=|u(t)|^{\frac{\ell}{2}} \in H^{\alpha}(\R^N)$, we have 
		\begin{equation*}
			\||u(t)|^{\frac{\ell}{2} }\|_{\LO{\frac{2(N+\alpha)}{N}}} \le C\| |u(t)|^{\frac{\ell}{2} } \|_{\LO{2}}^{\theta}\norm{(-\Delta)^{\frac\alpha 2} (|u(t)|^{\frac{\ell}{2} })}_{\LO{2}}^{1-\theta}
		\end{equation*}
		with $\theta = 1- \frac{N}{2(N+\alpha)} \in (0,1)$. Using this estimate, together with H\"older's inequality and Young's inequality, we have
		\begin{align*}
			\|u\|_{\LQ{\frac{\ell(N+2\alpha)}{N}}}^\ell
			&= \bra{\int_\tau^T\norm{|u(t)|^{\frac \ell 2}}_{\LO{\frac{2(N+2\alpha)}{N}}}^{\frac{2(N+2\alpha)}{N}}dt}^{\frac{N}{N+2\alpha}}\\
			&\leq C\bra{\int_\tau^T\norm{|u(t)|^{\frac \ell 2}}_{\LO{2}}^{\theta\cdot\frac{2(N+2\alpha)}{N}}\norm{(-\Delta)^{\frac\alpha 2}|u(t)|^{\frac \ell 2}}_{\LO{2}}^{(1-\theta)\frac{2(N+2\alpha)}{N}}dt}^{\frac{N}{N+2\alpha}}\\
			&\leq C\norm{|u|^{\frac \ell 2}}_{\Lxt{\infty}{2}}^{2\theta}\bra{\int_\tau^T\norm{(-\Delta)^{\frac\alpha 2}(|u|^{\frac \ell 2})}_{\LO{2}}^{(1-\theta)\frac{2(N+2\alpha)}{N}}dt}^{\frac{N}{N+2\alpha}}\\
			&\le C_{T-\tau}\norm{|u|^{\frac \ell 2}}_{\Lxt{\infty}{2}}^{2\theta}\norm{(-\Delta)^{\frac{\alpha}{2}}(|u|^{\frac \ell 2})}_{L^2(Q_{\tau,T})}^{2(1-\theta)}\\
			&\le C_{T-\tau}\bra{1 + \norm{|u|^{\frac{\ell}{2}}}_{\Lxt{\infty}{2}}^2 +  \norm{(-\Delta)^{\frac{\alpha}{2}}(|u|^{\frac \ell 2})}_{L^2(Q_{\tau,T})}^2}.
		\end{align*}
		Then, from \eqref{e2}, it follows that
		\begin{equation*}
			\|u\|_{L^{\frac{\ell(N+2\alpha)}{N}}(Q_{\tau,T})}^\ell \le C_{T-\tau}\bra{1+\|u_\tau \|_{\LO{\ell}}^\ell + C\|f\|_{L^p(Q_{\tau,T})}\|u\|_{L^{\frac{(\ell-1)p}{p-1}}(Q_{\tau,T})}^{\ell-1}}.
		\end{equation*}
		By choosing $\ell = \frac{Np}{N+2\alpha - 2p\alpha}$
		one can check that $\frac{(\ell-1)p}{p-1} = \frac{\ell(N+2\alpha)}{N}$.
		Thus
		\begin{align*}
			\|u\|_{\LQ{\frac{\ell(N+2\alpha)}{N}}}^\ell &\le C_{T-\tau}\bra{1+\|u_\tau\|_{\LO{\ell}}^\ell + \|f\|_{\LQ{p}}\|u\|_{\LQ{\frac{\ell(N+2\alpha)}{N}}}^{\ell-1}}\\
			&\le  C_{T-\tau}\bra{1+\|u_\tau \|_{\LO{\ell}}^\ell + \|f\|_{\LQ{p}}^\ell + \eps\|u\|_{\LQ{\frac{\ell(N+2\alpha)}{N}}}^\ell}.
		\end{align*}
		By choosing $\eps>0$ small enough, we get the desired estimate
		\begin{equation} \label{Lq-estimate}
			\|u\|_{L^{\frac{(N+2\alpha)p}{N+2\alpha - 2p\alpha}}(Q_{\tau,T})} \le C_{T-\tau}\bra{\|u_\tau \|_{ \LO{\frac{(N+2\alpha)p}{N+2\alpha - 2p\alpha}}} + \|f\|_{L^p(Q_{\tau,T})}}.
		\end{equation}
		From \eqref{Lq-estimate} and \eqref{est:mild-1}, by interpolation, we derive  \eqref{est:Lp-Lq} for any $q \in \left[p,   \frac{(N+2\alpha)p}{N+2\alpha - 2p\alpha} \right]$.

		\medskip
		\noindent \textbf{Case 3:} $p>\frac{N+2\alpha}{2\alpha}$. Let $\{S_{\alpha,\mu}(t)\}_{t\ge \tau}$ be the semigroup generated by the fractional Laplacian $\mu(-\Delta)^{\alpha}$. By applying estimate \eqref{est:Lp-S} to the formulation of solution to \eqref{e1_1} we have
		\begin{align*}
			\|u(t)\|_{\LO{\infty}} &\le \|S_{\alpha,\mu}(t)u_\tau\|_{\LO{\infty}} + \int_\tau^t\|S_{\alpha,\mu}(t-s)f(s)\|_{\LO{\infty}}ds\\
			&\le C_{T-\tau,\mu}\left( \|u_\tau\|_{\LO{\infty}} + \int_\tau^t(t-s)^{-\frac{N}{2\alpha p}}\|f(s)\|_{\LO{p}}ds \right) \\
			&\le C_{T-\tau,\mu}\left(\|u_\tau\|_{\LO{\infty}} + \|f\|_{L^p(Q_{\tau,T})}\bra{\int_\tau^t(t-s)^{-\frac{N}{2\alpha (p-1)}}ds}^{\frac{p-1}{p}} \right). 
		\end{align*}
		For $p>\frac{N+2\alpha}{2\alpha}$ it holds $-\frac{N}{2\alpha(p-1)} > -1$, and consequently the last singular integral is convergent and is bounded by a constant depending on $T-\tau$. Therefore
		\begin{equation*}
			\|u\|_{L^{\infty}(Q_{\tau,T})} \le C_{T-\tau,\mu}(\|u_\tau \|_{L^\infty(\R^N)} + \|f\|_{L^p(Q_{\tau,T})}),
		\end{equation*}
		which is the desired estimate for the case $p>\frac{N+2\alpha}{2\alpha}$. \medskip
		
		\noindent \textbf{Case 4:} $p = \frac{N+2\alpha}{2\alpha}$. First we assume $u_\tau \equiv 0$. Let ${\mathcal S}$ be the solution operator that maps $f$ to the unique solution $u$ to problem \eqref{e1_1}. By case 1 and case 3, we have
		\begin{align*}
			{\mathcal S}: &L^1(Q_{\tau,T}) \to L^1(Q_{\tau,T}), \\
			&L^p(Q_{\tau,T}) \to L^\infty(Q_{\tau,T}), \quad \text{for any } p>\frac{N+2\alpha}{2\alpha},
		\end{align*}
		is continuous.
		For $q \in (\frac{N+2\alpha}{2\alpha},+\infty)$, put $\theta = \frac{1}{q} \in (0,1)$ and $p=\frac{(q-1)(N+2\alpha)}{2\alpha q - N-2\alpha}$. 
By Riez-Thorin interpolation theorem, ${\mathcal S}: L^{\frac{N+2\alpha}{2\alpha}}(Q_{\tau,T}) \to L^q(Q_{\tau,T})$ is continuous and
		\begin{equation} \label{eq:q-supercritical} \| u \|_{L^q(Q_{\tau,T})} \leq C_{T-\tau,\mu}\| f \|_{L^{\frac{N+2\alpha}{2\alpha}}(Q_{\tau,T})}.
		\end{equation}
		For $q=\frac{N+2\alpha}{2\alpha}$, by \eqref{est:mild-1b}, we have 
		\begin{equation} \label{eq:q-critical} \| u \|_{L^{\frac{N+2\alpha}{2\alpha}}(Q_{\tau,T})} \leq C_{T-\tau,\mu}\| f \|_{L^{\frac{N+2\alpha}{2\alpha}}(Q_{\tau,T})}.
		\end{equation}
Next we assume that $f \equiv 0$ and let $u$ be the unique solution to problem \eqref{e1_1}. By \eqref{est:mild-1b},
		\begin{equation} \label{eq:Lq-estimate} \| u \|_{L^q(Q_{\tau,T})} \leq C_{T-\tau,\mu}\| u_\tau \|_{L^q(\R^N)}.
		\end{equation}
In the general case, by linearity, \eqref{eq:q-supercritical}, \eqref{eq:q-critical} and \eqref{eq:Lq-estimate}, we obtain, 
		$$ \| u \|_{L^q(Q_{\tau,T})} \leq C_{T-\tau,\mu}(\| u_\tau \|_{L^q(\R^N)} + \| f \|_{L^{\frac{N+2\alpha}{2\alpha}}(Q_{\tau,T})}),
		$$
		which implies the desired result for $p=\frac{N+2\alpha}{2\alpha}$.
		The proof is complete.
	\end{proof}
	
	Lemma \ref{lem:heat_unsigned} shows the $L^p-L^q$ regularization of the nonlocal heat operator for solutions whose signs can be arbitrary. If the solution is known to be non-negative, as in the case of RDS that we are considering, we can have the same regularization for sub-solutions. The proof is almost the same as that of Lemma \ref{lem:heat_unsigned} except that we utilize the non-negativity of solution to deal with sub-solutions, so we omit the details of the proof of the following lemma.
	\begin{lemma}\label{lem:heat_signed}
		Assume $\mu>0$, $p\in [1,\infty]$, $0 \leq \tau < T$ and $f\in \LQ{p}$. Let $u$ be a smooth {\normalfont non-negative function} satisfying
		\begin{equation*}
			\begin{cases}
				\partial_tu(x,t) + \mu(-\Delta)^\alpha u(x,t) \le f(x,t), &\text{in } Q_{\tau,T},\\
				u(x,\tau) = u_\tau(x), &x\in\R^N.
			\end{cases}
		\end{equation*}
		Then \eqref{est:Lp-Lq} holds.
	\end{lemma}

	\subsection{Duality methods}\label{subsec:duality}
	\begin{lemma}[Improved duality estimate]\label{lem:improved_duality}
		Let $A$ be a smooth function satisfying
		$\underline{a} \le A(x,t) \le \overline{a}$ for all $(x,t)\in Q_{\tau,T}$,
		and let $u$ be a non-negative smooth function such that $u, (-\Delta)^\alpha (Au) \in L^q(Q_{\tau,T})$ for any $q \in (1,\infty)$ and 
		\begin{equation*}
			\begin{cases}
				\pa_t u + (-\Delta)^{\alpha}(Au) \leq \Psi, &\text{in } Q_{\tau,T},\\
				u(x,\tau) = u_\tau(x), &x\in\R^N,
			\end{cases}
		\end{equation*}
		where $\|\Psi\|_{\LQ{1}} \le C_{T-\tau}$, that is the constant depends only on $T-\tau$. Then there exist $\eps_0>0$ and $\delta\in(0,1)$ such that
		\begin{equation*} 
			\|u\|_{\LQ{2+\eps_0}} \le C_{T-\tau}(\|u_\tau\|_{\LO{2+\eps_0}})\bra{1+\|\Psi\|_{\LQ{2+\eps_0}}^{1-\delta}},
		\end{equation*}
		where $C_{T-\tau}(\|u_\tau\|_{\LO{2+\eps_0}})$ depends only $\|u_\tau\|_{\LO{2+\eps_0}}$, $\underline{a}$, $\overline{a}$, and $T-\tau$, with the dependence on $T-\tau$ being at most polynomial.
	\end{lemma}
	\begin{proof}
		Recall from Lemma \ref{Lp-regularity} that $C_{p,N}$ is the smallest constant such that \eqref{e3} holds where $u$ solves \eqref{e1} and $C_{2,N}\le 1$. We will show that there exists $p_*>2$ such that
		\begin{equation} \label{f3} 
			\frac{\overline a - \underline a}{\overline a + \overline a}C_{p_*', N} < 1
		\end{equation}
		where $p_*' = p_*/(p_*-1)$. Let $\eta > 0$ be small and $2_\eta$ be defined as
		\begin{equation*} 
			\frac{1}{2_\eta} = \frac 12\bra{\frac{1}{2} + \frac{1}{2-\eta}} \quad \text{or equivalently}\quad 2_\eta = 2 - \frac{2\eta}{4-\eta}.
		\end{equation*}
		By applying the Riesz-Thorin interpolation theorem (see e.g. \cite[Chapter 2]{lunardi2018interpolation}), we have 
		$C_{2_\eta,N} \le C_{2,N}^{1/2}C_{2-\eta,N}^{1/2} \le C_{2-\eta,N}^{1/2}$. 
		Therefore
		$
			C_{2,N}^-:= \liminf_{\eta \to 0}C_{2_\eta,N} \le \lim_{\eta \to 0}C_{2-\eta,N}^{1/2} = (C_{2,N}^-)^{1/2},
		$
		and thus $C_{2,N}^{-} \le 1$. {Therefore, we can choose $p_*$ sufficiently close to $2$ such that \eqref{f3} holds.}
		
Fix this $p_*$ and let $0\le \theta \in \LQ{p_*'}$ be arbitrary. Consider $\psi$ to be the solution of
		\begin{equation*}
			\partial_t \psi - \omega(-\Delta)^{\alpha}\psi = -\theta \, \text{ in } Q_{\tau,T},\quad 
				\psi(x,T) = 0,  \, x\in \R^N
		\end{equation*}
		with $\omega = (\overline a + \underline a)/2$. By Lemma \ref{Lp-regularity}, $\psi \in L^{p_*'}((\tau,T);H^{2\alpha,p_*'}(\R^N))$, we have the estimate
		\begin{equation*} 
			\|(-\Delta)^\alpha \psi\|_{\LQ{p_*'}} \leq \frac{C_{p_*',N}}{\omega}\|\theta\|_{\LQ{p_*'}},
		\end{equation*}
		hence
		$\|\pa_t\psi\|_{\LQ{p_*'}} \le (C_{p_*',N}+1)\|\theta\|_{\LQ{p_*'}}$. 
		From Lemma \ref{lem:heat_unsigned}, there is some $q>p_*'$ such that $\|\psi\|_{\LQ{q}} \le C_{T-\tau}\|\theta\|_{\LQ{p_*'}}$.
		From this and Minskowski's inequality for integrals, we have
		\begin{align*}
			\|\psi(\tau)\|_{\LO{p_*'}}^{p_*'} = \int_{\R^N}\abs{\int_\tau^T\pa_t\psi(s)ds}^{p_*'}dx
			&\le (T-\tau)^{\frac{1}{p_*-1}}\|\pa_t\psi\|_{\LQ{p_*'}}^{p_*'}\\ &\le (T-\tau)^{\frac{1}{p_*-1}}(C_{p_*',N}+1)^{p_*'}\|\theta\|_{\LQ{p_*'}}^{p_*'}.
		\end{align*}
		By using the regularity $\psi, (-\Delta)^\alpha \psi \in L^{p_*'}(Q_{\tau,T})$ and $u, (-\Delta)^\alpha (Au) \in L^{p_*}(Q_{\tau,T})$ (which allows to take the integration by parts), we have
		\begin{align*}
			&\intQT u\theta dxdt = \intQT u(-\pa_t \psi + \omega(-\Delta)^{\alpha}\psi)dxdt\\
			&= \intR u_\tau {\psi(\tau)}dx + \intQT \psi (\pa_t u + (-\Delta)^{\alpha}(Au))dxdt + \intQT (\omega - A)u(-\Delta)^{\alpha}\psi dxdt\\
			&\le \|u_\tau\|_{\LO{p_*}}\|\psi(\tau)\|_{\LO{p_*'}} + \intQT \psi \Psi dxdt +  \|\omega -A\|_{\LQ{\infty}}\|u\|_{\LQ{p_*}}\|(-\Delta)^{\alpha}\psi\|_{\LQ{p_*'}}\\
			&\le \bra{\|u_\tau\|_{\LO{p_*}}(T-\tau)^{\frac{1}{p_*}}(C_{p_*',N}+1) + \frac{\overline a - \underline a}{\overline a + \underline a}C_{p_*',N}\|u\|_{\LQ{p_*}}}\|\theta\|_{\LQ{p_*'}} + \intQT \psi \Psi dxdt.
		\end{align*}
		From $q>p_*'$ it follows that $q' = \frac{q}{q-1} < p_*$. Therefore, we can estimate for some $\delta\in (0,1)$
		\begin{equation*} 
			\intQT \psi \Psi dxdt \le \|\psi\|_{\LQ{q}}\|\Psi\|_{\LQ{q'}} \le C_{T-\tau}\|\theta\|_{\LQ{p_*'}}\|\Psi\|_{\LQ{p_*}}^{1-\delta}\|\Psi\|_{\LQ{1}}^{\delta}.
		\end{equation*}
		Thanks to the assumption $\|\Psi\|_{\LQ{1}} \le C_{T-\tau}$, we obtain by duality the desired estimate
		\begin{equation*} 
			\|u\|_{\LQ{p_*}} \le \bra{1-\frac{\overline a - \underline a}{\overline a + \overline a}C_{p_*',N}}^{-1}\bra{(C_{p_*',N}+1)(T-\tau)^{\frac{1}{p_*}}\|u_\tau\|_{\LO{p_*}} + C_{T-\tau}\|\Psi\|_{\LQ{p_*}}^{1-\delta}}.
		\end{equation*}
		The proof is complete.
	\end{proof}
	
	\begin{lemma}[Propagation of regularity]\label{lem:propagation}
		Let $u, v$ be non-negative functions such that $u$, $(-\Delta)^\alpha u$,  $v$, $(-\Delta)^\alpha (kv) \in L^r(Q_{\tau,T})$ for any $r \in (1,\infty)$ and 
		\begin{equation*} 
			\begin{cases}
				\pa_t(u+v)+(-\Delta)^{\alpha}(du + k(x,t)v) \le g &\text{in } Q_{\tau,T},\\
				u(x,\tau) = u_\tau(x), v(x,\tau) = v_\tau(x), &x\in\R^N,
			\end{cases}
		\end{equation*}
		where $d>0$ is a constant, the function $k: Q_{\tau,T} \to \R_+$ satisfies $\|k\|_{\LQ{\infty}} \le K$. Then we have, for some $\delta\in (0,1)$,
		\begin{equation*} 
			\|u\|_{\LQ{q}} \le C_{T-\tau}(\|u_\tau+v_\tau\|_{\LO{q}} +\|v\|_{\LQ{q}} + \|g\|_{\LQ{p}})
		\end{equation*}
		provided $p\ge 1$ and $p\le q \le \widehat{p}$ with $\widehat{p}$ defined in \eqref{est:q},
		where the constant $C_{T-\tau}$ depends only on $T-\tau$, $d$, $K$, and $N$.
	\end{lemma}
	\begin{proof}
		For $0\le \theta \in \LQ{q'}$ arbitrary, let $\psi$ be the solution to the equation
		\begin{equation*}
				\pa_t\psi - d(-\Delta)^{\alpha}\psi = -\theta \, \text{ in } Q_{\tau,T},\\
				\psi(x,T) = 0, \, x\in\R^N.
		\end{equation*}
		Similarly to Lemma \ref{lem:improved_duality}
		\begin{equation*} 
			\|(-\Delta)^{\alpha}\psi\|_{\LQ{q'}} \le C_{q',N}d^{-1}\|\theta\|_{\LQ{q'}}, \quad
			\|\pa_t\psi\|_{\LQ{q'}} \le (C_{q',N}+1)\|\theta\|_{\LQ{q'}},
		\end{equation*}
		\begin{equation*}
			\|\psi(\tau)\|_{\LO{q'}}^{q'} \le (T-\tau)^{\frac{1}{q-1}}(C_{q',N}+1)^{q'}\|\theta\|_{\LQ{q'}}^{q'}.
		\end{equation*}
		Using the heat regularization in Lemma \ref{lem:heat_unsigned} to the equation of $\psi$, we have
		$\|\psi\|_{\LQ{\gamma}} \le C_{T-\tau,d}\|\theta\|_{\LQ{q'}}$  \quad for all $\gamma \le \widehat{q'}$,
		where $\widehat{q'}$ defined in \eqref{est:q} with $q'$ in place of $p$. It follows from $p \le q \le \widehat{p}$ that $q' \le p' \le \widehat{q'}$. Therefore, we can choose $\gamma = p'$ and therefore obtain
		$\|\psi\|_{\LQ{p'}} \le C_{T-\tau,d}\|\theta\|_{\LQ{q'}}$.
		Now using the above estimates, together with the assumptions on $u,v$ (which ensures the integration by parts), we can estimate
		\begin{align*}
			\intQT u\theta dxdt &= \intQT u(-\pa_t\psi + d(-\Delta)^\alpha \psi)dxdt\\
			&=\intR u_\tau \psi(\tau)dx + \intQT \psi(\pa_t u + d(-\Delta)^{\alpha}u)dxdt\\
			&\le \intR u_\tau \psi(\tau)dx + \intQT \psi(-\pa_t v -(-\Delta)^{\alpha}(k(x,t)v) + g)dxdt\\
			&= \intR(u_\tau+v_\tau)\psi(\tau)dx + \intQT [v(\pa_t \psi - k(x,t)(-\Delta)^\alpha \psi)+ \psi g] dxdt\\
			&\le \|u_\tau +v_\tau \|_{\LO{q}}\|\psi(\tau)\|_{\LQ{q'}} +  \|\psi\|_{\LQ{p'}}\|g\|_{\LQ{p}}\\
			&\qquad +\|v\|_{\LQ{q}}\bra{\|\pa_t\psi\|_{\LQ{q'}}+ K\|(-\Delta)^\alpha \psi\|_{\LQ{q'}}}\\
			&\le C_{T-\tau}(K,d)\bra{\|u_\tau +v_\tau \|_{\LO{q}} + \|v\|_{\LQ{q}} + \|g\|_{\LQ{p}}}\|\theta\|_{\LQ{q'}}.
		\end{align*}
		By duality, we obtain the desired estimate for $u$ and finish the proof of this lemma.
	\end{proof}
	
	\subsection{Global existence and boundedness}\label{subsec:bootstrap}

	\begin{proof}[\textbf{Proof of Theorem \ref{thm:ISC}}]
		We prove first the global existence. Let $0<T<T_*$ arbitrary. By summing the equations of \eqref{sys} and using the mass dissipation, we obtain
		\begin{equation*} 
			\pa_t\bra{\sum_{i=1}^mu_i} + (-\Delta)^{\alpha}\bra{\sum_{i=1}^md_iu_i} \le 0.
		\end{equation*}
		Denote $Z := \sum_{i=1}^mu_i$ and  $W:= \bra{\sum_{i=1}^{m}d_iu_i}\bra{\sum_{i=1}^mu_i}^{-1}$.
		Then we have
		\begin{equation*} 
			\min_{i=1,\ldots, m}\{d_i\} \le W(x,t) \le \max_{i=1,\ldots, m}\{d_i\}, \quad \forall (x,t)\in Q_T
		\end{equation*}
		and the relation
		\begin{equation*} 
			\pa_t Z + (-\Delta)^{\alpha}(WZ) \le 0.
		\end{equation*}
		From the definition of $Z$ and the regularity of the local solution in Proposition \ref{prop:local-existence}, we deduce that $Z, (-\Delta)^\alpha (WZ) \in L^{q}(Q_T)$ for any $1<q<\infty$. Thanks to Lemma \ref{lem:improved_duality} with $\Psi=0$, there exists some $p_0>2$ such that $\|Z\|_{L^{p_0}(Q_T)} \le C(T,\|Z_0\|_{\LO{p_0}})$,
		where $Z_0 = \sum_{i=1}^mu_{i,0}$. Thanks to the non-negativity of $u_i$, it follows that
		$\|u_i\|_{L^{p_0}(Q_T)} \le C(T,\|Z_0\|_{\LO{p_0}})$ for all $i=1,\ldots, m.$
		Put $f_*:= \Phi+C|\ub|^{{\rho}}$, we have 
		$\|f_*\|_{L^{\frac{p_0}{{\rho}}}(Q_T)}  \leq C(T, \|\Phi\|_{L^{\frac{p_0}{{\rho}}}(Q_T)},\|Z_0\|_{\LO{p_0}})$.
		From the intermediate sum condition \eqref{ISC}, $\pa_t u_1 + d_1(-\Delta)^{\alpha}u_1 = f_1(\ub) \le f_*$, and therefore, thanks to the heat regularization in Lemma \ref{lem:heat_signed}, 
		\begin{align*} 
			\|u_1\|_{\LQT{p_1}} &\leq C_T\bra{\|Z_0\|_{\LO{\frac{p_0}{{\rho}}}\cap \LO{p_1}} + \|f_*\|_{\LQT{\frac{p_0}{{\rho}}}}} \leq C(T, \|\Phi\|_{L^{\frac{p_0}{{\rho}}}(Q_T)},\|Z_0\|_{\LO{1}\cap \LO{p_1}}),
		\end{align*}
		where 
		\begin{equation}\label{p1}
			p_1 = 
			\left\{
			\begin{aligned}
				&\frac{(N+2\alpha)p_0}{{\rho}(N+2\alpha) - 2\alpha p_0}, &&\;\text{ if }\; \frac{p_0}{{\rho}} < \frac{N+2\alpha}{2\alpha},	\\
				&\text{ arbitrary in } \Big[\frac{p_0}{{\rho}}, +\infty\Big) , &&\;\text{ if }\; \frac{p_0}{{\rho}} \geq \frac{N+2\alpha}{2\alpha}.
			\end{aligned}
			\right.
		\end{equation}
		We now show that for $i\in \{1,\ldots, m-1\}$ if
		\begin{equation} \label{g1}
			\|u_j\|_{\LQT{p_0}} \le C(T,\|Z_0\|_{\LO{1}\cap\LO{p_1}}), \quad \forall j=1,2,\ldots, m,
		\end{equation}
		and
		\begin{equation}\label{g2} 
			\|u_j\|_{\LQT{p_1}} \le C(T, \|\Phi\|_{L^{\frac{p_0}{{\rho}}}(Q_T)}, \|Z_0\|_{\LO{1}\cap \LO{p_1}}), \quad \forall j=1,2,\ldots, i, 
		\end{equation}
		then
		\begin{equation*} 
			\|u_{i+1}\|_{\LQT{p_1}} \le  C(T, \|\Phi\|_{L^{\frac{p_0}{{\rho}}}(Q_T)}, \|Z_0\|_{\LO{1}\cap \LO{p_1}}).
		\end{equation*}
		From the intermediate sum condition \eqref{ISC}, by setting
		\begin{equation*} 
			v(x,t) := \sum_{j=1}^{i}a_{ij}u_j(x,t), \quad k(x,t):= \bra{\sum_{j=1}^id_ja_{ij}u_j}\bra{\sum_{j=1}^ia_{ij}u_j}^{-1}
		\end{equation*}
		we obtain $\pa_t(u_{i+1}+v) + (-\Delta)^{\alpha}(u_{i+1} + k(x,t)v) \le f_*$. 
		Note that $\|k\|_{\LQT{\infty}} \le \max_{j=1,\ldots, i}\{d_j\}$, and from \eqref{g1}, \eqref{g2}
		\begin{equation*} 
			\|f_*\|_{\LQT{\frac{p_0}{{\rho}}}} + \|v\|_{\LQT{p_1}} \le C(T, \|\Phi\|_{L^{\frac{p_0}{{\rho}}}(Q_T)},\|Z_0\|_{\LO{1}\cap\LO{p_1}}).
		\end{equation*}
		We note that $u_{i+1},v, (-\Delta)^\alpha u_{i+1}, (-\Delta)^\alpha (kv) \in L^r(Q_T)$ for any $r \in (1,\infty)$ due to the regularity of $u_i$, $i=1,\ldots,m$. Therefore, we can apply Lemma \ref{lem:propagation} with $q = p_1$, $p = p_0/{\rho}$, and {$g = f_*$} to get
		\begin{equation*}
			\begin{aligned} 
				\|u_{i+1}\|_{\LQT{p_1}} &\le C_T\bra{\|u_0+v_0\|_{\LQT{p_1}}+ \|v\|_{\LQT{p_1}} + \|f_*\|_{\LQT{\frac{p_0}{{\rho}}}}}\\
				&\le C(T, \|\Phi\|_{L^{\frac{p_0}{{\rho}}}(Q_T)},\|Z_0\|_{\LO{1}\cap\LO{p_1}})
			\end{aligned}
		\end{equation*}
		which is the desired claim.
		
		\medskip
		By repeating procedure, we obtain a sequence $p_0, p_1, \ldots$ such that 
		\begin{equation}\label{pn}
			p_{n+1} = \frac{(N+2\alpha)p_n}{{\rho}(N+2\alpha)-2\alpha p_n} \; \text{ as long as } \; p_n < \frac{N+2\alpha}{2\alpha {\rho}}
		\end{equation}
		and
		\begin{equation*} 
			\|u_i\|_{\LQT{p_{n}}} \le C(T, \|\Phi\|_{L^{\frac{p_0}{{\rho}}}(Q_T)},\|Z_0\|_{\LO{1}\cap \LO{\infty}}), \quad \forall i=1,\ldots, m.
		\end{equation*}
		We claim that there exists $n_0\ge 2$ such that $p_{n_0} \ge (N+2\alpha)/(2\alpha {\rho})$. Indeed, assume otherwise that $p_{n}<(N+2\alpha)/(2\alpha{\rho})$ for all $n\ge 1$. From the definition of $p_{n}$, we have
		$\frac{p_{n+1}}{p_n} = \frac{N+2\alpha}{{\rho}(N+2\alpha)-2\alpha p_n}$.
		Since ${\rho} \le 1 + 4\alpha/(N+2\alpha)$ and $p_0 > 2$, it holds
		$\frac{N+2\alpha}{{\rho}(N+2\alpha) - 2\alpha p_0} > 1$.
		Thus $\{p_n\}_{n\ge 1}$ is strictly increasing with $\frac{p_{n+1}}{p_n} > \frac{N+2\alpha}{{\rho}(N+2\alpha) - 2\alpha p_0} > 1$.
		This implies that $\lim_{n\to\infty}p_n =+\infty$, which is a contradiction. Now, with $p_{n_0} \ge (N+2\alpha)/(2\alpha{\rho})$, it yields $\|f_*\|_{\LQT{\frac{N+2\alpha}{2\alpha}}} \le C(T, \|\Phi\|_{L^{\frac{p_0}{{\rho}}}(Q_T)},\|Z_0\|_{\LO{1}\cap\LO{\infty}})$. This, in combination with the heat regularization in Lemma \ref{lem:heat_signed}, gives
		\begin{equation*}
			\|u_1\|_{\LQT{s}} \le C(T, \|\Phi\|_{L^{\frac{p_0}{{\rho}}}(Q_T)},\|Z_0\|_{\LO{1}\cap \LO{\infty}}), \quad \forall s\in [1,\infty).
		\end{equation*}
		Using Lemma \ref{lem:propagation}, 
		$
			\|u_j\|_{\LQT{s}} \leq C(T, \|\Phi\|_{L^{\frac{p_0}{{\rho}}}(Q_T)},\|Z_0\|_{\LO{1}\cap \LO{\infty}}), \quad \forall s\in [1,\infty), \quad  \forall j=2,\ldots, m.
		$
		Finally, using the polynomial bounds \eqref{Pol} and the heat regularization in Lemma \ref{lem:heat_signed},
		\begin{equation*} 
			\|u_i\|_{\LQT{\infty}} \le C(T, \|\Phi\|_{L^{\frac{p_0}{{\rho}}}(Q_T)},\|Z_0\|_{\LO{1}\cap\LO{\infty}}), \quad \forall i=1,\ldots,m.
		\end{equation*}
		This completes the proof of Theorem \ref{thm:ISC} for $N\ge 2$.
		
		\medskip
		We turn to show the uniform-in-time bound of solutions. It follows from the assumption of $\Phi$ that $\|\Phi\|_{L^1(Q_{\tau,\tau+2})\cap L^{\infty}(Q_{\tau,\tau+2})} \le C$ for all $\tau \ge 0$. Fix $\tau \ge 0$ arbitrary and let $\varphi_\tau: \R \to [0,1]$ be a smooth function such that $\varphi_\tau|_{(-\infty,\tau]} = 0$, $\varphi_{\tau}|_{[\tau+1,\infty)} = 1$, $0\le \varphi_\tau' \le 2$ for $s\in [\tau,\tau+1]$. Denote by $v_i = \varphi_\tau u_i$, it follows that
		\begin{equation*}
			\begin{cases}
				\partial_t v_i + d_i(-\Delta)^{\alpha}v_i = \varphi_\tau' u_i + \varphi_\tau f_i(x,t,\ub), &\text{in } Q_{\tau,\tau+2},\\
				v_i(x,\tau) = 0, &x\in\R^N.
			\end{cases}
		\end{equation*}
		Summing up the equations, denote by $\tilde Z = \sum_{i=1}^m v_i$, $\tilde W = \bra{\sum_{i=1}^md_iv_i}\bra{\sum_{i=1}^mv_i}^{-1}$ and $\tilde \Psi = \sum_{i=1}^m\varphi_\tau' u_i$, we get
		\begin{equation*} 
			\begin{cases}
				\partial_t\tilde Z + (-\Delta)^{\alpha}(\tilde W \tilde Z) \leq \tilde \Psi, &\text{in } Q_{\tau,\tau+2},\\
				Z(x,0) = 0, &x\in\R^N.
			\end{cases}
		\end{equation*}
By the $\LO{1}$-bound of $u_i$ in \eqref{est:ui_L1}, we obtain
		$\| \tilde \Psi \|_{L^1(Q_{\tau,\tau+2})} = \sum_{i=1}^m d_i \| u_i \|_{L^1(Q_{\tau,\tau+2})} \leq C$,
where $C$ depends only on $d_i$ and $\sum_{i=1}^m \| u_{0,i}\|_{L^1(\R^N)}$. Using the above estimate and the fact that $\min \{d_i\} \le \tilde W \le \max \{d_i\}$, we can apply Lemma \ref{lem:improved_duality} with $T = \tau+2$ and $\Psi = \tilde \Psi$ to get for some $p_0>2$
		\begin{align*} 
			\sum_{i=1}^m\|u_i\|_{L^{p_0}(Q_{\tau+1,\tau+2})}\leq C\|\tilde{Z}\|_{L^{p_0}(Q_{\tau,\tau+2})} \le C\bra{1+\sum_{i=1}^m\|u_i\|_{L^{p_0}(Q_{\tau,\tau+1})}^{1-\delta}}.
		\end{align*}
		Consider $\tau \in \N$ such that ${\sum_{i=1}^m}\|u_i\|_{L^{p_0}(Q_{\tau,\tau+1})} \leq {\sum_{i=1}^m} \|u_i\|_{L^{p_0}(Q_{\tau+1,\tau+2})}$. Then we can use Young's inequality to get
		$\sum_{i=1}^m\|u_i\|_{L^{p_0}(Q_{\tau+1,\tau+2})} \le C$, 
		where \textit{$C$ is independent of $\tau$}. Thus, by Lemma \ref{lem:sequence}, 
		\begin{equation*} 
			{\sum_{i=1}^m} \|u_i\|_{L^{p_0}(Q_{\tau,\tau+1})} \le \max \Big\{C, {\sum_{i=1}^m}\|u_i\|_{L^{p_0}(Q_{1})}\Big\} \quad \forall \tau \geq 0.
		\end{equation*}
		Combining ${\sum_{i=1}^m}\|u_i\|_{L^1(Q_{\tau,\tau+1})} \le C$ for all {$\tau\geq 0$}, it follows that ${\sum_{i=1}^m}\|u_i\|_{L^p(Q_{\tau,\tau+1})} \le C$ for all $1\le p \le p_0$ and all ${ \tau\ge 0}$. From the intermediate sum condition {\eqref{ISC}}, 
		$$\partial_t v_1 + d_1(-\Delta)^{\alpha}v_1 = \varphi_\tau' u_1 + \varphi_\tau f_1(x,t,\ub) \le \varphi_\tau' u_1 + \Phi + C|\ub|^{\mu} =: \varphi_\tau' u_1 + g_*.
		$$ 
		We have $\|u_1\|_{L^{p_0/{\rho}}(Q_{\tau,\tau+2})} + \|g_*\|_{L^{p_0/{\rho}}(Q_{\tau,\tau+2})} \le C$. Applying Lemma \ref{lem:heat_signed} with $T = \tau+2$ yields $\|v_i\|_{L^{p_1}(Q_{\tau,\tau+2})} \le C$ with $p_1$ is defined in \eqref{p1}. By the intermediate sum condition \eqref{ISC}, we have
		\begin{equation*} 
			\partial_t(v_{i+1}+v) + (-\Delta)^{\alpha}(v_{i+1}+k(x,t)v) \le \varphi_\tau'\sum_{j=1}^iu_j + g_*
		\end{equation*}
		with $v = \sum_{j=1}^ia_{ij}v_j$ and $k(x,t) = \bra{\sum_{j=1}^i d_ja_{ij}u_j}\bra{\sum_{j=1}^ia_{ij}u_j}^{-1}$. Now, we can apply Lemma \ref{lem:propagation} with $T = \tau+2$ to obtain
		\begin{equation*} 
			\|v_{i+1}\|_{L^{p_1}(Q_{\tau,\tau+2})} \le C\bra{\|v\|_{L^{p_1}(Q_{\tau,\tau+2})} + \Big\|\sum_{j=1}^i\varphi_\tau' v_j\Big\|_{L^{p_0/{\rho}}(Q_{\tau,\tau+2})} + \|g_*\|_{L^{p_0/{\rho}}(Q_{\tau,\tau+2})}}.
		\end{equation*}
		By induction, similarly to the proof the global existence, we have
		$\|v_i\|_{L^{p_n}(Q_{\tau,\tau+2})} \le C$, for all $\tau \ge 0$, $i=1,\ldots, m$ and
		for all $p_n$ defined as in \eqref{pn}, where \textit{$C$ is a constant independent of $\tau$}. By repeating arguments for the proof of global existence, we obtain finally that $\|v_i\|_{L^{s}(Q_{\tau,\tau+2})} \leq C$ for all $i=1,\ldots, m$, all $s\in [1,\infty)$ and all $\tau \ge 0$, where $C$ is independent of $\tau$. A final application of Lemma \ref{lem:heat_signed} leads to 
		$\|v_i\|_{L^{\infty}(Q_{\tau,\tau+2})} \le C$ for all $i=1,\ldots, m$,
		where \textit{$C$ is independent of $\tau$}. This gives the desired uniform-in-time bound of solutions \eqref{uniform-bound-2}.	
	\end{proof}

	\section{H\"older continuity of a non-divergence fractional diffusion equation}\label{appendix:Holder_v}
	
	We show in this subsection that the function $v$ defined in \eqref{def_v} is H\"older continuous as claimed in Lemma \ref{lem:Holder_v}. {For simplicity, we assume that $\tau = 0$}. Recall that $v$ solves the equation \eqref{h4}
	\begin{equation*} 
		b(x,t)\partial_t v(x,t) + (-\Delta)^{\alpha}v(x,t) = U_0(x):= \sum_{i=1}^mu_{i,0}(x).
	\end{equation*}
	Due to the appearance of $b(x,t)$ in front of the time derivative, as well as the inhomogeneous term $U_0$, the results of \cite{caffarelli2011regularity} are not directly applicable to obtain H\"older continuity of $v$. However, by closely examining the proof in \cite{caffarelli2011regularity}, and using the ideas in \cite{vasseur2016giorgi} to deal with the inhomogeneous term, we find that the arguments therein can be used with some modifications. 
	
	For each $\eps>0$, the rescaled function $\tilde v(y,s) = v(x_0+\eps y, t_0 + \eps^{2\alpha}s)$ solves
	$$
		\tilde b(y,s)\partial_s \tilde v + (-\Delta)^{\alpha}\tilde v(y,s) = \eps^{2\alpha}\tilde U_0(y),
	$$
	where $\tilde b(y,s) = b(y_0 + \eps y, t_0 + \eps^{2\alpha}s)$ and $\tilde U_0(y) = U_0(y_0 + \eps y)$. It is clear that $\tilde b$ satisfies the same bounds \eqref{bound_b} as $b$. Indeed, we have
	\begin{equation*}
		\tilde b(y,s)\partial_s \tilde v(y,s) = \eps^{2\alpha}b(x_0+\eps y,t_0+\eps^{2\alpha}s)(\partial_s v)(x_0+\eps y,t_0+\eps^{2\alpha}s),
	\end{equation*}
	and
	\begin{align*}
		(-\Delta)^{\alpha}\tilde v(y,s) = \eps^{2\alpha}((-\Delta)^{\alpha}v)(x_0+\eps y, t_0+\eps^{2\alpha}s).
	\end{align*}
	Because of this rescaling, we can assume that $\|U_0\|_{\LO{1}\cap \LO{\infty}}$ as small as required.

	\medskip
	In the following, we sketch the main steps of the proof with suitable modifications.
	
	\medskip
	\noindent
	\textbf{Step 1:} Local energy estimates. We define the function $\psi(x) = (|x|^{\alpha}-1)_+$ for $x\in \mathbb R^N$ and the lifted function $\psi_L(x) = L + \psi(x)$ for $L>0$. By multiplying both sides of equation \eqref{h4} with $(v - \psi_L)_+$ we have
	$$
		\intR b(x,t)(\partial_t v) (v-\psi_L)_+dx + \intR (-\Delta)^{\alpha}v (v-\psi_L)_+dx = \intR U_0(v-\psi_L)_+dx.
	$$
	The right hand side is estimated as
	\begin{equation*}
		\intR U_0(v-\psi_L)_+dx \le \|U_0\|_{\LO{\infty}}\intR (v-\psi_L)_+dx.
	\end{equation*}
	By defining the bilinear form
	\begin{equation*}
		B[u,v] = \intR\intR \frac{(u(x)-u(y))(v(x)-v(y))}{|x-y|^{N+2\alpha}}dxdy,
	\end{equation*}
	we compute
	\begin{align*}
				\intR (-\Delta)^{\alpha}v (v-\psi_L)_+dx 
				= \frac 12 B[v, (v-\psi_L)_+].
	\end{align*}
	Finally, using the fact that $\partial_tv = \sum_{i=1}^m u_i \ge 0$, we estimate
	\begin{equation*}
		\intR b(x,t)(\partial_t v) (v-\psi_L)_+dx \ge \underline{b}\intR (\partial_t v)(v-\psi_L)_+dx = \frac{\underline{b}}{2}\frac{d}{dt}\intR(v-\psi_L)_+^2dx.
	\end{equation*}
	Therefore, we have
	\begin{equation*}
		\frac{\underline{b}}{2}\intR (v-\psi_L)_+^2dx + \frac 12 B[v, (v-\psi_L)_+] \le \|U_0\|_{\LO{\infty}}\intR (v-\psi_L)_+dx.
	\end{equation*}
	Now, we can proceed similarly to \cite[Lemma 3.1, \textbf{first step}]{caffarelli2011regularity} to get the energy estimates
	\begin{equation}\label{a2}
		\begin{aligned}
			&\frac{d}{dt}\intR (v-\psi_L)_+^2dx + \frac{1}{\underline{b}}\|(v-\psi_L)_+\|_{H^{\alpha}(\R^N)}^2 + \frac{1}{\underline{b}}B[(v-\psi_L)_-, (v-\psi_L)_+]\\
			&\le C_{N,\underline{b},\alpha}
			\bra{\bra{1+\|U_0\|_{\LO{\infty}}}\intR (v-\psi_L)_+dx + \intR \chi_{\{v-\psi_L>0\}}dx + \intR(v-\psi_L)_+^2dx}.
		\end{aligned}
	\end{equation}
	
	\medskip
	\noindent \textbf{Step 2:} The first De Giorgi's lemma. From the energy estimates \eqref{a2}, one can produce the nonlinear recurrence and obtain what is called the first De Giorgi's lemma.
	\begin{lemma}\label{lem:apx1}
		There is $\eps_0\in (0,1)$ depending only on $N, \alpha, \underline{b}$ and $\|U_0\|_{\LO{\infty}}$ such that the following implication holds true:
		\begin{equation*}
			\int_{-2}^0\intR (v(x,t) - \psi(x))_+^2dxdt \le \eps_0 \quad \Longrightarrow \quad v(x,t) \le \frac 12 + \psi(x) \; \text{ for all } \; (x,t)\in \R^N\times[-1,0].
		\end{equation*}
	\end{lemma}
	As a consequence, one gets
	\begin{corollary}\label{cor:apx1}
		There exists $\delta = \delta(N,\alpha,\underline{b},\|U_0\|_{\LO{\infty}}) \in (0,1)$ such that if
		\begin{equation*}
			v(x,t) \le 1 + (|x|^{\frac{\alpha}{2}}-1)_+ \text{ on } \R^N\times [-2,0] \quad \text{and} \quad
			|\{v > 0 \}\cap \{B_2\times [-2,0]\}| \le \delta,
		\end{equation*}
		then $v(x,t) \le \frac{1}{2}$  for all $(x,t)\in B_1\times [-1,0]$,
		where $B_1$ and $B_2$ are balls in $\R^N$ centered at $0$ with radius $1$ and $2$, respectively.
		\begin{proof}
			The proof follows almost exactly as in \cite[Corollary 3.3]{caffarelli2011regularity}. It should only be remarked that for $(x_0,t_0)\in B_1\times [-1,0]$ the shifted solution $v_R(y,s) = v(x_0 + \frac{y}{R}, t_0 + \frac{s}{R^{2\alpha}})$
			solves 
			\begin{equation}\label{a3}
				b(x_0+\frac{y}{R},t_0+\frac{s}{R^{2\alpha}})\partial_s v_R + (-\Delta)^{\alpha}v_R(y,s) = \frac{1}{R^{2\alpha}}U_0(x_0+\frac{y}{R}),
			\end{equation}
			and since $R\ge 1$, the right hand side of \eqref{a3} is bounded by $\|U_0\|_{\LO{\infty}}$. Because of that Lemma \ref{lem:apx1} is applicable as the constant $\eps_0$ depends only on the bound of the right hand side, which is now smaller than or equal $\|U_0\|_{\LO{\infty}}$.
		\end{proof}
	\end{corollary}

	\medskip
	\noindent \textbf{Step 3:} The second De Giorgi's lemma. Define the function 
	\begin{equation*} 
		F(x) = \sup\{-1, \inf\{0,|x|^2-9\}\}, x\in \R^N
	\end{equation*}
	and the function, for $\lambda < 1/3$, 
	\begin{equation*} 
		\theta_\lambda(x) = \begin{cases}
			0 & \text{ if } |x|\le \lambda^{-2/\alpha},\\
			((|x|-\lambda^{-2/\alpha})^{\alpha/2}-1)_+ & \text{ if } |x| \ge \lambda^{-2/\alpha}.
		\end{cases}
	\end{equation*}
	Using these, we define the three cut-off functions
	\begin{equation*} 
		\varphi_0 = 1 + \theta_\lambda + F, \quad \varphi_1 = 1 + \theta_\lambda + \lambda F, \quad \varphi_2 = 1 + \theta_\lambda + \lambda^2 F.
	\end{equation*}
	We have the following result.
	\begin{lemma}\label{lem:apx2}
		Let $\delta$ be the constant in Corollary \ref{cor:apx1}. Then there are $\mu>0, \gamma > 0$, and $\lambda \in (0,1)$, depending only on $N, \underline{b}, \alpha$ and $\|U_0\|_{\LO{1}\cap \LO{\infty}}$, such that for any solution $v: [-3,0]\times \R^N \to \R$ of \eqref{h4} satisfying
		\begin{gather*}
			v(x,t) \le 1 + \theta_\lambda(x) \; \text{ on } \R^N\times [-3,0], \quad 
			|\{v < \varphi_0\}\cap (B_1\times (-3,-2))| \ge \mu,
		\end{gather*}
		then we have either
		\begin{equation*} 
			|\{v > \varphi_2\}\cap (\R^N\times(-2,0))| \le \delta, \quad\text{ or } \quad  |\{\varphi_0 < w < \varphi_2 \}\cap (\R^N\times(-3,0))| \ge \gamma.
		\end{equation*}
	\end{lemma}
	\begin{proof}
		The proof of this lemma follows exactly from \cite[Lemma 4.1]{caffarelli2011regularity} because it uses only the energy estimates \eqref{a2}.
	\end{proof}

	\medskip
	\noindent \textbf{Step 4:} Proof of H\"older continuity. By defining the function for any $\eps>0$,
	\begin{equation} \label{theta}
		\theta_{\eps,\lambda}(x) = \begin{cases}
			0 & \text{ if } |x| \le \lambda^{-2/\alpha},\\
			((|x|-\lambda^{-2/\alpha})^\eps - 1)_+ & \text{ if } |x| \ge \lambda^{-2/\alpha},
		\end{cases}
	\end{equation}
	with $\lambda$ in Lemma \ref{lem:apx2}, we have the following oscillation lemma.
	\begin{lemma}\label{lem:apx3}
		There exist $\eta>0$ and $\eps>0$ and $\lambda^*$ such that if $\|U_0\|_{\LO{1}\cap 
			\LO{\infty}} \le \eta$, $u$ solves \eqref{h4} in $\R^N\times [-3,0]$ and satisfies
		$-1 - \theta_{\eps,\lambda^*} \le v \le 1 + \theta_{\eps,\lambda^*}$,
		then we have
		\begin{equation*} 
			\sup_{B_1 \times [-1,0]}v - \inf_{B_1\times[-1,0]}v \le 2-\lambda^*.
		\end{equation*}
	\end{lemma}
	\begin{proof}
		This follows again \cite[Proof of Lemma 5.1]{caffarelli2011regularity}, since it only uses the first and second De Giorgi's lemmata, which are obtained previously. Since the problem is not linear anymore, due to the right hand side, we need the smallness assumption $\eta$.
	\end{proof}
	
	Finally, we arrive at the proof of the H\"older continuity of $v$.
	
	\begin{proof}[\textbf{Proof of Lemma \ref{lem:Holder_v}}]
	By the remark before Step 1, we can assume that $\|U_0\|_{\LO{1}\cap \LO{\infty}} \le \eta$ for a sufficiently small constant $\eta>0$.
		For $(t_0,x_0) \in (0,\infty)\times \R^N$, define $K_0 = \min\{1, t_0/4\}^{1/2\alpha}$ and
		$\hat v(x,t) = v(x_0 + K_0x, t_0 + K_0^{2\alpha}t)$.
		We see that $v_0$ solves the equation
		\begin{equation*}
			\hat b(x,t)\partial_t \hat v(x,t) + (-\Delta)^{\alpha}\hat v(x,t) = {K_0^{2\alpha}}\hat U_0(x)
		\end{equation*}
		where $\hat b(x,t) = b(x_0+K_0x, t_0+K_0^{2\alpha}), \hat U_0(x) = U_0(x_0 + K_0x)$. 
		Choose  $L\in (0,1)$ such that
		\begin{equation*}
			\frac{1}{1-\lambda^*/2}\theta_{\eps,\lambda^*}(Lx) \le \theta_{\eps,\lambda^*}, \quad \forall |x| \ge 1/L,
		\end{equation*}
		where $\theta_{\eps,\lambda}$ is defined in \eqref{theta} and $\lambda^*$ is in Lemma \ref{lem:apx3}, and let $k_0$ be a natural number so that
		$\bra{\frac{L^{2\alpha}}{1 - \lambda^*/4}}^k \le \eta^{-1}$ for all $k=1,\ldots, k_0$.
		We define by induction the following functions for $(x,t)\in \R^N\times (-3,0)$,
		\begin{equation*}
			v_1(x,t) = \frac{v_1(x,t)}{\|\hat v\|_{L^{\infty}(\R^N\times (-3,0))}}, \quad b_1(x,t) = \hat b(x, t), \quad  U_1(x) = \frac{\hat U_0(x)}{K_0^{2\alpha}\|\hat v\|_{L^{\infty}(\R^N\times (-3,0))}},
		\end{equation*}
		and for each $k\ge 1$,
		\begin{equation*} 
			v_{k+1}(x,t) = \frac{1}{1-\lambda^*/4}(v_k(Lx,L^{2\alpha}t) - \overline{v}_k), \quad \text{ with } \quad \overline{v}_k  = \frac{1}{|B_1|}\int_{-1}^0\int_{B_1}v_k(x,t)dxdt,
		\end{equation*}
		\begin{equation*} 
			b_{k+1}(x,t) = b_k(Lx, L^{2\alpha}t), \quad U_{k+1}(x) = \frac{L^{2\alpha}}{1-\lambda^*/4}U_k(Lx).
		\end{equation*}
		We can now straightforwardly check that $v_{k}$ solves the equation
		\begin{equation*} 
			b_k(x,t)\partial_t v_k + (-\Delta)^{\alpha}v_k = U_k(x), \quad \text{ for each } k\ge 1.
		\end{equation*}
		Due to the choice of $L$ and $k_0$, we obtain the bound for $U_{k+1}$
		\begin{align*} 
			\|U_{k+1}\|_{L^{\infty}(\R^N\times (-3,0))} &= \frac{L^{2\alpha}}{1-\lambda^*/4}\|U_k\|_{L^{\infty}(\R^N\times (-3,0))}\le\cdots \le \bra{\frac{L^{2\alpha}}{1-\lambda^*/4}}^{k}\|U_1\|_{L^{\infty}(\R^N\times (-3,0))}\\
			&\le \bra{\frac{L^{2\alpha}}{1-\lambda^*/4}}^{k}\frac{\|\hat U_0\|_{\LO{\infty}}}{K_0^{2\alpha}\|\hat v\|_{L^{\infty}(\R^N\times (-3,0))}}\le \frac{1}{K_0^{2\alpha}\|v\|_{L^{\infty}(\R^N\times (t_0/4,t_0))}}
		\end{align*}
		for all $k=1,\ldots, k_0$. The rest now follows exactly as in the proof of Theorem 2.2 in \cite[End of page 865]{caffarelli2011regularity}, which shows that $v$ is $C^{\tilde \gamma}$ with the H\"older exponent given by $\tilde \gamma = \frac{\ln(1-\lambda^*/4)}{\ln(L^{2\alpha})} \in (0,1)$.
		The proof is complete.
	\end{proof}	

	
	\appendix 
	\setcounter{equation}{0}
	\renewcommand{\theequation}{A.\arabic{equation}}
	\section{Appendix}\label{appendix}
	
	
	\subsection{Fractional heat semigroup in the $L^p$ setting}\label{appendix:Lp}
	
	Let $0 \leq \tau < T$ and $1<p,q<\infty$. Denote
		$$
		W_{p,q}^{2\alpha,1}(Q_{\tau,T}):=\{ v \in W^{1,q}((\tau,T);L^p(\R^N)) \text{ such that } (-\Delta)^\alpha v \in L^q((\tau,T);L^p(\R^N)) \}
		$$
		with the norm
		$$ \| v \|_{W_{p,q}^{2\alpha,1}(Q_{\tau,T})}:= \| u \|_{L^q((\tau,T);L^p(\R^N))} + \| \partial_t u \|_{L^q((\tau,T);L^p(\R^N))} + \| (-\Delta)^\alpha u \|_{L^q((\tau,T);L^p(\R^N))}.
		$$
		When $p=q=2$, we simply write $W^{2\alpha,1}_2(Q_{\tau,T})$.
		
		We first recall the Lamberton type estimate \cite[Corrolary 1.1]{lamberton1987}.
		\begin{theorem} \label{thm:LpLq-maximal}
			Let $1<p,q<\infty$, $0 \leq \tau < T$ and $f \in L^q((\tau,T);L^p(\R^N)$. Assume $u \in C([\tau,T];L^p(\R^N))$ is the solution of  the non-homogeneous problem
			\begin{equation} \label{prob:linear-nonhomogeneous}
				\left\{  \begin{aligned}
					\partial_t u + (-\Delta)^\alpha u &= f \quad &&\text{in } Q_{\tau,T}, \\
					u(\cdot,\tau) &= 0\quad &&\text{in } \R^N.  
				\end{aligned} \right.	
			\end{equation}
			Then 
			$ \| u \|_{W_{p,q}^{2\alpha,1}(Q_{\tau,T})} \leq C_{p,q} \| f \|_{L^q((\tau,T);L^p(\R^N))}$.
	\end{theorem}
	For $s \in (0,2)$ and $1 \leq p \leq \infty$, let $\dot{H}^{s,p}(\R^N)$ and $H^{s,p}(\R^N)$ be defined by
	\begin{align} \label{H^alphadot}
		\dot{H}^{s,p}(\R^N)&:= \{ u \in \CS'(\R^N):  \mathcal{F}^{-1}(|\cdot|^s \mathcal{F}(u)) \in L^p(\R^N)   \}, \\ \label{H^alpha}
		H^{s,p}(\R^N)&:= \{ u \in \CS'(\R^N):  \mathcal{F}^{-1}((1+|\cdot|^2)^{\frac{s}{2}}\mathcal{F}(u)) \in L^p(\R^N)   \},
	\end{align}
	with the corresponding norms
	\begin{align*}
		\| u \|_{\dot{H}^{s,p}(\R^N) }&:=\| \mathcal{F}^{-1}(|\cdot|^s\mathcal{F}(u))  \|_{L^p(\R^N)} = \| (-\Delta)^{\frac{s}{2}}u\|_{L^p(\R^N)},  \\
		\| u \|_{H^{s,p}(\R^N) }&:=\| \mathcal{F}^{-1}((1+|\cdot|^2)^{\frac{s}{2}}\mathcal{F}(u))  \|_{L^p(\R^N)} = \| (I-\Delta)^{\frac{s}{2}}u\|_{L^p(\R^N)}. 
	\end{align*}
	When $p=\infty$,  $\dot{H}^{s,\infty}(\R^N)$ and $H^{s,\infty}(\R^N)$ are respectively the homogeneous H\"older space $\dot{C}^{s}(\R^N)$ and the nonhomogeneous H\"older space $C^{s}(\R^N)$. When $p=2$, we simply write $\dot{H}^{s}(\R^N)$ and $H^{s}(\R^N)$ in place of $\dot{H}^{s,2}(\R^N)$ and $H^{s,2}(\R^N)$ respectively.

	\begin{lemma} \label{lem:interchange}
		Let $0 <\beta < \alpha<1$, $0 \leq \tau <T$, $1<p<\infty$ and $f \in L^\infty((\tau,T);L^p(\R^N))$. Then for $\tau<s<t<T$, $S_\alpha(t-s)f(s,\cdot) \in H^{2\beta,p}(\R^N)$ and 
		\begin{equation} \label{est:fracDS(t-s)f-1} 
			\|S_\alpha(t-s)f(s,\cdot)\|_{H^{2\beta,p}(\R^N)} \leq  C[1+(t-s)^{-\frac{\beta}{\alpha}}]\| f(\cdot,s)\|_{L^p(\R^N)},
		\end{equation}
		where $C$ depends only on $N,\alpha,\beta,p$.	Moreover,
		\begin{equation} \label{interchange-1}
			(-\Delta)^{\beta}\int_\tau^t S_\alpha(t-s)f(x,s)ds = \int_\tau^t (-\Delta)^{\beta}[S_\alpha(t-s)f(x,s)]ds \quad \text{for a.e. } x \in \R^N. 	
		\end{equation}	
	\end{lemma}
	\begin{proof}
		First we see that the operator $(-\Delta)^{\beta}: H^{2\beta,p}(\R^N) \to L^p(\R^N)$ is continuous because
		\begin{align*} \| (-\Delta)^{\beta} \varphi \|_{L^p(\R^N)}  = \| \CF^{-1}(|\cdot|^{2\beta} \CF(\varphi)) \|_{L^p(\R^N)} = \| \varphi \|_{\dot{H}^{2\beta,p}(\R^N)} \leq \| \varphi \|_{H^{2\beta,p}(\R^N)} \quad \forall \varphi \in H^{2\beta,p}(\R^N).
		\end{align*}
		Next we observe that for $\tau<s<t<T$, $S_\alpha(t-s)f(\cdot,s) \in H^{\alpha,p}(\R^N)$. Indeed, since $f \in L^\infty((\tau,T);L^p(\R^N))$, for a.e. $s \in (\tau,T)$, $f(s,x) \in L^p(\R^N)$. Therefore, by \eqref{est:Lp-S} and \eqref{est:Lp-DS}, for any $0 \leq \tau<s<t<T$,
		\begin{align*} \nonumber
			\|S_\alpha(t-s)f(s,\cdot)\|_{L^p(\R^N)} &\leq C\| f(\cdot,s)\|_{L^p(\R^N)}, \\
			\|(-\Delta)^{\beta}[S_\alpha(t-s)f(s,\cdot)]\|_{L^p(\R^N)} &\leq C(t-s)^{-\frac{\beta}{\alpha}}\| f(\cdot,s)\|_{L^p(\R^N)},
		\end{align*} 
		which lead to \eqref{est:fracDS(t-s)f-1} . 
		This in turn implies that
		$$
		\| S_\alpha(t-s)f(\cdot,s)\|_{L^1((\tau,T);H^{2\beta,p}(\R^N))} \leq C(t-\tau + (t-\tau)^{\frac{\alpha-\beta}{\alpha}})\| f \|_{L^\infty((\tau,t);L^p(\R^N))}.
		$$
		Thus in view of Hille's theorem for Bochner integrals, we obtain \eqref{interchange-1}.
	\end{proof}
	
	\begin{lemma} \label{lem:interchange-2}
		Assume $0<\beta<\frac{1}{2}$, $0 \leq \tau <T$, $1 \leq p \leq \infty$ and $f \in L^\infty((\tau,T);L^p(\R^N))$. Then for any $\tau<s<t<T$,
		\begin{equation} \label{est:exchange-2}
			(-\Delta)^{\beta} [S_\alpha(t-s)f(\cdot,s)](x) = [((-\Delta)^{\beta} K_\alpha(\cdot,t-s)) * f(\cdot,s)](x) \quad \text{for a. e. } x \in \R^N.	
		\end{equation}	
		Moreover, 
		\begin{equation} \label{est:fracDKalpha-1}
			\| (-\Delta)^{\beta} [S_\alpha(t-s)f(\cdot,s)]\|_{L^\infty(\R^N)} \leq C(t-s)^{-\frac{N}{2\alpha p} - \frac{\beta}{\alpha} } \| f \|_{L^\infty((\tau,T);L^p(\R^N))},
		\end{equation}
		where $C$ depends only on $N,\alpha,\beta,p$.
		
	\end{lemma}
	\begin{proof}
		Since $f \in L^\infty((\tau,T);L^p(\R^N)$, we have $f(\cdot,s) \in L^p(\R^N)$ for a.e. $s \in (\tau,T)$, hence by \eqref{est:Lp-DS}, for any $0 \leq \tau<s<t<T$, $(-\Delta)^{\beta} [S_\alpha(t-s)f(\cdot,s)] \in L^p(\R^N)$. We write
		\begin{equation} \label{est:exchange-3} \begin{split} &(-\Delta)^{\beta} [S_\alpha(t-s)f(\cdot,s)](x) \\
				&= C_{N,\beta} \mathrm{P.V.} \int_{\R^N} \int_{\R^N} \frac{[K_\alpha(x-z,t-s)  - K_\alpha(y-z,t-s)]f(z,s)}{|x-y|^{N+2\beta}}dzdy.
		\end{split} \end{equation}
		By the self similarity property of the heat equation $K_\alpha$ in  \eqref{selfsim}, the pointwise estimate on the gradient of $\tilde K_\alpha$ in \eqref{est:grad-tilde-funda-1}, H\"older's inequality and Minkowski's inequality, we have
		\begin{equation} \label{Ka-4a} \begin{aligned}
				&\int_{\R^N}	|(K_\alpha(x-z,t-s)  - K_\alpha(y-z,t-s))f(z,s)|dz  \\
				&\leq C(t-s)^{-\frac{N}{2\alpha p}-\frac{1}{2\alpha} } |x-y| \| f \|_{L^\infty((\tau,T);L^p(\R^N))}.
			\end{aligned}
		\end{equation}
		On the other hand, by using \eqref{est:tilde-funda-1}, we have
		\begin{equation} \label{Ka-5} \begin{split}
				\int_{\R^N} |(K_\alpha(x-z,t-s)  - K_\alpha(y-z,t-s))f(z,s)|dz  \leq C(t-s)^{-\frac{N}{2\alpha p}} \| f \|_{L^\infty((\tau,T);L^p(\R^N))}.
		\end{split} \end{equation}

		Let $\lambda>0$. Combining \eqref{Ka-4a} and \eqref{Ka-5} leads to
		\begin{equation*} \begin{aligned}
				&\int_{\R^N} \int_{\R^N} \frac{|(K_\alpha(x-z,t-s)  - K_\alpha(y-z,t-s))f(z,s)|}{|x-y|^{N+2\beta}}dzdy  \\
				&\leq C(t-s)^{-\frac{N}{2\alpha p}-\frac{1}{2\alpha} } \| f \|_{L^\infty((\tau,T);L^p(\R^N))} \int_{B(x,\lambda)}|x-y|^{-N-2\beta+1}dy \\
				&+ C(t-s)^{-\frac{N}{2\alpha p}} \| f \|_{L^\infty((\tau,T);L^p(\R^N))} \int_{\R^N \setminus B(x,\lambda)} |x-y|^{-N-2\beta}dy \\
				&\leq C(t-s)^{-\frac{N}{2\alpha p} }[ (t-s)^{-\frac{1}{2\alpha}}\lambda^{1-2\beta}+\lambda^{-2\beta}] \| f \|_{L^\infty((\tau,T);L^p(\R^N))}.
		\end{aligned} \end{equation*}
		Minimizing over $\lambda>0$, we obtain
		\begin{align} \nonumber
			&\int_{\R^N} \int_{\R^N} \frac{|(K_\alpha(x-z,t-s)  - K_\alpha(y-z,t-s))f(z,s)|}{|x-y|^{N+2\beta}}dzdy \\ \label{Ka-7}
			&\leq C(t-s)^{-\frac{N}{2\alpha p} - \frac{\beta}{\alpha} } \| f \|_{L^\infty((\tau,T);L^p(\R^N))}.	
		\end{align}


			Therefore, we can remove the principal value in \eqref{est:exchange-3} and hence
			\begin{equation} \label{removePV-1} \begin{aligned}   &(-\Delta)^{\frac{\alpha}{2}} (S_\alpha(t-s)f(\cdot,s))(x) \\
					&= C_{N,\alpha}  \int_{\R^N} \int_{\R^N} \frac{[K_\alpha(x-z,t-s)  - K_\alpha(y-z,t-s)]f(z,s)}{|x-y|^{N+2\beta}}dzdy.
			\end{aligned} \end{equation}
			Next for any $0 \leq \tau<s<t<T$, $x \in \R^N$, by estimate \eqref{est:tilde-funda-1}, estimate \eqref{est:grad-tilde-funda-1}, we derive
			\begin{equation*}
				\int_{\R^N} \frac{|K_\alpha(x,t-s) - K_\alpha(y,t-s)|}{|x-y|^{N+2\beta}}dy \leq C(t-s)^{-\frac{N}{2\alpha p} - \frac{\beta}{\alpha} }.	
			\end{equation*}
			Hence one can remove the principal value in the definition of $(-\Delta)^{\frac{\alpha}{2}}K_\alpha(x,t-s)$, namely
			\begin{equation} \label{removePV-2}
				(-\Delta)^{\beta}K_\alpha(x,t-s) = C_{N,\alpha} \int_{\R^N} \frac{K_\alpha(x,t-s) - K_\alpha(y,t-s)}{|x-y|^{N+2\beta}}dy. 	
			\end{equation}
			Combining \eqref{removePV-1} and \eqref{removePV-2} yields \eqref{est:exchange-2}. Finally, \eqref{est:fracDKalpha-1} follows from \eqref{Ka-7}. 
		\end{proof}
		
		\begin{remark} \label{est:Holder-u} Assume $p \in [1,+\infty]$. From \eqref{Ka-4a}, \eqref{Ka-5}, we can derive that, for $0<\beta<1$,
			$$ \left\|  S_{i,\alpha}(t-\tau)\varphi \right \|_{C^\beta(\R^N)} \leq C(t-\tau)^{-\frac{N}{2\alpha p}- \frac{\beta}{2\alpha}}\| \varphi \|_{L^p(\R^N)}, \quad \forall t \in (\tau,T),
			$$	
			and, for $0<\beta<\min\{1,2\alpha- \frac{N}{p}\}$, 
			$$ \left\| \int_\tau^t S_{i,\alpha}(t-s)f(s)ds \right \|_{C^\beta(\R^N)} \leq C(t-\tau)^{1-\frac{N}{2\alpha p}- \frac{\beta}{2\alpha}}\| f \|_{L^\infty((\tau,T);L^p(\R^N))}, \quad \forall t \in (\tau,T].
			$$
		\end{remark}
		
		\subsection{Fractional heat semigroup in the H\"older setting}\label{appendix:Holder}

		\begin{lemma} \label{lem:interchange-3}
			Let $0<\beta < \frac{\gamma}{2} < \alpha < 1$, $0<\gamma<1$, $0 \leq \tau < T$ and $f \in L^\infty(Q_{\tau,T})$. 
			
			(i) For $\tau<s<t<T$, there hold 
			\begin{equation} \label{est:f-infty-Linfty}
				\|(-\Delta)^\beta S_\alpha(t-s)f(\cdot,s)\|_{L^\infty(\R^N)}
				\leq C\| f \|_{L^\infty(Q_{\tau,T})} (t-s)^{-\frac{\beta}{\alpha}}, 
			\end{equation}
			where $C$ depends only on $N,\alpha,\beta$, and 
			\begin{equation} \label{est:f-infty-Holder}
				\| (-\Delta)^\beta [S_\alpha(t-s)f(\cdot,s)] \|_{\dot{C}^{\gamma-2\beta}(\R^N)} \leq C\| f \|_{L^\infty(Q_{\tau,T})}(t-s)^{-\frac{\gamma}{2\alpha}},
			\end{equation}
			where $C$ depends only on $N,\alpha,\beta,\gamma$, $\dot{C}^{\gamma-2\beta}$ denotes the homogeneous H\"older space.
			
			(ii) For any $\tau<t<T$, \eqref{interchange-1} holds and
			\begin{equation} \label{est:f-infty-Linfty-Holder}
				\left\| (-\Delta)^\beta \int_\tau^t S_\alpha(t-s)f(\cdot,s)ds \right\|_{L^\infty((\tau,T);C^{\gamma-2\beta}(\R^N))} \leq C\| f \|_{L^\infty(Q_T)}((T-\tau)^{\frac{\alpha-\beta}{\alpha}} + (T-\tau)^{\frac{2\alpha-\gamma}{2\alpha}}),
			\end{equation}
			where $C$ depends only on $N,\alpha,\beta,\gamma$.
		\end{lemma}
		\begin{proof}
			(i) It is easy to see that \eqref{est:f-infty-Linfty} follows from \eqref{est:Lp-DS}. Next we will prove \eqref{est:f-infty-Holder}. For $\tau<s<t<T$ and $x \neq y$, we have
			\begin{align} \nonumber
				&|S_\alpha(t-s)f(x,s)-S_\alpha(t-s)f(y,s)| \\  \label{est:Holder-00}
				&\leq \| f \|_{L^\infty(Q_{\tau,T})}\int_{\R^N} |K_\alpha(x-z,t-s)-K(y-z,t-s)|dz.
			\end{align}


			Let $\lambda>0$. By a similar argument as above, we can show that
			\begin{align} \nonumber
				\frac{1}{|x-y|^{\gamma}}	\int_{\R^N}	|(K_\alpha(x-z,t-s)  - K_\alpha(y-z,t-s))|dz  \nonumber
				\leq C(t-s)^{-\frac{1}{2\alpha}}\lambda^{1-\gamma} + C\lambda^{-\gamma}.
			\end{align}	
			By minimizing the right hand side over $\lambda>0$, we obtain
			\begin{align*} 
				\frac{1}{|x-y|^{\gamma}}	\int_{\R^N}	|K_\alpha(x-z,t-s)  - K_\alpha(y-z,t-s)|dz 
				\leq C(t-s)^{-\frac{\gamma}{2\alpha}}.
			\end{align*}	
			Plugging the above estimate into \eqref{est:Holder-00}, we derive
			$$\| S_\alpha(t-s)f(\cdot,s) \|_{\dot{C}^{\gamma}(\R^N)} \leq C\| f \|_{L^\infty(Q_{\tau,T})}(t-s)^{-\frac{\gamma}{2\alpha}}.$$ 
			This and the fact that $(-\Delta)^\beta: \dot{C}^{\gamma}(\R^N) \to \dot{C}^{\gamma-2\beta}(\R^N)$ is continuous yield \eqref{est:f-infty-Holder}. \medskip

			(ii) Adding estimates \eqref{est:f-infty-Linfty} and \eqref{est:f-infty-Holder} leads to
			\begin{equation*} 
				\| (-\Delta)^\beta S_\alpha(t-s)f(\cdot,s) \|_{C^{\gamma-2\beta}(\R^N)} \leq C\| f \|_{L^\infty(Q_{\tau,T})}[(t-s)^{-\frac{\gamma}{2\alpha}} + (t-s)^{-\frac{\beta}{\alpha}}].
			\end{equation*}	
			Since $\beta<\alpha$ and $\gamma <2\alpha$, it follows that, for $t \in (\tau,T)$,
			\begin{equation} \label{est:fracDS-Holder-3} 
				\| (-\Delta)^\beta S_\alpha(t-s)f(\cdot,s)\|_{L^1((\tau,t);C^{\gamma-2\beta}(\R^N))} \leq C\| f \|_{L^\infty(Q_{\tau,T})}((t-\tau)^{\frac{2\alpha-\gamma}{2\alpha}} + (t-\tau)^{\frac{\alpha-\beta}{\alpha}}).
			\end{equation}
			Thus in view of Hille's theorem for Bochner integrals, we obtain \eqref{interchange-1}. Estimate \eqref{est:f-infty-Linfty-Holder} follows from \eqref{interchange-1} and \eqref{est:fracDS-Holder-3}. 
		\end{proof}

		\begin{lemma} \label{lem:HolderHolder}
			Let $0 \leq \tau < T$ and $f \in L^\infty((\tau,T);C^{\gamma}(\R^N))$ for some $0<\gamma<1$. 
			
			(i) Assume $0<2\beta<\gamma<1$. Then for any $\tau<s<t<T$,  
			\begin{equation} \label{est:fracDSalpha-Holder}
				\| (-\Delta)^\beta [S_\alpha(t-s)f(\cdot,s)] \|_{C^{\gamma-2\beta}(\R^N)} \leq C\| f \|_{L^\infty((\tau,T);C^{\gamma}(\R^N))},
			\end{equation}
			where $C$ depends only on $N,\alpha,\beta,\gamma$.
			
			(ii) For any $\tau<s<t<T$, \eqref{interchange-1} holds and
			\begin{equation} \label{est:f-Holder-Linfty-Holder}
				\left\| (-\Delta)^\beta \int_\tau^t S_\alpha(t-s)f(\cdot,s)ds \right\|_{L^\infty((\tau,T);C^{\gamma-2\beta}(\R^N))} \leq C_{T-\tau}\| f \|_{L^\infty((\tau,T);C^{\gamma}(\R^N))},
			\end{equation}
			where $C$ depends only on $N,\alpha,\beta,\gamma$.
		\end{lemma}
		
		\begin{proof} 
		(i)	First we estimate $L^\infty$-norm of $(-\Delta)^\beta [S_\alpha(t-s)f(\cdot,s)]$. We write
			\begin{align} \nonumber 
				&| (-\Delta)^{\beta}[S_\alpha(t-s)f(\cdot,s)](x)| \\ \label{est:fracDeltaSu0-1}
				&= C_{N,\alpha} \left|\mathrm{P.V.} \int_{\R^N}  \int_{\R^N} \frac{K_\alpha(z,t-s)(f(x-z,s)-f(y-z,s))}{|x-y|^{N+2\beta}} dz dy   \right|.
			\end{align}
			By using the assumption $f \in L^\infty((\tau,T);C^{\gamma}(\R^N))$, we have
			\begin{align*} 
				\int_{\R^N} K_\alpha(z,t-s)|f(x-z,s)-f(y-z,s)| dz  \leq C\| f \|_{L^\infty((\tau,T);\dot{C}^{\gamma}(\R^N))}  |x-y|^{\gamma}.
			\end{align*}
			We also obtain
			$\int_{\R^N} K_\alpha(z,t-s)|f(x-z,s)-f(y-z,s)|dz  \leq C\| f \|_{L^\infty(Q_{\tau,T})}$.
			By using a standard minimization argument as before, we derive
			\begin{align*} 
				\int_{\R^N}  \int_{\R^N} \frac{K_\alpha(z,t-s)|f(x-z,s)-f(y-z,s)|}{|x-y|^{N+2\beta}} dz dy \leq 	C \| f \|_{L^\infty(Q_{\tau,T})}^{\frac{\gamma-2\beta}{\gamma}} \| f \|_{L^\infty((\tau,T);\dot{C}^{\gamma}(\R^N))}^{\frac{2\beta}{\gamma}}.
			\end{align*}
			This and \eqref{est:fracDeltaSu0-1} imply
			\begin{equation} \label{est:fracDeltaSu0-2}
				\| (-\Delta)^{\beta}[S_\alpha(t-s)f(s)] \|_{L^\infty(\R^N)}  \leq C \| f \|_{L^\infty(Q_{\tau,T})}^{\frac{\gamma-2\beta}{\gamma}} \| f \|_{L^\infty((\tau,T);\dot{C}^{\gamma}(\R^N))}^{\frac{2\beta}{\gamma}},
			\end{equation}
			where $C$ depends only on $N,\alpha,\beta,\gamma$. Next, for any $\tau<s<t<T$ and $x \neq y$, we have
			\begin{align*} 
				|S_\alpha(t-s)f(x,s)-S_\alpha(t-s)f(y,s)|
				\leq C\| f \|_{L^\infty((\tau,T);\dot{C}^{\gamma}(\R^N))}|x-y|^{\gamma}. 
			\end{align*}
			Consequently, we obtain
			\begin{equation} \label{est:fracDSalpha-Holder-a}
				\| (-\Delta)^\beta [S_\alpha(t-s)f(\cdot,s)] \|_{\dot{C}^{\gamma-2\beta}(\R^N)} \leq C\| f \|_{L^\infty((\tau,T);\dot{C}^{\gamma}(\R^N))}.
			\end{equation}
			Combining \eqref{est:fracDSalpha-Holder-a} and \eqref{est:fracDeltaSu0-2} yields \eqref{est:fracDSalpha-Holder}.
			
			(ii) It follows from \eqref{est:fracDSalpha-Holder} that 
			\begin{equation} \label{est:f-Holder-Linfty-Holder-b}
				\| (-\Delta)^\beta S_\alpha(t-s)f(\cdot,s) \|_{L^1((\tau,t);C^{\gamma-2\beta}(\R^N))} \leq C_{T-\tau} \| f \|_{L^\infty((\tau,T);C^{\gamma}(\R^N))}.
			\end{equation}
			This allows to apply Hille's theorem for Bochner integrals to derive \eqref{interchange-1}. Estimate \eqref{est:f-Holder-Linfty-Holder} follows from \eqref{interchange-1} and \eqref{est:f-Holder-Linfty-Holder-b}.
		\end{proof}

		By using a similar argument as in the proof of Lemma \ref{lem:HolderHolder}, we can establish the following result:
		\begin{lemma} \label{lem:fracDeltaSphi-Holder}
			Assume $0<2\beta<\theta<1$,  $\varphi \in C^{\theta}(\R^N)$ and $0 \leq \tau <T$. Then for any $\tau<t<T$, 
			\begin{equation*} 
				\| (-\Delta)^{\beta}S_\alpha(t-\tau)\varphi \|_{C^{\theta-2\beta}(\R^N)} \leq C\| \varphi \|_{C^{\theta}(\R^N)},	
			\end{equation*}
		where $C$ depends only on $N,\alpha,\beta,\theta$.
		\end{lemma}

		\subsection*{Acknowledgements}
		The authors are supported by the OeAD Project No. CZ 03/2023 and the Mobility Project 8J23AT032. The work has been completed during the visits of B. T. to Masaryk University, and P.-T. N. to University of Graz, and the universities' hospitality is greatly acknowledged. B.T. is also supported by the FWF project number I-5213 (Quasi-steady-state approximation for PDE).

	\end{document}